\documentclass[a4paper]{amsart}
\usepackage{amscd,amssymb,amsmath,mathrsfs,latexsym,url}
\usepackage[all]{xy}
\usepackage{color,graphicx}
\usepackage{hyperref}
\hypersetup{breaklinks=true}

\usepackage[textwidth=27mm, textsize=tiny]{todonotes}

\def\ov#1{{\overline{#1}}}

\def\wh#1{{\widehat{#1}}}
\def\wt#1{{\widetilde{#1}}}
\def\?{\ ???\ \immediate\write16{}%
\immediate\write16{Warning: There was still a question mark . . . }%
\immediate\write16{}}

\newcommand{\conv}{\operatorname{conv}}

\newcommand{\rec}{\operatorname{rec}}
\newcommand{\vol}{\operatorname{vol}}

\newcommand{\dist}{\operatorname{dist}}

\newcommand{\tor}{\operatorname{tor}}
\newcommand{\htor}{\h^{\tor}}

\newcommand{\Car}{\operatorname{Car}}
\newcommand{\Rat}{\operatorname{Rat}}
\newcommand{\conc}{\operatorname{conc}}

\newcommand{\avol}{\operatorname{\wh{vol}}}

\newcommand{\MI}{\operatorname{MI}}

\newcommand{\VSF}{\operatorname{VSF}}
\newcommand{\SF}{\operatorname{SF}}

\newcommand{\pur}{\text{\rm pur}}

\newcommand{\dd}{\hspace{1pt}\operatorname{d}\hspace{-1pt}}

\renewcommand{\div}{\operatorname{div}}
\newcommand{\LL}{\operatorname{L}}
\newcommand{\lL}{\operatorname{\it l}}
\newcommand{\Div}{\operatorname{Div}}

\newcommand{\whL}{\operatorname{\wh {\LL}}}
\newcommand{\whl}{\operatorname{\wh {\it l}}}

\newcommand{\Spec}{\operatorname{Spec}}
\newcommand{\mult}{\operatorname{mult}}

\newcommand{\stab}{\operatorname{stab}}

\newcommand{\chern}{\operatorname{c}}

\renewcommand{\i}{\operatorname{i}}
\newcommand{\ri}{\operatorname{ri}}
\newcommand{\DV}{\operatorname{DV}}

\newcommand{\e}{\operatorname{e}}
\newcommand{\h}{\operatorname{h}}

\newcommand{\Hom}{\operatorname{Hom}}

\newcommand{\can}{{\operatorname{can}}}
\newcommand{\ord}{{\operatorname{ord}}}

\newcommand{\val}{{\operatorname{val}}}

\newcommand{\an}{{\operatorname{an}}}

\def \A{\mathbb{A}}

\def \C{\mathbb{C}}
\def \F{\mathbb{F}}
\def \G{\mathbb{G}}
\def \K{\mathbb{K}}

\def \N{\mathbb{N}}
\def \P{\mathbb{P}}
\def \Q{\mathbb{Q}}
\def \R{\mathbb{R}}
\def \SS{\mathbb{S}}

\def \T{\mathbb{T}}
\def \Z{\mathbb{Z}}

\def\cC {{\mathcal C}}
\def\cD {{\mathcal D}}
\def\cE {{\mathcal E}}

\def\cK {{\mathcal K}}

\def\cO {{\mathcal O}}
\def\cP {{\mathcal P}}

\def\cX {{\mathcal X}}
\def\cY {{\mathcal Y}}

\newcommand{\bfalpha}{{\boldsymbol{\alpha}}}

\numberwithin{equation}{section}
\theoremstyle{definition}
\newtheorem{defn}[equation]{Definition}

\newtheorem{rem}[equation]{Remark}
\newtheorem{exmpl}[equation]{Example}

\theoremstyle{plain}
\newtheorem{lem}[equation]{Lemma}

\newtheorem{prop}[equation]{Proposition}
\newtheorem{thm}[equation]{Theorem}
\newtheorem{cor}[equation]{Corollary}
\newtheorem{prop-def}[equation]{Proposition-Definition}
\newtheorem{ques}[equation]{Question}
\newtheorem{introthm}{Theorem}

\begin{document}

\title[Arithmetic positivity on toric varieties]{Arithmetic positivity on toric varieties}

\author[Burgos Gil]{Jos\'e Ignacio Burgos Gil}
\address{Instituto de Ciencias Matem\'aticas (CSIC-UAM-UCM-UCM3).
  Calle Nicol\'as Ca\-bre\-ra~15, Campus UAB, Cantoblanco, 28049 Madrid,
  Spain} 
\email{burgos@icmat.es}
\urladdr{\url{http://www.icmat.es/miembros/burgos}}
\author[Moriwaki]{Atsushi Moriwaki}
\address{Department of Mathematics,
Faculty of Science, Kyoto University.
606-8502 Kyoto, Japan
}
\email{moriwaki@math.kyoto-u.ac.jp}
\author[Philippon]{Patrice Philippon}
\address{Institut de Math{\'e}matiques de
Jussieu -- U.M.R. 7586 du CNRS, \'Equipe de Th\'eorie des Nombres.
Case 247, 4 place Jussieu, 75252 Paris cedex 05, France}
\email{patrice.philippon@imj-prg.fr}
\urladdr{\url{http://www.math.jussieu.fr/~pph}}
\author[Sombra]{Mart{\'\i}n~Sombra}
\address{ICREA \& Universitat de Barcelona, Departament d'{\`A}lgebra i Geometria.
Gran Via 585, 08007 Bar\-ce\-lo\-na, Spain}
\email{sombra@ub.edu}
\urladdr{\url{http://atlas.mat.ub.es/personals/sombra}}

\thanks{Burgos Gil was  partially supported by the MICINN
  research projects MTM2009-14163-C02-01 and MTM2010-17389.
  Moriwaki was partially 
  supported by the Ministry of Education, Science, Sports and Culture,
  Grant-in-Aid for Scientific Research (A) 22244003, 2011.  Philippon
  was partially supported by the CNRS international projects for
  scientific cooperation (PICS) ``Properties of the heights of
  arithmetic varieties'' and ``G\'
  eom\' etrie diophantienne et calcul formel'' and the ANR research project ``Hauteur,
  modularit\'e, transcendance''. 
  Sombra was partially
  supported by the MICINN
  research project MTM2009-14163-C02-01 and the MINECO research project
  MTM2012-38122-C03-02.}

\date{\today} 

\begin{abstract}
  We continue the study of the arithmetic geometry of toric varieties
  started in \cite{BurgosPhilipponSombra:agtvmmh}.  In this text, we
  study the positivity properties of metrized $\R$-divisors in the
  toric setting.  For a toric metrized $\R$-divisor, we give formulae
  for its arithmetic volume and its $\chi$-arithmetic volume, and we
  characterize when it is arithmetically ample, nef, big or
  pseudo-effective, in terms of combinatorial data.  As an
  application, we prove a higher-dimensional analogue of  Dirichlet's
  unit theorem for toric varieties, 
  we give a characterization for the existence of a Zariski
  decomposition of a toric metrized $\R$-divisor, and we prove a toric
  arithmetic Fujita approximation theorem.

\end{abstract}
\maketitle

\overfullrule=0.3mm
\vspace{-8mm}

\tableofcontents

\section*{Introduction}
\label{sec:introduction}

The study of the positivity properties of a divisor on an algebraic
variety is a central subject in algebraic geometry which has many
important results and applications. A modern account about this
subject can be found in the book~\cite{Lazarsfeld:posit_I}.

There are different notions of positivity for a divisor: it can be
ample, nef, big, or pseudo-effective. There are also numerical
invariants of a divisor related with positivity, like its degree and
its volume. The degree of a divisor is the top intersection product of
the divisor with itself, while the volume measures the asymptotic
growth of the space of global sections of the multiples of the
divisor. When the divisor is ample both invariants agree but, in
general, the volume is always nonnegative while the degree can be
  either positive, zero or negative. The different notions of
positivity, the degree and the volume of a divisor are invariant under
numerical equivalence and can be extended to $\R$-divisors.

Analogues of these invariants and notions of positivity have been
introduced in Arakelov geometry and their study has interesting
applications to Diophantine geometry.  In \cite{Zhang:plbas}, Zhang
started the study of a theory of arithmetic ampleness and proved an
arithmetic Nakai-Moishezon criterion. Using this theory, he obtained
the so-called ``theorem on successive algebraic minima'' relating the minimal
height of points in the variety which are Zariski dense with the
height of the variety itself. This result plays an important role in
Diophantine geometry, for example in the context of the Manin-Mumford conjecture, the Bomogolov and Lehmer questions, and the Zilber-Pink conjecture.

In \cite{Moriwaki:cav}, Moriwaki introduced the notion of arithmetic
volume measuring the growth of the number of small sections of the
multiples of an arithmetic divisor, and proved the continuity of
this invariant.  In \cite{Yuan:blbav}, Yuan studied the basic
properties of big arithmetic divisors, that is, arithmetic divisors
with strictly positive arithmetic volumes. As an application, he
obtained a very general criterion for the equidistribution of points
of small height in an arithmetic variety, generalizing the previous
equidistribution theorems of Szpiro--Ullmo--Zhang
\cite{SzpiroUllmoZhang:epp}, Bilu \cite{Bilu:ldspat},
Favre--Rivera-Letelier \cite{FavreRivera:eqpph}, Baker--Rumely
\cite{BakerRumely:esp} and Chambert-Loir \cite{ChambertLoir:meeB}.

Given these results, it is interesting to dispose of effective
criteria to test the positivity properties of an arithmetic divisor
and to be able to calculate the associated invariants in concrete
situations.  In this direction, Moriwaki has studied a family of
twisted Fubini-Study metrics on the hyperplane divisor of
$\P^{n}_{\Z}$. He has obtained criteria for when these metrics define an
ample, nef, big or pseudo-effective arithmetic divisor, he also computed the
arithmetic volumes of such divisors, proved a Fujita approximation
theorem and gave a criteria for when a special type of Zariski
decomposition exists \cite{Moriwaki:bad}.  The present text
generalizes these results to arbitrary toric (adelically) metrized
$\R$-divisors on toric varieties.

Toric varieties can be described in combinatorial terms and many of
their algebro-geometric properties can be translated in terms of this
description.  A proper toric variety~$X$ of dimension $n$ over an arbitrary
field is given by a complete fan $\Sigma$ on a vector space
$N_{\R}\simeq \R^{n}$.
A toric $\R$-divisor~$D$ on~$X$ defines a function $\Psi_{D}\colon
N_{\R}\to \R$ which is linear in each cone of the fan $\Sigma$.
Following the usual terminology in toric geometry, we call such
function a ``virtual support function''. One can also associate to $D$
a polytope $\Delta_{D}$ in the dual space $M_{\R}:=N_{\R}^{\vee}$.
There is a ``toric dictionary'' that translates algebro-geometric
properties of the pair $(X,D)$ into combinatorial properties of the
fan, the virtual support function and the polytope.  For instance, the
set of points of $\Delta _{D}$ in the dual lattice $M=N^{\vee}$ gives
a basis for the space of global sections of $\mathcal{O}(D)$, and the
volume of $D$ can be computed as $\vol(X,D)= n!\vol_{M} (\Delta _{D})$, where
$\vol_{M}$ is the Haar measure on $M_{\R}$ which gives covolume 1 to
the lattice $M$. The divisor $D$ is nef if and only if the function
$\Psi_{D}$ is concave and, if this is the case, its degree coincides
with its volume.

In \cite{BurgosPhilipponSombra:agtvmmh}, this toric dictionary has
been extended to cover some of the arithmetic properties of toric
varieties. Let $\K$ be a global field, that is, a field which is
either a number field or the field of rational functions of a
projective curve, and suppose that $X$ is a toric variety over
$\K$. Then, a toric metrized divisor $\ov D$ on $X$ defines a
family of functions $\psi_{\ov D,v}\colon N_{\R}\to \R$ indexed by
the set of places $\mathfrak M_{\K}$ of $\K$ such that $\psi_{\ov
  D,v}=\Psi_{D}$ for all $v$ except for a finite number of them.  By
duality, this family of functions gives rise to a family of concave
functions on the polytope $\vartheta_{\ov D, v}\colon \Delta_{D}\to
\R$, called the local roof functions of $\ov D$.  The global roof function
$\vartheta_{\ov D}$ of a metrized divisor $\ov D$ is defined as a
weighted sum over all places of these local roof functions.  The
convex subset $\Theta_{\ov D}\subset \Delta_{D}$ is the set
of points where $\vartheta_{\ov D}$ takes nonnegative values.

These objects encode many of the Arakelov-theoretical properties of
$\ov D$.  For instance, a metrized divisor $\ov D$ is called
semipositive in \cite{BurgosPhilipponSombra:agtvmmh} if its metrics
are uniform limits of semipositive  smooth (respectively algebraic) 
metrics for the Archimedean (respectively non-Archimedean)
places. Then, $\ov D$ is semipositive if and only if all the functions
$\psi_{\ov D,v}$ are concave. If this is the case, the height of $X$
with respect to $\ov D$ can be computed as
\begin{displaymath}
  \h_{\ov D}(X)=(n+1)!\int_{\Delta _{D}}\vartheta _{\ov D}\dd\vol_{M}.
\end{displaymath}
Moreover, we show in the present text how these notions and results extend to toric metrized $\R$-divisors. We refer the reader to \S \ref{sec:toric-metrized-r} for the precise definitions and more details.

The arithmetic volume and the $\chi$-arithmetic volume of a metrized
$\R$-divisor~$\ov D$ measure respectively the growth of the number of small
sections of the multiples of $\ov D$ and the growth of the Euler
characteristic of the space of sections of the multiples of $\ov D$
(Definition \ref{def:3}).  Our first main result in this text are
formulae for the arithmetic volume and the $\chi$-arithmetic volume of
a toric metrized $\R$-divisor (Theorem~\ref{thm:1}).

\begin{introthm}\label{thm:6} Let  $X$ be a proper toric variety  over
  $\K$ and  $\ov D$ a toric  
  metrized $\R$-divisor on $X$.  Then  
  the arithmetic volume of $\ov D$ is given by
  \begin{displaymath}
    \avol(X,\ov D)= 
    (n+1)!\int _{\Delta_{D} }\max(0,\vartheta_{\ov D}) \dd\vol_{M}=
    (n+1)!\int _{\Theta_{\ov D} }\vartheta_{\ov D} \dd\vol_{M},
  \end{displaymath}
while its $\chi$-arithmetic volume is given by
  \begin{displaymath}
    \wh \vol_{\chi}(X,\ov D)=(n+1)!\int _{\Delta_{D} }\vartheta_{\ov
      D} \dd\vol_{M}.
  \end{displaymath}
\end{introthm}

The height is defined for DSP metrized $\R$-divisors, that is,
differences of semipositive ones, whereas  the arithmetic
volume and the $\chi$-arithmetic volume are defined for any
metrized $\R$-divisor.  Observe that when $\ov D$ is semipositive, the
$\chi$-arithmetic volume agrees with the height and
the formula for $\wh \vol_{\chi}(X,\ov D)$ coincides with that for
$\h_{\ov D}(X)$. Nevertheless, we show that the notion of height no
longer coincides with that of $\chi$-arithmetic volume for arbitrary DSP
$\R$-divisors (Examples \ref{exm:4} and~\ref{exm:5}).
 
Formulae similar to those in Theorem \ref{thm:6} were previously
obtained by Yuan \cite{Yuan:valb}, \cite{Yuan:valbII} and by Boucksom
and Chen \cite{BoucksomChen:Obfs} for a metrized divisor $\ov D$ on a
variety over a number field, under the hypothesis that the underlying
divisor~$D$ is big and that the metrics at the non-Archimedean places
are given by a global projective model over the ring of integers of
the number field.  These formulae are expressed in terms of the
integral of a function over the Okounkov body of the divisor. The
Okounkov body is a generalization to arbitrary divisors
of the polytope that appears in toric geometry. Indeed, the functions
introduced by Yuan and by Boucksom and Chen can be seen as a
generalization to arbitrary  metrized divisors (under the
aforementioned hypothesis) of the global roof function.

Our second main result is the following characterization of 
positive toric metrized $\R$-divisors (Theorem \ref{thm:2}). 

\begin{introthm}\label{thm:7}
Let  $X$ be a proper toric variety  over $\K$ and  $\ov D$ a toric 
  metrized $\R$-divisor on $X$.  Then
  \begin{enumerate}
  \item $\ov D$  is ample if
    and only if  $\Psi_{D} $ is strictly
    concave on $\Sigma $, the
    function $\psi _{\ov D,v}$ is concave for all $v\in \mathfrak
    M_{\K}$,  and $ \vartheta _{\ov D}(x)>0$ for all $x\in \Delta_{D}$;

  \item $\ov D$ is nef if and only if $\psi _{\ov D,v}$ is concave for
    all $v\in \mathfrak M_{\K}$ and $ \vartheta _{\ov D}(x)\ge0$ for
    all $x\in \Delta_{D}$;
  \item $\ov D$ is big if and only if
    $\dim(\Delta_{D})=n $ and there exists $x\in\Delta_{D}$ such that
    $\vartheta _{\ov D}(x)>0$; 
  \item $\ov D$ is pseudo-effective if and only if 
    there exists $x\in\Delta_{D}$ such that $\vartheta
    _{\ov D}(x)\ge0$; 
\item $\ov D$ is effective if and only if $0\in
  \Delta_{D}$ and $\vartheta_{\ov D,v}(0)\ge0$ for all $v\in
  \mathfrak{M}_{\K}$.  
  \end{enumerate}
\end{introthm}

There are several questions one can ask about the relations between the different
notions of positivity. An effective metrized
$\R$-divisor is  also pseudo-effective and, conversely, one can ask if any
pseudo-effective metrized $\R$-divisor is linearly equivalent to an
effective one. As Moriwaki pointed out, this question can be seen as
an extension of Dirichlet's unit theorem to metrized $\R$-divisors on
 varieties \cite{Moriwaki:tdutav}.

 Another relevant question is whether one can approximate
 pseudo-effective or big metrized $\R$-divisors by nef or ample
 ones. A Zariski decomposition of a big metrized $\R$-divisor
 $\ov D$ amounts to its decomposition, up to a birational
 transformation, into an effective part and a nef part which has the
 same arithmetic volume as $\ov D$.  Such a decomposition always
 exists when the underlying variety is a curve over a number field
 \cite{Moriwaki:zdas} but it does not always exist for 
 varieties of higher dimension \cite{Moriwaki:tdutav}.

 In the absence of a Zariski decomposition, one can ask for the
 existence of an arithmetic Fujita approximation.  The existence of an
 arithmetic Fujita approximation was proved by Yuan \cite{Yuan:valb}
 and by Chen \cite{Chen:afa} for the case when $\K$ is a number field,
 $D$ is a divisor, the metrics at the infinite places are smooth, and
 those at the finite places come from a common projective model over
 $\cO_{\K}$.

 As a consequence of our characterization of the different notions of
 arithmetic positivity, we give a positive answer to the Dirichlet's
 unit theorem for toric varieties when the base field $\K$ is an
 $A$-field, that is, a number field or the function field of a curve
 over a finite field. We also give a criterion for when a toric
 Zariski decomposition exists and we prove a toric Fujita
 approximation theorem (Theorem~\ref{thm:4}).
 
\begin{introthm} \label{thm:8}
Let $X$ be a proper toric variety over $\K$ and $\ov D$ a toric metrized
$\R$-divisor on $X$. 
\begin{enumerate}
\item Assume that $\K$ is an $A$-field. Then $\ov D$
  is pseudo-effective if and only if there exists $a\in \Delta_{D}$
  and $\alpha\in \K^{\times}\otimes \R$ such that
\begin{displaymath}
  \ov D+ \wh\div(\alpha \chi^{a})\ge 0.
\end{displaymath}

\item \label{item:28} Assume that $\ov D$ is big. Then there
  exists a birational toric map $\varphi\colon X'\to X$ and toric metrized
  $\R$-divisors $\ov P$, $\ov E$ on $X'$ such that $\ov P$ is nef, $\ov E$ is effective,
\begin{displaymath}
 \varphi^{*}\ov D=\ov P+\ov E\quad\text{and}\quad \wh\vol(X',\ov P)= 
 \wh\vol(X,\ov D)
\end{displaymath}
if and only if $\Theta_{\ov D}$ is a quasi-rational polytope
(Definition \ref{def:13}).

\item Assume that $\ov D$ is big. Then, for every
  $\varepsilon>0$, there exists a birational toric map $\varphi\colon
  X'\to X$ and toric 
metrized $\R$-divisors $\ov A$, $\ov E$ on $X'$ such that $\ov A$ is
ample, $\ov E$ is effective,
\begin{displaymath}
 \varphi^{*}\ov D=\ov A+\ov E\quad\text{and}\quad \wh\vol(X',\ov A)\ge
 \wh\vol(X,\ov D)-\varepsilon. 
\end{displaymath}
\end{enumerate}
\end{introthm}

A stronger version of the Zariski decomposition asks that the nef part
is maximal in a precise sense (Definition \ref{def:2}).  In the toric
setting, one can ask for the existence of a decomposition which is
maximal among all toric ones (Definition \ref{def:27}). Indeed, we
show that the criterion in Theorem \ref{thm:8}\eqref{item:28} extends
to pseudo-effective metrized $\R$-divisors if one uses this stronger
version of the Zariski decomposition (Theorem~\ref{thm:4}\eqref{item:55}).

A related question in whether the existence of a non-necessarily toric
Zariski decomposition of a big toric metrized $\R$-divisor is
equivalent to the existence of a toric one. In \S
\ref{sec:zariski-decomp} we give a partial affirmative answer to this
question restricting to toric varieties defined over $\Q$ and
arithmetic $\R$-divisors  (Theorem \ref{thm:5}). 
Roughly speaking, arithmetic $\R$-divisors correspond to metrized
$\R$-divisors whose metrics at the non-Archimedean places are given by
a single integral model and they are closer to the more traditional 
language of arithmetic varieties, see Example \ref{exm:3} for the
precise definition and more details.

\medskip

Since Arakelov geometry can be developed in different
frameworks, we discuss briefly the one in the present text. We
have chosen to use the adelic language introduced in this context 
by Zhang in \cite{Zhang:_small} instead of the
language of arithmetic varieties of Gillet and Soul\'e as
in~\cite{GilletSoule:ait}. This point of view is more general and flexible, and
allows to treat the cases of number fields and of function fields in a
uniform way.  Moreover, since general proper toric varieties are not
necessarily projective nor smooth, we do not add any hypothesis of
projectivity or smoothness.  Also, we work in the framework of
$\R$-divisors since it is the appropriate one for
Dirichlet's unit theorem on varieties, and it is also suitable for
discussing the Zariski decomposition and Fujita approximation problems.

There are several different definitions in the literature for the
various notions of arithmetic positivity, depending on the used
framework. Adding the appropriate technical hypothesis, these
different definitions are equivalent but, in general, they are not.
Due to our choice of working framework, we had to adjust these
pre-existing definitions (Definition
\ref{def:22}).
A systematic study of the definitions we propose here, including the openness of the ample and the big cones, the closedness of the nef and the
pseudo-effective cones, and the continuity of the arithmetic volume,
falls outside the scope of the present text. Nevertheless, our results
show that these definitions behave as expected  in the toric
case.

\vspace{3mm}
\noindent{\bf Acknowledgements.} 
We thank S\'ebastien Boucksom, Carlo Gasbarri and Juan Carlos Naranjo
for many useful discussions and pointers to the literature.  {We
  are also grateful to François Balla\"y for pointing us a mistake in
  the published version of Corollary~6.2.}

Part of this work was done while the authors met at the Universitat de
Barcelona, the Instituto de Ciencias Matem\'aticas (Madrid), the
Institut de Math\'ematiques de Jussieu (Paris) and the Servei
d'Urg\`encies of the
Hospital del Mar de Barcelona. We thank these
institutions for their hospitality. 

\section{Global fields}
 \label{sec:global-fields-adelic}
 Throughout this text, by a \emph{valued field} we mean a field $K$
 together with an absolute value $|\cdot|$ that is either Archimedean
 or associated to a nontrivial discrete valuation.  If $(K,|\cdot|)$
 is a valued field, then we set
\begin{math}
  K^{\circ}=\{x\in K\mid |x|\le 1\}.
\end{math}
When $|\cdot|$ is non-Archimedean, the unit ball $K^{\circ}$ is a ring.

Let $\K$ be a field and $\mathfrak{M}$ a
  family of absolute values on $\K$ with positive real weights. For
  each $v\in \mathfrak{M}$ we denote by $|\cdot|_v$ the
  corresponding absolute value, by $n_v\in \R_{>0}$ the weight, and by
  $\K_{v}$ the completion of $\K$ with respect to $|\cdot|_{v}$. 
We also set
\begin{displaymath}
    \lambda_{v}=
\begin{cases}
1& \text{ if } |\cdot|_{v} \text{ is Archimedean},\\
-\log|\varpi_{v}|_{v} & \text{ otherwise,}
\end{cases}
\end{displaymath}
where $\varpi_{v}$ is a uniformizer of the maximal ideal of $\K^{\circ}_{v}$.

 We say
  that $(\K,\mathfrak{M})$ is an \emph{adelic field} if the
  following conditions hold:
  \begin{enumerate}
  \item for each $v\in \mathfrak{M}$, the completion $\K_{v}$ is a
    valued field;
  \item for each $\alpha\in
  \K^{\times}$, 
  $|\alpha|_v=1$ except for a finite number of $v$.
  \end{enumerate}
  The adelic field $(\K, \mathfrak{M})$ is said to satisfy the
  \emph{product formula} if, for 
  all $\alpha\in 
  \K^{\times}$, 
    $$
    \sum_{v\in \mathfrak{M}}n_v\log|\alpha|_v=0.
    $$
    Let $\F$ be a finite extension of $\K$. For each $v\in
    \mathfrak{M}$, let $\mathfrak{N}_{v}$ be the set of pairs
    $w=(|\cdot|_{w},n_{w})$ where $|\cdot|_{w}$ is an absolute value
    on $\F$ that extends $|\cdot|_v$ and
\begin{equation}
\label{eq:13}
n_w=\frac{[\F_w:\K_v]}{[\F:\K]}n_v. 
\end{equation}
If $\mathfrak{N}= \sqcup_{v} \mathfrak{N}_{v}$, then
$(\F,\mathfrak{N})$ is an adelic field. For $w\in \mathfrak{N}$, we
note $w\mid v$ if $|\cdot|_{w}$ extends $|\cdot|_{v}$. By
\cite[Proposition 4.3]{Lang:FDG}, if
$(\K,\mathfrak{M})$ 
satisfies the product formula and $\F$ is separable over $\K$, then
$(\F,\mathfrak{N})$ satisfies the product formula too. 

\begin{exmpl}\label{exm:1}
  Let $\mathfrak{M}_\Q$ be the set formed by the Archimedean and the $p$-adic
  absolute values of $\Q$, normalized in the standard way, with all
  weights equal to $1$.  Then $(\Q,\mathfrak{M}_{\Q})$ is an adelic
  field that satisfies the product formula. We identify
  $\mathfrak{M}_{\Q}$ with the set $\{\infty\}\cup \{\text{primes of }
  \Z\}$.  For a number field $\K$, the construction above gives an
  adelic field $(\K,\mathfrak{M}_{\K})$ which satisfies the product
  formula too.
\end{exmpl}

\begin{exmpl}\label{exm:2} 
  Consider the function field ${\rm K}(C)$ of a smooth projective curve $C$
  over a field $k$. For each closed point $v\in C$ and $\alpha \in
  {\rm K}(C)^{\times}$, we denote by $\ord_{v}(\alpha )$ the order of
  $\alpha $ in the discrete valuation ring $\mathcal{O}_{C,v}$.  We
  associate to each $v$ the absolute value and weight given by
  $$
  |\alpha|_v=c_{k}^{-\ord_v(\alpha)}, \quad n_{v}= [k(v):k]
  $$ 
  with
\begin{displaymath}
    c_{k}=
    \begin{cases}
      \#k\ & \text{if } \#k < \infty,\\
      e&\text{if } \#k = \infty.
    \end{cases}
  \end{displaymath}
  Let  $\mathfrak{M}_{{\rm K}(C)}$ denote this set of absolute values and
  weights. The pair $({\rm K}(C),\mathfrak{M}_{{\rm K}(C)})$ is an
  adelic field which satisfies the product formula, since the degree
  of a principal divisor is zero.  In this case, 
  $\lambda_{v}=\log(c_{k})$ for all $v$.

  More generally, let $\K$ be a finite
  extension of ${\rm K}(C)$. Applying the construction in~\eqref{eq:13}, we obtain an
  adelic field $(\K,\mathfrak{M}_{\K/{\rm K}(C)})$.
  In this geometric setting, this construction can be formulated as
  follows.  Let $\pi \colon B \to C$ be a dominant morphism of smooth
  projective curves over $k$ such that the finite extension
  ${\rm K}(C)\hookrightarrow \K$ identifies with $\pi ^{\ast}\colon
  {\rm K}(C)\hookrightarrow {\rm K}(B)$.  For a closed point $v\in C$, the
  absolute values of~$\K$ that extend $|\cdot|_{v}$ are in bijection
  with the closed points of the fibre of $v$. For each closed point
  $w\in \pi^{-1}(v)$, the corresponding absolute value and weight are
  given, for $\beta\in {\rm K}(B)^{\times}$, by
  \begin{equation} \label{eq:14}
    |\beta|_w=c_{k}^{-\frac{\ord_w(\beta)}{e_{w}}}, 
     \quad n_{w}= \frac{e_{w}[k(w):k]}{[{\rm K}(B):{\rm K}(C)]},
  \end{equation}
  where $e_{w}$ is the ramification index of $w$ over $v$. We have
  \begin{equation}\label{eq:3}
    \lambda _{w}=\log(c_{k})/e_{w}.
  \end{equation}
  Observe that this structure of adelic field on $\K$ depends on the
  extension and not just on the field ${\rm K}(B)$. 
  For instance,  $({\rm K}(C), \mathfrak{M}_{{\rm K}(C)})$ corresponds to the
  identity map $C\to C$, but another finite morphism
  $\pi\colon C\to C$ may give a different structure of adelic field on
  ${\rm K}(C)$. The projection formula for the map $\pi $ implies that, for each
  $v\in \mathfrak{M}_{{\rm K}(C)}$, the equation
  \begin{displaymath}
    [\K:{\rm K}(C)]=\sum_{\substack{w\in \mathfrak{M}_{\K/{\rm K}(C)}\\v\mid w}}[\K_{w}:{\rm K}(C)_{v}]
  \end{displaymath}
  is satisfied. From which we obtain that $(\K,\mathfrak{M}_{\K/{\rm K}(C)})$
  also satisfies the product formula.
\end{exmpl}

\begin{defn} \label{def:6} A \emph{global field} is a finite extension
  $\K/\Q$ or $\K/{\rm K}(C)$ for a smooth projective curve $C$ over a field
  $k$, with the structure of adelic field given in
  examples~\ref{exm:1} or \ref{exm:2}, respectively. To lighten the
  notation, we will usually denote those global fields as $\K$ 
  although, in the function field case, the structure of adelic field
  depends on the particular extension. 
  In both cases, we will denote by  $\mathfrak{M}_{\K}$ the set of
  places and by $d_{\K}$ the degree of the extension.
\end{defn}

Note that our use of the terminology ``global field'' is slightly more
general than the usual one since, in
the function field case, we allow an arbitrary base field. The price to pay for this
greater generality is that, in the function field case, we can not use
the nice topology of the adeles. Instead, we will have to use geometric
arguments. Following Weil \cite{Weil:bnt}, we will use the terminology
``{\it A}-field'' for the global fields which are either a number field or
a finitely generated extension of degree of
transcendence 1 of a finite field. 

We recall the notion of
$\mathfrak{M}_{\K}$-divisor, which can be also found in \cite[Chapter~V]{Lang:ANT} for  the case of number fields.

\begin{defn} \label{def:16} Let $\K$ be a global field. An
  \emph{$\mathfrak{M}_{\K}$-divisor} is a collection
  $\mathfrak{c}=\{\mathfrak{c}_{v}\}_{v\in \mathfrak{M}_{\K}}$ of
  positive real numbers such that $\mathfrak{c}_{v}=1$ for all but
  finite number of~$v$ and such that $\mathfrak{c}_{v}$ belongs to the
  image of $|\cdot|_{v}$ for all non-Archimedean $v$. We set
\begin{displaymath}
\whL(\mathfrak{c})=\{ \gamma\in \K \ |\ |\gamma|_{v}\le
\mathfrak{c}_{v} \text{ for all } v\}
 \end{displaymath}
and 
\begin{displaymath}
\whl(\mathfrak{c})=
  \begin{cases}
    \log (\# \whL(\mathfrak{c}))&\text{ if $\K$ is a number field,} \\
    \log(c_{k}) \dim_{k} (\whL(\mathfrak{c})) &\text{ if $\K$ is a
      function field.}
  \end{cases}
\end{displaymath}
We also set $\wh \deg(\mathfrak{c})= \sum_{v} d_{\K}\, n_{v}\log(\mathfrak{c}_{v})$. 
\end{defn}

\begin{exmpl} \label{exm:10} Let $\K={\rm K}(B)/{\rm K}(C)$ be an extension of
  function fields viewed as a global field as in Example
  \ref{exm:2}. Let $\mathfrak{c}=(\mathfrak{c}_{v})_{v}$ be an
  $\mathfrak{M}_{\K}$-divisor. For each closed point $v\in B$, the
  condition that $\mathfrak{c}_{v}$ belongs to the image of the
  absolute value $|\cdot|_{v}$ is equivalent to $\log(\mathfrak
  c_{v})/\lambda _{v}\in \Z$. Consider the Weil divisor  on
  $B$ given by 
  \begin{displaymath}
  D(\mathfrak{c})=\sum_{v} d_{v} [v]  
  \end{displaymath}
  with $d_{v}=\log(\mathfrak c_{v})/\lambda _{v}$. Let
  $\LL(D(\mathfrak{c}))$ be the associated linear series and $\lL(D(\mathfrak{c}))$ its dimension. Then
\begin{equation}
  \label{eq:28}
 \whL(\mathfrak{c})=\LL(D(\mathfrak{c})),\quad  
\whl(\mathfrak{c})=\log(c_{k})\lL(D(\mathfrak{c})),\quad \wh
\deg(\mathfrak{c}) 
=\log(c_{k})\deg(D(\mathfrak{c})).
\end{equation}
These equalities follow easily from the definitions. For instance, we
prove the last equation with the help of \eqref{eq:14} and
\eqref{eq:3}:
\begin{multline*}
 \wh
\deg(\mathfrak{c}) 
=\sum_{v}d_{\K}n_{v}\log(\mathfrak{c}_{v})=
\sum_{v}d_{\K}\frac{e_{v}[k(v):k]}{d_{\K}}\lambda_{v} 
d_{v}\\=\log(c_{k})\sum_{v}d_{v}[k(v):k]=
\log(c_{k})\deg(D(\mathfrak{c})). 
\end{multline*}
Thus, 
an $\mathfrak{M}_{\K}$-divisor can be identified with a Weil divisor on the
curve $B$. This identification respects their linear series and the
associated invariants, up to the multiplicative constant
$\log(c_{k})$. 
\end{exmpl}

\begin{lem}\label{lemm:2}
  Let $\K$ be a global field. Then, there exists $\kappa>0$ depending
  only on~$\K$ such that,
  for any $\mathfrak{M}_{\K}$-divisor $\mathfrak{c}$, 
    \begin{displaymath}
\big|\whl(\mathfrak{c})-
    \max(0,\wh \deg(\mathfrak{c}))\big|\leq \kappa.
    \end{displaymath}
\end{lem}
\begin{proof} If $\K$ is a number field, then, in the
notation of  \cite[page 101]{Lang:ANT}, 
\begin{displaymath}
  \whl(\mathfrak{c})=\log(\lambda
  (\mathfrak{c})), \quad \wh \deg(\mathfrak{c})= \log\|\mathfrak{c}\|_{\K} 
\end{displaymath}
and \cite[Chapter V, Theorem 0]{Lang:ANT} gives the result. 

Hence, we only have to consider the case when $\K$ is the function
field of a smooth projective curve $B$. Let $D=D(\mathfrak{c})$ be the
Weil divisor associated to the $\mathfrak{M}_{\K}$-divisor
$\mathfrak{c}$ as in Example \ref{exm:10}.  If $\deg(D)<0$, then
$\lL(D)=0$ and so $\whl(\mathfrak{c})=\max(0,\wh \deg(\mathfrak{c}))$
by \eqref{eq:28}, and the lemma is proved in this case. When
$\deg(D)\ge 0$, we have that  $\lL(D)\le \deg(D)+1$ and, by the
Riemann-Roch theorem \cite[Theorem 3.17]{Liu:agac}, 
\begin{displaymath}
  \lL(D)\ge \deg(D)-(g(B)-1)
\end{displaymath}
where $g(B)$ is the genus of $B$. Hence, $ |\whl(\mathfrak{c})-
    \max(0,\wh \deg(\mathfrak{c}))|\leq (g(B)-1)\log(c_k)$, thus
    proving the lemma.
\end{proof}

\begin{lem}\label{lemm:10}
  Let $\K$ be a global field and $S\subset  \mathfrak{M}_{\K}$ a finite
  subset. Let $\{\gamma _{v}\}_{v\in \mathfrak{M}_{\K}}$ be a collection of
  positive real numbers such that  $\gamma _{v}=1$ for all except a
  finite number of $v$, and such that
  \begin{displaymath}
    \prod_{v\in \mathfrak{M}_{\K}}\gamma_{v} ^{n_{v}}<1.
  \end{displaymath}
  Let $0<\eta\le 1$ be a real number.  Then, there is an integer
  $\ell_{0}\ge 1$ such that, for all~$\ell\ge \ell_{0}$, there exists
  $\alpha \in \K^{\times}$ with $|\alpha |_{v}\gamma_{v}^{\ell}\le 1$
  for all $v\in \mathfrak{M}_{\K}$ and $|\alpha
  |_{v}\gamma_{v}^{\ell}< \eta$ for all $v\in S$.
\end{lem}
\begin{proof}
Consider the finite set of places
  \begin{displaymath}
    S'=S\cup \{v\in \mathfrak{M}_{\K}\mid \gamma _{v}\not = 1\}\cup
    \{v\in \mathfrak{M}_{\K}\mid v \text{ is Archimedean}\}.
  \end{displaymath}
 For each $v\in S'$ we pick $\wt
  \gamma_{v}> \gamma_{v}$ in such a way that
  $\prod_{v}\wt \gamma_{v}^{n_{v}}<1$ while, for  $v\in
  \mathfrak{M}_{\K}\setminus S'$, we set $ \wt
  \gamma_{v}=\gamma_{v}=1$. Then for each  $v\in S'$  we can find
  an integer $\ell_{v}$ such that
  \begin{equation}
    \label{eq:5}
   \ell_{v}(\log (\wt \gamma _{v})-\log (\gamma _{v}))>\lambda_{v}-\log(\eta).
  \end{equation}
Choose $\ell_{0}\ge \max_{v\in S'}{\ell_{v}}$ satisfying also that
  \begin{displaymath}
    -\ell_{0}\sum_{v}d_{\K}\,n_{v}\log(\wt \gamma _{v})>\kappa,
  \end{displaymath}
  where $\kappa$ is the constant in Lemma \ref{lemm:2}.  Let $\ell\ge
  \ell_{0}$.  By \eqref{eq:5}, for each $v\in S'$ the interval $[-\ell
  \log (\wt \gamma _{v}),\log (\eta)-\ell\log (\gamma _{v})]$ has
  length bigger than $\lambda _{v}$. Therefore, we can choose
  $x_{v}\in \K_{v}^{\times}$ with
  \begin{displaymath}
    \frac{1}{\wt\gamma ^{\ell}_{v}}\le |x_{v}|_{v}<\frac{\eta}{\gamma
      ^{\ell}_{v}}. 
  \end{displaymath}
 Set $\mathfrak{c}_{v}=|x_{v}|_{v}$ for $v\in S'$ and
  $\mathfrak{c}_{v}=1$ for $v\not \in S'$. Then
  $\mathfrak{c}=(\mathfrak{c}_{v})_{v}$ is an $\mathfrak{M}_{\K}$-divisor with
  \begin{displaymath}
    \wh \deg (\mathfrak{c})=\sum_{v}d_{\K}\,n_{v}\log|x_{v}|_{v}\ge -\ell \sum_{v}d_{\K}\,n_{v}\log(\wt \gamma _{v})>\kappa.
  \end{displaymath}
 Lemma \ref{lemm:2} then implies that 
  $L(\mathfrak{c})\ne\{0\}$. Hence, we can find $\alpha \in
  \K^{\times}$ such that $|\alpha |_{v}\le \mathfrak{c}_{v}\le \gamma
  _{v}^{-\ell}$ for all $v\in \mathfrak{M}_{\K}$ and
  \begin{math}
    |\alpha |_{v}\le \mathfrak{c}_{v}< \eta \gamma _{v}^{-\ell}\end{math}
  for $v\in S'$, proving the result.
\end{proof}

\section{Adelic vector spaces}
\label{sec:adelic-vector-spaces}

Let $(K,|\cdot|)$ be a valued field and $V$ a vector space over
$K$. By a 
\emph{norm} on $V$ we will mean a norm in the usual sense in the
Archimedean case and a norm satisfying the ultrametric inequality in
the non-Archimedean case.
Let $(V,\|\cdot\|)$ be a normed vector space over $K$.  
If $E\subset K$ and $F\subset V$ we write
\begin{displaymath}
  |E|=\{|\alpha |\mid \alpha \in E\}\subset \R_{\ge 0}, \quad
  \|F\|=\{\|x \|\mid x \in F\}\subset \R_{\ge 0}.
\end{displaymath}
Let 
\begin{math}
  V^{\circ}=\{x\in V\mid \|x\|\le
  1\}
\end{math} be the \emph{unit ball} of $V$.  When $K$ is non-Archimedean,
$V^{\circ}$ is a $K^{\circ}$-module.

\begin{exmpl} \label{exm:12} Let $K$ be a valued field and $r\ge0$. We
  can give a structure of normed vector space to $K^{r}$ by
  considering, if $|\cdot|$ is Archimedean, the Euclidean norm and,
  if $|\cdot|$ is non-Archimedean, the $\ell^{\infty}$-norm. In
  precise terms, for $x=(x_{1},\dots,x_{r})\in K^{r}$,
\begin{displaymath}
 \Vert x\Vert= \begin{cases}
 \big(\sum_{i=1}^{r}|x_{i}|^{2}\big)^{1/2}  & \text{ if }
 |\cdot|\text{ is Archimedean},\\
\max_{i}|x_{i}|_{v} & \text{ otherwise}.
  \end{cases}
\end{displaymath}
This choice of norm gives the \emph{standard} structure of normed
vector space on $K^{r}$.
\end{exmpl}

We recall the notion of orthogonality of a basis in a normed
vector space.

\begin{defn} \label{def:10}
  Let $K$ be a valued field and $V$ a normed vector space over $K$. A
  set of vectors 
  $\{b_{1},\dots,b_{r}\}$ of $V$ is \emph{orthogonal} if, for all
  $\gamma  _{1},\dots,\gamma _{r}\in K$,
  \begin{displaymath}
    \left\| \sum_{i=1}^{r}\gamma _{i}b_{i}
    \right\|\ge \max_{i}\|\gamma _{i}b_{i}\|.
  \end{displaymath}
  An orthogonal set of vectors $\{b_{1},\dots,b_{r}\}$ is 
  \emph{orthonormal} if $\|b_{i}\|=1$ for all $i$.
\end{defn}

For an $r$-dimensional normed vector space, the presence of an
orthogonal basis allows to compare its unit ball with an ellipsoid in
the standard normed vector space $K^{r}$.

\begin{lem} \label{lemm:1} Let $(K,|\cdot|)$ be a valued field and
  $(V,\|\cdot\|)$ a normed vector space over $K$ of finite dimension
  $r$. Suppose that there is an orthogonal basis
  $b=\{b_{1},\dots,b_{r}\}$ of $V$ and let $\phi_{b}\colon V\to K^{r}$
  be the induced isomorphism.
 \begin{enumerate}
  \item \label{item:11} Suppose that $|\cdot|$ is Archimedean and consider the
    ellipsoid 
    \begin{displaymath}
E=\Big\{ (\gamma_{1},\dots, \gamma_{r})\in K^{r} \ \Big|\ 
    \sum_{i=1}^{r}|\gamma_{i}|^{2} \Vert b_{i} \Vert^{2}\le 1\Big\}.
    \end{displaymath}
 Then ${r}^{-1}E\subset \phi_{b}(V^{\circ})\subset \sqrt{r}E$.
\item \label{item:12} Suppose that $|\cdot|$ is 
  associated to a nontrivial discrete valuation with uniformizer
  $\varpi$. Let  $\alpha_{i}\in K$ such that $\|b_{i}\|\le |\alpha_{i}|
  <|\varpi|^{-1}\|b_{i}\|$. Then the vectors
  $\{\alpha_{1}^{-1}b_{1},\dots,\alpha_{r}^{-1}b_{r}\}$ form a basis
  of $V^{\circ}$.
In particular, if we consider the set 
    \begin{displaymath}
E=\big\{ (\gamma_{1},\dots, \gamma_{r})\in K^{r} \ \big|\ 
    \max_{i}|\gamma_{i}| \, \Vert b_{i}\Vert\le 1\big\},
\end{displaymath}
then $\phi_{b}(V^{\circ})= E$. 
\end{enumerate}
\end{lem}

\begin{proof}
We consider first the Archimedean case. On the one hand, let $x\in V^{\circ}$ and write
$x= \sum_{i}\gamma_{i}b_{i}$ with $\gamma_{i}\in K$. Then
$\max_{i}|\gamma_{i}|\|b_{i}\|\le \|x\|\le 1$, since the basis $b$ is
orthogonal. Hence $ \sum_{i}|\gamma_{i}|^{2}\|b_{i}\|^{2}\le r$,
which implies that $\phi_{b}(V^{\circ})\subset \sqrt{r}E$. On the
other hand, let $(\gamma_{1},\dots, \gamma_{r})\in {r}^{-1}E$ and
write $x=\sum_{i}\gamma_{i}b_{i}$ for the corresponding point of
$V$. We have 
$\|x\|\le r\max_{i}|\gamma_{i}|\|b_{i}\|\le
r (\sum_{i}|\gamma_{i}|^{2}\|b_{i}\|^{2})^{1/2}\le 1,
$
which implies that ${r}^{-1}E\subset V^{\circ}$.

Now consider the non-Archimedean case. Let $x=
\sum_{i}\gamma_{i}b_{i}\in V$. If $x\in V^{\circ}$, then
$$\max_{i}|\gamma_{i} | \, |\alpha_{i}|<
|\varpi|^{-1}\max_{i}|\gamma_{i}|\|b_{i}\|\le
|\varpi|^{-1}\Big\|\sum_{i}\gamma_{i}b_{i}\Big\|
\le|\varpi|^{-1}.
$$ 
Since the first inequality is strict, this implies that
$\max_{i}|\gamma_{i} | \, |\alpha_{i}|\le 1$, hence
$\gamma_{i} \alpha_{i}\in K^{\circ}$ and so
$V^{\circ}\subset\sum_{i}K^{\circ}\alpha_{i}^{-1}b_{i}$.  Conversely,
let $x\in \sum_{i}K^{\circ}\alpha_{i}^{-1}b_{i}$. Then $\|x\|\le
\max_{i}|\alpha^{-1}_{i}|\|b_{i}\| \le 1$, which proves the reverse
inclusion.
\end{proof}

The notion of orthogonality on general normed vector spaces is
delicate in the Archimedean case. By contrast this notion behaves
nicely in the non-Archimedean case. For instance, if $(V,\|\cdot\|)$
is a normed 
vector space of dimension $r$ over a
non-Archimedean valued field and $\{b_{1},\dots,b_{r}\}$
is a orthogonal basis of $V$, then
\begin{equation}
  \label{eq:18}
\left\| \sum_{i=1}^{r}\gamma _{i}b_{i}
    \right\|= \max_{i}\|\gamma _{i}b_{i}\|.  
\end{equation}
Moreover,  orthogonal bases always exist in the non-Archimedean
case. 

\begin{prop} \label{prop:18}
  Let $(V,\|\cdot\|)$
be a normed 
vector space of dimension $r$ over a
non-Archimedean valued field $(K,|\cdot|)$. Then there exists an
orthogonal basis of $V$.
\end{prop}
\begin{proof}
  When $K$ is locally compact, the proof can be found in
  \cite[Proposition~II.3]{Weil:bnt}. For completeness we include a
  proof for an arbitrary discrete valuation. 

  Let $|K^{\times}|$ be the set of nonzero values of $K$. This is a
  discrete subgroup of $\R_{>0}$.  It can be verified that, if
  $x_1,\dots,x_k$ are nonzero vectors of $V$ such that the norms
  $\|x_{i}\|$ belong to different cosets with respect to
  $|K^{\times}|$, then these vectors are orthogonal. Since orthogonal
  vectors are linearly independent, we deduce that the set of norms
  $\|V\setminus \{0\}\|$ is a finite union of at most $r$ cosets of
  $|K^{\times}|$. Hence, this is a discrete subset of $ \R_{>0}$.
  
  Let now $b_{1},\dots,b_{r}$ be a basis of $V$. We construct
  a orthogonal basis inductively. Put $e_1=b_1$. For  $2\le k\le r-1$, assume that we have
  already chosen a set $e_1,\dots,e_k$ of orthogonal
  vectors that span 
  the same subspace as $b_1,\dots,b_k$. Choose a vector 
$e_{k+1}=b_{k+1}+\sum_{j=1}^{k}\alpha  _{j}e_{j}$ with the property that
\begin{equation}\label{eq:21}
  \|e_{k+1}\|=\inf\Big\{\Big \|b_{k+1}+\sum_{j=1}^{k}\alpha 
  _{j}e_{j}\Big\|\ \Big| \
  \alpha _{1},\dots,\alpha _{k}\in K\Big\}.
\end{equation}
This vector exists because of the discreteness of
$\|V\setminus\{0\}\|$. Condition \eqref{eq:21} implies that the set
$e_1,\dots,e_{k+1}$ is orthogonal.
\end{proof}

\begin{cor} \label{cor:3}
  The unit ball $V^{\circ}$ is a free $K^{\circ}$-module of rank $r$.
\end{cor}
\begin{proof}
  By Proposition \ref{prop:18}, $V$ admits an orthogonal basis. Thus,
  the statement follows from Lemma \ref{lemm:1}\eqref{item:12}.
\end{proof}

In the Archimedean case, a norm is determined by its unit ball. This
is not true in the discrete valuation case. For a normed space
$(V,\|\cdot\|)$ over a valued field $(K,|\cdot|)$, the \emph{norm associated to
the unit ball} is defined, for $x\in V$, as
\begin{equation}
  \label{eq:22}
  \|x\|_{V^{\circ}}=\inf\{|\alpha | \mid \alpha \in K, x\in \alpha
  V^{\circ}\}
\end{equation}
In general,  $\|x\|\le \|x\|_{V^{\circ}}$.
Following \cite{Gaudron:gnals}, we say that the normed space
$(V,\|\cdot\|)$ is \emph{pure} if
  $\|x\|= \|x\|_{V^{\circ}}$ for all $x\in V$. The \emph{purification}
  of $(V,\|\cdot\|)$
  is the normed vector space
  $(V,\|\cdot\|_{V^{\circ}})$. 

All normed spaces over an Archimedean field are pure. In the
non-Archimedean case, we have the following criterion.

\begin{prop}\label{prop:9} Let $(V,\|\cdot\|)$ be a normed space over
  a non-Archimedean valued field $(K,|\cdot|)$. Then the following
  conditions are equivalent:
  \begin{enumerate}
  \item \label{item:19} $(V,\|\cdot\|)$ is pure;
  \item \label{item:20} $\|V\|=|K|$;
  \item \label{item:21} there exists an orthonormal basis of $V$;
  \item \label{item:22} every $K^{\circ}$-basis of $V^{\circ}$ is orthonormal.
  \end{enumerate}
\end{prop}

\begin{proof}
Since the valuation is
discrete, we have that 
\begin{displaymath}
 \|x\|_{V^{\circ}}=\min\{ t\in |K|\mid t\ge \|x\|\}.  
\end{displaymath}
The equivalence of \eqref{item:19} and \eqref{item:20} follows easily
from this.  The fact that \eqref{item:21} implies \eqref{item:20} is
clear, whereas the reverse implication follows from Proposition
\ref{prop:18}. By Corollary~\ref{cor:3},  $V^{\circ}$ admits a
$K^{\circ }$ basis, and so \eqref{item:22} implies
\eqref{item:21}. Thus, it only remains to show that \eqref{item:19}
implies \eqref{item:22}.

Consider $K^{r}$ with its standard structure of normed vector space
as in Example~\ref{exm:12}. With this structure, the standard
basis is orthonormal. Let $b=\{b_{1},\dots,b_{r}\}$ be a
$K^{\circ}$-basis of $V^{\circ}$ and $\phi_{b} \colon V\to K^{r}$ the
isomorphism given by this basis. The image of $V^{\circ}$ by this
isomorphism is the unit ball of $K^{r}$. Therefore, if $V$ is pure,
$\phi_{b}$ is an isometry and $b$ is an orthonormal basis.
\end{proof}

Partly following \cite{Gaudron:gnals}, we introduce a notion of adelic
vector space. As Gaudron points out, this notion
extends that of Hermitian vector bundle, which is at the base of
Arakelov geometry.

\begin{defn}
  Let $(\K,\mathfrak{M})$ be an adelic field. An \emph{adelic
    vector space} over $(\K,\mathfrak{M})$ is a pair $\ov V= (V,\{\|\cdot\|_{v}\}_{v\in
    \mathfrak{M}})$ where $V$ is a vector space 
  over $\K$ and, for each $v\in \mathfrak{M}$, 
  $\|\cdot\|_{v}$ is a norm on the completion $V_{v}:=V\otimes
  \K_{v}$, satisfying, 
  for each 
  $x\in V\setminus \{0\}$, that $\|x\|_{v}=1$ for all but a finite
  number of $v$. 

  Let $\ov V$ be an adelic vector space over $(\K,\mathfrak{M})$. An
  element $x\in V$ is
  \emph{small} if $\|x\|_{v}\le 1$ for all $v\in \mathfrak{M}$. A
  small element $x\in V$ is \emph{strictly small} if $\prod_{v\in
    \mathfrak{M}}\|x\|^{n_{v}}_{v}< 1$. If $S\subset \mathfrak{M}$ is
  a finite set, then a small element $x\in V$ is
  \emph{strictly small on} $S$ if $\|x\|_{v}< 1$ for all $v\in S$.
  
  The adelic vector space $\ov V$ is called \emph{pure} if
  $(V_{v},\|\cdot\|_{v})$ is pure for all $v\in \mathfrak{M}$. The
  \emph{purification} of $\ov V$ is the adelic vector space $\ov
  V_{\pur}=(V,\{\|\cdot\|_{v,V_v^\circ}\}_{v})$.  If $\ov V$ is finite
  dimensional, it is called \emph{generically trivial} if there is a
  $K$-basis of $V$ that is an orthonormal basis of $V_{v}$ for all but
  a finite number of $v$.  Clearly, if $\ov V$ is generically trivial,
  the same is true for its purification.
\end{defn}

Note that Gaudron's definition of adelic vector space includes the
condition of being generically trivial. 

\begin{exmpl} \label{exm:7} Let $\K$ be a global field and
  $r\ge0$. The \emph{standard} structure of adelic vector space on
  $\K^{r}$ is defined by choosing the standard norm on $\K_{v}^{r}$ for each
  place $v\in \mathfrak{M}_{\K}$,
  as explained in Example \ref{exm:12}. The obtained adelic vector
  space is pure and generically trivial.
\end{exmpl}

\begin{exmpl} \label{exm:13}
  Let $\K$ be a global field. A \emph{normed vector bundle} is:
  \begin{enumerate}
  \item when $\K$ is a number field, a locally free
    $\mathcal{O}_{\K}$-module $\mathcal{E}$, together with the choice
    of a norm $\|\cdot\|_{v}$ on $\mathcal{E}\otimes \K_{v}$ for each
    Archimedean place $v$;
  \item when $\K={\rm K}(B)$ is the function field of a smooth projective
    curve, a locally free $\mathcal{O}_{B}$-module~$\mathcal{E}$.
  \end{enumerate}
  To a normed vector bundle $\mathcal{E}$, we associate the adelic
  vector space $\ov E=(E,\{\|\cdot\|_{v}\})$ given by the vector space
  $E=\mathcal{E}\otimes \K$, the given norm for each Archimedean
  place $v\in \mathfrak{M}_{\K}$, and the norm 
  \begin{displaymath}
    \|x\|_{v}=\inf\{|\alpha |_{v}\mid 
    \alpha \in \K, x\in \alpha \mathcal{E}_{v}\},
  \end{displaymath}
  for each non-Archimedean place $v$. Clearly, this adelic
  vector space is pure and generically trivial.
\end{exmpl}

The previous example covers all cases of pure and generically
trivial adelic vector spaces over a global field. 

\begin{prop} \label{prop:11} Let $\ov V$ be a finite dimensional
  adelic vector space over a global field. Assume that~$\ov V$ is pure and
  generically trivial. Then, it is the adelic vector space associated
  to a normed vector bundle.
\end{prop}

\begin{proof}
  We give the proof of this statement for the case of function fields only, the case of number
  fields being analogous. Let $\K={\rm K}(B)$ for a smooth projective
  curve~$B$. For each open subset $U\subset B$, we write
  \begin{displaymath}
    \mathcal{E}_{\ov V}(U)=\{x\in V \mid \|x\|_{v}\le 1, \ \forall
    v\in U\}.
  \end{displaymath}
  Clearly, $\mathcal{E}_{\ov V}$ is a sheaf of
  $\mathcal{O}_{B}$-modules. Let $v_{0}\in B$ and choose a
  $\K^{\circ}_{v_{0}}$-basis $b$ of $V^{\circ}_{v_{0}}$. Since $\ov V$
  is generically trivial, there is a basis $e$ of $V$ that is an
  orthonormal basis of $\ov V_{v}$ for all but a finite number of
  places $v$.  Let $U$ be the subset of $B$ containing the generic
  point, the point $v_{0}$, and all the closed points $v\in B$ such
  that $\det(e/b)$ is a unit of $\K^{\circ}_{v}$. Then, $U$ is a
  neighbourhood of $v_{0}$ such that $b$ is a $\K^{\circ}_{v}$-basis of
  $V^{\circ}_{v}$ for all closed points $v\in U$. This shows that
  $\mathcal{E}_{\ov V}$ is locally free.

Since
  $\ov V$ is pure, its norms agree with the norms induced by
  $\mathcal{E}_{\ov V}$, which completes the proof. 
\end{proof}

\begin{defn}\label{def:12}
  Let $\ov V$ be an adelic  vector space over a global field $\K$. 
  The set of small elements of $ \ov V$ is denoted by $\wh
  H^{0}(\ov V)$. We further write
  \begin{displaymath}
    \wh
  h^{0}(\ov V) =
  \begin{cases}
    \log (\# \wh H^{0}(\ov V))&\text{ if $\K$ is a number field,} \\
    \log(c_{k}) \dim_{k} (\wh H^{0}(\ov V)) &\text{ if $\K$ is a
      function field.}
  \end{cases}
  \end{displaymath}
\end{defn}
If $\K$ is a function field over a finite field, then
\begin{math}
      \wh
  h^{0}(\ov V) =
    \log (\# \wh H^{0}(\ov V))
\end{math}
since $c_{k}= \#k$. Thus, both definitions agree for $A$-fields.

\begin{exmpl}
  Let $\K$ be a global field and $\mathfrak{c}=(\mathfrak{c}_{v})_{v}$ an
  $\mathfrak{M}_{\K}$-divisor. We define a normed vector space $\ov
  V(\mathfrak{c})$, given by $V(\mathfrak{c})=\K$ and $\|\alpha
  \|_{v}=\mathfrak{c}_{v}^{-1}|\alpha |_{v}$. Then $\ov
  V(\mathfrak{c})$ is a pure and generically trivial adelic vector
  space over $\K$. Moreover, $\wh
  H^{0}(\ov V(\mathfrak{c}))=\whL(\mathfrak{c})$ and  $\wh
  h^{0}(\ov V(\mathfrak{c}))=\whl(\mathfrak{c})$.  
\end{exmpl}

When $\K$ is a function field and the adelic vector space
comes from a normed vector bundle, the sets $\wh H^{0}(\ov V)$ can be
interpreted as the space of global sections of the model defining the
metric.

\begin{exmpl}\label{exm:8} Let
  $\K={\rm K}(B)$ be the function field of a smooth projective curve
  with the structure of global field given by Example \ref{exm:2}.
  Let $\mathcal{E}$ be a locally free $\mathcal{O}_{B}$-module and
  $\ov E$ the associated adelic vector space as in Example
  \ref{exm:13}.  Then there is a canonical isomorphism
  \begin{displaymath}
    H^{0}(B,\mathcal{E})\simeq \wh H^{0}(\ov E),
  \end{displaymath}
  given by restriction to the generic fibre. In particular, 
  \begin{equation}\label{eq:24}
     \wh h^{0}(\ov E)=\log (c_{k}) h^{0}(B,\mathcal{E}).
  \end{equation}
\end{exmpl}

We next recall the definition of the Euler characteristic of an adelic
vector space.  For general function fields, the definition 
differs from that for  $A$-fields since, in that case,  we do not
dispose of a Haar measure on the corresponding space of adeles.

\begin{defn} \label{def:9} Let $\K$ be a global field and $\ov V$ a
  generically trivial 
  adelic vector space over $\K$ of finite dimension $r$.

Assume first that $\K$ is a number field. 
Let $\A$ be its ring of adeles, set $V_{\A}=V{\otimes}_{\K}\A$, and
consider the adelic unit ball 
\begin{displaymath}
 V^{\circ}_{\A}=\prod_{v}V_{v}^{\circ}\subset V_{\A}. 
\end{displaymath}
 Let $\mu$ be a Haar measure on $\K^{r}_{\A}$.  The
\emph{Euler characteristic} of $\ov V$ is defined as
\begin{displaymath}
  \wh\chi(\ov V)=\log \Big(\frac{\mu (V^{\circ}_{\A})}{\mu
    (V_{\A}/V)}\Big). 
\end{displaymath}
This number does not depend on the choice of $\mu$. In other words,
the Euler characteristic is the ratio between the volume of the unit
ball and the covolume of the lattice $V\subset V_{\A}$.

Next assume that $\K$ is a function field. Let $b$ be a basis of $V$
over $\K$ and, for each $v\in \mathfrak{M}_{\K}$, choose a basis $b_{v}$
of $V^{\circ}_{v}$ over $\K^{\circ}_{v}$. The \emph{Euler
  characteristic} of $\ov V$ is defined as
\begin{equation}\label{eq:33}
  \wh\chi(\ov V)=\sum_{v}d_{\K}\, n_{v}\log \left|\det(b_{v}/b)\right|_{v},
\end{equation}
where $b_{v}/b$ denotes the matrix of $b_{v}$ with respect to the
basis $b$.  This quantity does not depend on the choices of bases
because of the product formula. 

In both cases, the formula makes sense because the adelic vector space
is generically trivial.
\end{defn}

\begin{rem} \label{rem:1} Let $\K$ be a number field and $\ov V$  an
  adelic vector space over $\K$ of finite dimension $r$.  With
  notations as in Definition \ref{def:9}, for each Archimedean place $v$ , we consider the Lebesgue measure $\mu_{v}$ on
  $\K_{v}^{r}$. Therefore, for $v$ real, $\mu_{v}(
  (\K_{v}^{r})^{\circ})=2^{r} $ whereas, for $v$ complex, $\mu_{v}(
  (\K_{v}^{r})^{\circ})=\pi ^{r} $. We choose a basis $b$ of $V$ over $\K$
  and, for each non-Archimedean  $v$,
  we also choose a basis $b_{v}$ of $V^{\circ}_{v}$ over
  $\K^{\circ}_{v}$.  Then
\begin{equation}\label{eq:12}
  \wh\chi(\ov V)=  \sum_{v\mid \infty}\log(\mu
  _{v}(\phi_{b,v}(V^{\circ}_{v})))+\sum_{v\nmid
    \infty}d_{\K}\, n_{v}\log|\det 
  (b_{v}/b)|_{v},
\end{equation}
where $\phi_{b,v}\colon V_{v}\to \K_{v}^{r}$ is the isomorphism
induced by the basis $b$. This is the analogue for number fields of
formula \eqref{eq:33}.
\end{rem}

When $\ov V$ is an adelic vector space over a function field coming from a
normed vector bundle, its Euler characteristic coincides with the
Euler characteristic of the associated model, up to an additive
constant. 

\begin{exmpl} \label{exm:9} Let $\pi\colon B\to C$ be a dominant
  morphism of smooth projective curves over a field $k$ and consider the function field
  $\K={\rm K}(B)$ with the structure of global field as in Example
  \ref{exm:2}. Let $\mathcal{E}$ be a locally free
  $\mathcal{O}_{B}$-module of rank $r$ and $\ov E$ the associated
  generically trivial adelic vector space as in Example
  \ref{exm:13}. Choose a $\K$-basis $b$ of $E$ and a
  $\K^{\circ}_{v}$-basis $b_{v}$ of $E^{\circ}_{v}$ for each $v\in
  \mathfrak{M}_{\K}$. Then
  \begin{multline*}
    \wh \chi(\ov E)=\sum_{v}d_{\K}\, n_{v}\log |\det
    (b_{v}/b)|_{v}=
    \sum_{v}e_{v}[k(v):k]\log \Big(c_{k}^{\frac{-\ord_{v}(\det
        (b_{v}/b))}{e_{v}}}\Big)\\=\log (c_{k})
    \sum_{v}[k(v):k]\ord_{v}(\det
        (b/b_{v}))=\log (c_{k})\deg(\mathcal{E}),
  \end{multline*}
  where $e_{v}$ denotes the ramification index of $v$ over $\pi(v)$.
 The Riemann-Roch theorem for vector bundles on curves then implies
  \begin{equation*}
    \wh \chi(\ov E)=\log (c_{k})(\chi(\mathcal{E})+r(g(B)-1)),
  \end{equation*}
  where $\chi(\cE)$ is the Euler characteristic of $\cE$ and $g(B)$ is the genus of $B$.
\end{exmpl}

The space of small elements and the Euler characteristic of an adelic
vector space depend only on its purification, as it follows
immediately from the definitions. 

\begin{prop} \label{prop:10}
  Let $\ov V$ be a generically trivial adelic vector space over a
  global field. Then   \begin{displaymath}
    \wh H^{0}(\ov V)=\wh H^{0}(\ov V_{\pur}), \quad \wh
    \chi(\ov V)= \wh\chi(\ov V_{\pur}).
  \end{displaymath}
\end{prop}

The presence of an orthogonal basis allows us to estimate the
Euler characteristic of an adelic vector space.

\begin{prop} \label{prop:7} Let $\ov V$ be a generically trivial
  adelic vector space over a global field $\K$ of finite dimension
  $r$. Let $b=\{b_{1},\dots,b_{r}\}$ be a basis of $V$ over $\K$ which is an
  orthogonal basis of $V_{v}$ for all $v\in \mathfrak{M}_{\K}$. Let
  $S\subset \mathfrak{M}_{\K}$ be the finite set of non-Archimedean
  places such that $\|b_{i}\|_{v}\not = 1$ for some $i$. Then
  \begin{displaymath}
    \left| \wh\chi(\ov V)-\sum_{v\in
        \mathfrak{M}_{\K}}\sum_{i=1}^{r}d_{\K}\, n_{v}\log
      (\|b_{i}\|_{v}^{-1})\right| 
    \le d_{\K}\, r\Big(\log(\pi r)+\sum_{v\in S}n_{v}\lambda_{v}\Big).
  \end{displaymath}
\end{prop}
\begin{proof}
  We consider the case when $\K$ is a number field. We will freely use
  the notation in Remark \ref{rem:1}. Let $v$ be an Archimedean place
 and consider the ellipsoid
    \begin{math}
      E_{v}=\{ (\gamma_{1},\dots, \gamma_{r})\in \K_{v}^{r} \mid
      \sum_{i}|\gamma_{i}|^{2}_{v}\Vert b_{i}\Vert _{v}^{2}\le 1\}.
    \end{math}
    When $v$ is real, the volume of this ellipsoid is $\mu
    _{v}(E_{v})= 2^{r}\prod_{i}\Vert b_{i}\Vert _{v}^{-d_{\K}n_{v}}$
    whereas, when $v$ is complex, it is $\mu _{v}(E_{v})=
    \pi^{r}\prod_{i}\Vert b_{i}\Vert
    _{v}^{-d_{\K}n_{v}}$. Lemma~\ref{lemm:1}\eqref{item:11} then
    implies
\begin{equation}\label{eq:17}
- {r}\log(r)\le \log(\mu _{v}(\phi_{b,v}(V^{\circ}_{v})))  - \sum_{i} d_{\K}n_{v}
\log(\|b_{i}\|_{v}^{-1}) \le  \frac{r\log(r)}{2}+r\log \pi. 
\end{equation}
Now let $v$ be a non-Archimedean place.  Let $\alpha_{v,i}\in
\K_{v}$ such that $\|b_{i}\|_{v}\le |\alpha_{v,i}|_{v}
<|\varpi_{v}|^{-1}_{v}\|b_{i}\|_{v}$ for $v\in S$ and $\alpha_{v,i}=1$
for $v\notin S$. By Lemma
\ref{lemm:1}\eqref{item:12}, the vectors
$b_{v,i}:=\alpha_{v,i}^{-1}b_{i}$, $i=1,\dots, r$, form a basis of
$V_{v}^{\circ}$ and $ |\det (b_{v}/b)|_{v} =
\prod_{i}|\alpha_{v,i}|_{v}^{-1}$. Hence, for $v\in S$, 
\begin{equation}\label{eq:16}
-r\lambda_{v} \le \log  |\det (b_{v}/b)|_{v} -
\sum_{i}\log(\Vert b_{i}\Vert _{v}^{-1}) \le 0,  
\end{equation}
whereas $\log  |\det (b_{v}/b)|_{v} = \sum_{i}\log(\Vert b_{i}\Vert _{v}^{-1}) =0$ for
$v\notin S$. Adding up \eqref{eq:17} and \eqref{eq:16} and using the
formula \eqref{eq:12}, we conclude
\begin{displaymath}
- d_{\K}\, r\Big(\log(r)+\sum_{v\in S}n_{v}\lambda_{v}\Big)\le  \wh \chi
(\ov V) - \sum_{v}\sum_{i}d_{\K}\, n_{v}\log
  (\|b_{i}\|_{v}^{-1}) \le d_{\K} \, r\log(\pi r).  
\end{displaymath}

The case of a  function field follows similarly by applying Lemma
\ref{lemm:1}\eqref{item:12}. The resulting upper bound is better, since it does not
have the term $\log(\pi r)$.
\end{proof}

\section{Metrized $\R$-divisors}
\label{sec:metrized-r-divisors}

In this section, we introduce the basic definitions concerning
metrized $\R$-divisors and the problems we are interested in. We
start by recalling the geometric analogues for $\R$-divisors.  Details
on the theory of $\R$-divisors can be found in
\cite{Lazarsfeld:posit_I}.

Let $K$ be a field and $X$ a proper normal variety over $K$ of
dimension $n$.  We denote by $\Car(X)$ and $\Div(X)$ the groups of
Cartier divisors and of Weil divisors of $X$, respectively. 
The spaces of {$\R$-divisors} and of $\R$-Weil divisors of $X$ are
 defined as
\begin{displaymath}
\Car(X)_\R=\Car(X)\otimes_\Z \R,\quad   \Div(X)_\R=\Div(X)\otimes_\Z \R.
\end{displaymath}
Thus, an $\R$-Cartier divisor on $X$ is a formal linear combination $\sum_{i}\alpha
_{i}D_{i}$ with $\alpha _{i}\in \R$ and $D_{i}\in \Car(X)$, and
similarly for an $\R$-Weil divisor.  In this text, we will be  mainly
concerned with Cartier divisors and $\R$-Cartier divisors and so we
will call them just divisors and $\R$-divisors, for short.

Since $X$ is normal, there is an injective morphism
$\Car(X)\hookrightarrow \Div(X)$.  Then, the fact that $\Div(X)$ is a
free Abelian group implies that the maps $\Car(X)\to \Car(X)_{\R}$ and
$\Car(X)_{\R}\to \Div(X)_{\R}$ are injective. Given an $\R$-divisor
$D$ on $X$, its \emph{support} is defined as the support of the associated $\R$-Weil
divisor and is denoted by $|D|$.

Let ${\rm K}(X)^{\times}$ be the multiplicative group of nonzero rational
functions of $X$ and set
${\rm K}(X)^{\times}_\R:={\rm K}(X)^{\times}\otimes_\Z\R$. The map $\div\colon
{\rm K}(X)^{\times}\to \Car(X)$ extends by linearity to a map ${\rm K}(X)^{\times}_{\R}\to
\Car(X)_{\R}$, also denoted by $\div$.

The spaces of nonzero \emph{global sections} and \emph{global
  $\R$-sections} of an $\R$-divisor $D$ on $X$ are respectively defined as
\begin{align*}
  \Gamma(X,D)^{\times}&=\{(f,D)\mid f \in {\rm K}(X)^{\times},\
  \div(f)+D\ge 0\},\\
  \Gamma(X,D)^{\times}_{\R}&=\{(f,D)\mid f \in {\rm K}(X)_{\R}^{\times},\
  \div(f)+D\ge 0\}.
\end{align*}
The spaces of nonzero \emph{rational
  sections} and \emph{rational $\R$-sections} of $D$ are
respectively defined as
\begin{displaymath}
  \Rat(X,D)^{\times}:={\rm K}(X)^{\times}\times\{D\},\quad 
\Rat(X,D)_{\R}^{\times}:={\rm K}(X)_{\R}^{\times}\times\{D\}.
\end{displaymath}
For $s=(f,D)\in \Rat(X,D)_{\R}^{\times}$, we write
\begin{displaymath}
  \div(s)=\div(f)+D\in \Car(X)_{\R}.
\end{displaymath}
Let $\LL(D)=\Gamma(X,D)^{\times}\cup\{(0,D)\}$ be the
\emph{Riemann-Roch space} of $D$. Note that, contrary to the usual
convention, we have added the label ``$D$" to the elements of
$\LL(D)$.  This is a $K$-vector space and we set
$\lL(D)=\dim_{K}(\LL(D))$ for its dimension.
The \emph{volume} of $D$ is
  defined as
  \begin{displaymath}
    \vol(X,D)=\limsup_{\ell\rightarrow\infty}\frac{\lL(\ell D)}{\ell^{n}/n!}.
  \end{displaymath}
 
Let $Y$ be a $d$-dimensional subvariety of $X$ and $D_{1},\dots,D_{d}$
a family of  $\R$-divisors on $X$.  The {intersection product}
$(D_{1}\cdots D_{d}\cdot Y)$ is defined by multilinearity from the
intersection product of $Y$ with divisors.  The \emph{degree} of $Y$
with respect to an $\R$-divisor~$D$ is $\deg_{D}(Y)=(D^{d}\cdot Y)$.

  An $\R$-divisor is \emph{ample} (respectively
  \emph{big}, \emph{effective})
  if it is a linear 
  combination of ample (respectively
  big, effective)  divisors with positive
  coefficients. Given divisors $D_{1}$ and  $D_{2}$, the condition that
  $D_{1}-D_{2}$ is effective is denoted by $D_{1}\ge D_{2}$. 
  An $\R$-divisor $D$ on $X$ is \emph{nef} if $\deg_{D}(C)\ge0$ for
  every curve $C\subset X_{\ov K}$. It is \emph{pseudo-effective} if
  there exists a  birational map $\varphi\colon X'\to X$ of
  normal proper varieties over~$K$ and a divisor
  $E$ on $X'$ such that $\ell \varphi^{*}D +E$ is big for all $\ell
  \ge1$.

\begin{rem} \label{rem:18}
  The definition of pseudo-effective given here differs from the one
  in \cite[Definition 2.2.25]{Lazarsfeld:posit_I} because  we are not assuming that the
  variety $X$ is projective. For projective normal varieties, both definitions
  are equivalent.
\end{rem}

We now introduce metrized divisors and metrized $\R$-divisors on
varieties over global fields. Throughout the rest of this section,
$X$ will be  a normal  proper variety over a
global field
$\K$ of dimension $n$.  For each place $v\in \mathfrak M_{\K}$, we
denote by $X_{v}^{\an}$  the $v$-adic analytification of $X$. In the
Archimedean case, if $\K_v\simeq \C$, then $X_{v}^{\an}$ is an
analytic space over $\C$, whereas if $\K_v\simeq \R$, then $X_{v}^{\an}$ is an
analytic space over $\R$, that is, an
analytic space over $\C$ together with an antilinear involution (see
for instance~\cite[Remark 1.1.5]{BurgosPhilipponSombra:agtvmmh}).
In the non-Archimedean case, $X_{v}^{\an}$ is a Berkovich
space (see \cite[\S
1.2]{BurgosPhilipponSombra:agtvmmh}). Similarly, a line bundle $L$ on
$X$ defines a collection of analytic line bundles
$\{L^{\an}_{v}\}_{v\in \mathfrak{M}_{\K}}$.

Following {\it loc. cit.}, all considered metrics will be continuous, by
definition.  For $v\in \mathfrak M_{\K}$, a metric $\|\cdot\|_{v}$ on
$L_{v}^{\an}$ is \emph{semipositive} if it is the uniform limit of a
sequence of semipositive smooth (respectively algebraic) metrics in
the Archimedean (respectively non-Archimedean) case.  The metric
$\|\cdot\|_{v}$ is \emph{DSP} if it is the quotient of two
semipositive ones. A metric on $L$ is, by definition, an adelic
collection of metrics $\|\cdot\|_{v}$ on $L_{v}^{\an}$, ${v\in
  \mathfrak M_{K}}$. Such a collection 
is \emph{quasi-algebraic} if there is an integral model which defines
the metric $\|\cdot\|_{v}$ except for a finite number of 
$v$. We refer to \cite[Chapter 1]{BurgosPhilipponSombra:agtvmmh} for
the precise definitions and
more details.

\begin{defn} \label{def:14} For $v\in \mathfrak{M}_{\K}$, a
  \emph{$v$-adically metrized divisor} on $X$ is a pair
  $\ov D=(D,\|\cdot\|)$ formed by a divisor and a metric on the
  analytic line bundle $\mathcal{O}(D)_{v}^{\an}$. We say that $\ov D$
  is \emph{smooth} (in the Archimedean case), \emph{algebraic} (in the
  non-Archimedean case), or \emph{semipositive} (smooth or algebraic)
  if so is the metric $\|\cdot\|$. The \emph{$v$-adic Green function}
  of $\ov D$ is the function
  $g_{\ov D}\colon X_{v}^{\an}\setminus |D|\to \R$ given by
  \begin{displaymath}
    g_{\ov D}(p)=-\log\|s_{D}(p)\|,
  \end{displaymath}
  where $s_{D}$ is the canonical section of $\mathcal{O}(D)$.  The
  space of all $v$-adically metrized divisors on $X$ is denoted by
  $\wh \Car(X)_{v}$.

  A \emph{quasi-algebraic metrized divisor} on $X$ is a pair
  $\ov D=(D,\{\|\cdot\|_{v}\}_{v\in \mathfrak{M}_{\K}})$ formed by a
  divisor $D$ and a quasi-algebraic metric on the line bundle
  $\mathcal{O}(D)$.  The space of quasi-algebraic metrized divisors on
  $X$ is denoted by $\wh\Car(X)$.  {For $\ov D \in \wh\Car(X)$ and
  $v\in \mathfrak{M}_{\K}$ we denote by $g_{\ov D,v}$ the
  \emph{$v$-adic Green function} of $\ov D$, defined as the Green
  function of the $v$-adically metrized divisor $(D,\|\cdot\|_{v})$.}\todo{This  introduces
  the (so far missing) notation $g_{\ov D,v}$ for a metrized divisor}
\end{defn}

\begin{defn} \label{def:1} Let $X$ be a proper normal variety over $\K$.  The
  \emph{space of quasi-algebraic metrized $\R$-divisors} on $X$ is the
  quotient
  \begin{displaymath}
    \wh\Car(X)_{\R}=\left. \wh \Car(X)\otimes_\Z \R \right/ \sim
  \end{displaymath}
  where $\sim$ is the equivalence relation given by $ \sum_{i}\alpha
  _{i}\ov D_{i}\sim \sum_{j}\beta _{j}\ov E_{j} $ if and only if $
  \sum_{i}\alpha _{i} D_{i}= \sum_{j}\beta _{j} E_{j} $ and,
  for each $v\in \mathfrak{M}_{\K}$, there is a dense open subset
  $U_{v}$ of $ X^{\an}_{v}$ such that
  \begin{displaymath}
      \sum_{i}\alpha _{i} g_{\ov D_{i},v}(p)= \sum_{j}\beta _{j}g_{\ov
        E_{j},v} (p)\quad \text{ for } p\in U_{v}.
  \end{displaymath}
   {The function $g_{\ov D,v}\colon U_{v}\to \R$ defined by
    $g_{\ov D,v}(p)=\sum_{i}\alpha _{i} g_{\ov D_{i},v}$} is called
  the \emph{$v$-adic Green function} of $\ov D$. todo{Idem for a
    metrized $\mathbb{R}$-divisor} An element of
  $\ov D\in \wh\Car(X)_{\R}$ is called a \emph{quasi-algebraic
    metrized $\R$-divisor} or, for short, a \emph{metrized
    $\R$-divisor}.

  For $v\in \mathfrak{M}_{\K}$, we can similarly define the space of
  $v$-adically metrized $\R$-divisors on $X$, denoted by $\wh \Car(X)_{v,\R}$. A
  $v$-adically metrized $\R$-divisor $\ov D$ is called \emph{smooth}
  or \emph{algebraic} if it can be written as $\ov D=\sum \alpha
  _{i}\ov D_{i}$ with $\ov D_{i}\in \wh \Car(X)_{v}$ smooth or
  algebraic, respectively. In each of these cases, it is called
 \emph{semipositive}   (smooth or algebraic) if it can
  be written as $\ov D=\sum \alpha _{i}\ov D_{i}$, with $\alpha
  _{i}>0$ and $\ov D_{i}\in \wh \Car(X)_{v}$ semipositive (smooth or algebraic).
\end{defn}

\begin{defn}\label{def:20}
A rational function $f\in {\rm K}(X)^{\times}$ defines a metrized divisor
\begin{displaymath}
\wh \div(f)=(\div(f),\{\|\cdot\|_{f,v}\}_{v\in \mathfrak{M}_{\K}}),  
\end{displaymath}
where $\{\|\cdot\|_{f,v}\}_{v}$ is the metric on the line bundle
$\cO(\div(f))_{v}^{\an}$ which is given by
$\|f^{-1}(p)\|_{f,v}=1$ for the section $f^{-1}\in
\cO(\div(f))_{v}^{\an}$ and  $p\in X_{v}^{\an}$. 
This construction defines a group morphism from $\wh \div\colon
  {\rm K}(X)^{\times}\to \wh \Car(X)$, which extends by linearity to a
  group morphism
\begin{displaymath}
  \wh \div \colon  {\rm K}(X)^{\times}_{\R}\longrightarrow \wh \Car(X)_{\R}.
\end{displaymath}
Let $\ov D$ and $\ov {D'}$ be metrized $\R$-divisors. We say that they
are \emph{linearly equivalent} if there exists $f\in
{\rm K}(X)^{\times}_{\R}$ such that
\begin{displaymath}
  \ov D-\ov {D'}=\wh \div(f).
\end{displaymath}
\end{defn}

For $f\in {\rm K}(X)^{\times}_{\R}$ and  $v\in 
\mathfrak{M}_{\K}$, the $v$-adic Green function of $\wh \div (f)$ is
the function given by $g_{\wh \div(f),v}(p)=-\log |f(p)|_{v}$.

\begin{lem}\label{lemm:3} Let $v\in \mathfrak{M}_{\K}$ and $(D,\|\cdot\|)$ 
  a $v$-adically metrized $\R$-divisor on $X$. Let $r\ge1$ and
  consider a decomposition $D=\sum_{i=1}^{r} \alpha _{i}D_{i}$ with
  $\alpha _{i}\in \R$ and $D_{i}\in \Car(X)$. Then there are metrics
  $\|\cdot\|_{i}$ on $\mathcal{O}(D_{i})_{v}^{\an}$, $i=1,\dots, r$, 
  such that
  \begin{displaymath}
 (D,\|\cdot\|)=\sum_{i=1}^{r} \alpha _{i}\, (D_{i},\|\cdot\|_{i}).
  \end{displaymath}
\end{lem}
\begin{proof}
  Set $\ov D= (D,\|\cdot\|)$ and let $\ov D=\sum_{j=1}^{l}\beta _{j}\ov E_{j}$ be a decomposition
  with $\beta _{j}\in \R$ and $\ov E_{j}\in \wh \Car(X)_{v}$.  We
  first consider the case when the decomposition is $D=0$. 

  Let $\{\gamma _{1},\dots,\gamma _{r}\}$ be a basis of the
  $\Q$-vector space generated by the numbers $\beta _{1},\dots,\beta
  _{l}$. Write $\beta _{j}=\frac{1}{m}\sum_{i=1}^{s}n_{i,j}\gamma_i$
  with $n_{i,j}\in \Z$ and $m\in \Z_{\ge1}$. Then
  \begin{displaymath}
    \ov
    D=\frac{1}{m}\sum_{i=1}^{s}\gamma _{i}\ov F_{i}     
  \end{displaymath}
  with $\ov F_{i}=\sum_{j}n_{i,j}\ov E_{j}\in \wh \Car(X)_{v}$. Since
   $D=0$ and $\gamma _{1},\dots,\gamma _{s}$ are
  linearly independent over $\Q$, we deduce that $F_{i}=0$ for all $i$.  Since
  $\mathcal{O}(0)^{\an}_{v}=\mathcal{O}_{X^{\an}_{v}}$, to give a
  metric $\|\cdot\|$ on this line bundle is equivalent
  to give the continuous function $\|1\|\colon X^{\an}_{v}\to
  \R_{>0}$. Therefore, we can gather together the different metrics $\|\cdot\|_{F_{i}}$
  on $\mathcal{O}(F_{i})^{\an}_{v}$ with their
  coefficients, into a 
  single metric on $\mathcal{O}_{X^{\an}_{v}}$ given by
  \begin{displaymath}
    \|1\|=\prod_{i=1}^{r}\|1\|_{F_i}^{\gamma _{i}/m}.
  \end{displaymath}
Then $\ov D= (0, \|\cdot\|)$,  which finishes the proof in this case.
  
  We now consider the general case. Choose metrics $\|\cdot\|_{1}'$ on
  $\mathcal{O}(D_{1})^{\an}_{v}$ and $\|\cdot\|_{i}$ on
  $\mathcal{O}(D_{i})^{\an}_{v}$, $i=2,\dots, r$, and denote by $\ov D_{1}'$
  and $\ov D_{i}$ the corresponding $v$-adically metrized divisors. By
  applying the previous case to the metrized $\R$-divisor $\sum_{j}\beta _{j}\ov
  E_{j}-\alpha _{1}\ov D'_{1} -\sum_{i=2}^{r}\alpha_{i}\ov D_{i}$, we
  obtain a metric $\|\cdot\|_{0}$ on the trivial line bundle.  Finally,
  we define a new metric on $\mathcal{O}(D_{1})^{\an}_{v}$ by
  $\|\cdot\|_{1}=\|1\|_{0}^{1/\alpha _{1}}\|\cdot\|'_{1}$ and we set $\ov
  D_{1}=(D_{1},\|\cdot\|_{1})$. Hence 
$ \ov D=\sum_{i=1}^{r} \alpha _{i} \ov D_{i}$, proving the result.
\end{proof}

  Let $\ov D$ be a $v$-adically metrized $\R$-divisor on $X$ and
  $s=(f,D)$ a rational $\R$-section of $D$. Consider the function $\|s\|$ given by
  \begin{displaymath}
    \|s(p)\|=|f(p)|_{v}\e^{-g_{\ov D}(p)}, 
  \end{displaymath}
  where, if $f=\prod_{i} f_{i}^{\alpha _{i}}$, then $|f(p)|_{v}$ is
  defined as $\prod_{i} |f_{i}(p)|^{\alpha _{i}}_{v}$.  This function
  is well-defined on a dense open subset of $X_{v}^{\an}$. The
  following result shows that it can be extended everywhere outside the
  support of $\div(s)$.

\begin{prop} \label{prop:5}
  Let $v\in \mathfrak{M}_{\K}$ and  $\ov D$  a $v$-adically metrized
  $\R$-divisor on $X$. Let   $s=(f,D)\in \Rat(X,D)^{\times}_{\R}$. Then  $\|s\|$
  can be extended to a continuous function from $X^{\an}_{v}\setminus
  |\div(s)|$ to $\R_{>0}$. 

In particular, the $v$-adic Green function $g_{\ov
    D}$ can be extended to a continuous function from
  $X^{\an}_{v}\setminus |D|$ to $\R$. \end{prop}

\begin{proof} Write $\ov D=\sum_{i=1}^{r}\alpha_{i}\ov D_{i}$ and
  $f=\prod_{j=1}^{k}f_{j}^{\beta_{j}}$ with
  $\alpha_{i},\beta_{j}\in \R$, $\ov D_{i}\in \wh\Car(X)_{v}$ and
  $f_{j}\in  {\rm K}(X)^{\times}$. Let $\{\gamma_{l}\}_{l}$ be a basis of the
  $\Q$-linear space generated by the numbers
  $\alpha_{1},\dots,\alpha_{r},\beta_{1},\dots, \beta_{k}$ and write
  \begin{displaymath}
    D=\sum_{{l}}\frac{\gamma_{l}}{m}E_{l},\quad f=\prod_{l}g_{l}^{
    \gamma_{l}/m}
  \end{displaymath}
  with $m\in\Z_{\ge1}$, $E_{l}\in \Car(X)$ and $g_{l}\in
  {\rm K}(X)^{\times}$. The pair
$s_{l}=(g_{l},E_{l})$ defines a rational section of the Cartier
divisor $E_{l  }$. We have $|\div(s)|= \bigcup_{l}|\div(s_{l})|$
  because the $\gamma_{l}$ are $\Q$-linearly independent.

By Lemma \ref{lemm:3}, there are $v$-adic metrics
 $\Vert\cdot\Vert_{l}$ on each $E_{l}$ such that 
\begin{displaymath}
  \ov D=\sum_{l}\frac{\gamma_{l}}{m}\ov E_{l}
\end{displaymath}
with $\ov E_{l}=(E_{l},\Vert\cdot\Vert_{l})$. The function
$\Vert s_{l}\Vert_{l}$ is continuous from $X^{\an}_{v}\setminus
|\div(s_{l})|$ to $\R_{>0}$ and there is an open dense subset $U_{v}\subset
X_{v}^{\an}$ such that, for $p\in U_{v}$,
\begin{displaymath}
\Vert s(p)\Vert = \prod_{l}\Vert s_{l}(p)\Vert_{l}^{\gamma_{l}/m}.
\end{displaymath}
This latter function is well-defined outside
$\bigcup_{l}|\div(s_{l})|=|\div(s)|$ and gives the sought extension of
$\|s\|$. The second statement follows from the
first one applied to the canonical section $s_{D}=(1,D)$.
\end{proof}

The definition below introduces the notion of $v$-adic metric of an
$\R$-divisor directly, without passing through $v$-adic metrics on
divisors as in Definition \ref{def:1}. The following proposition
shows that both points of view are equivalent.

\begin{defn}\label{def:26} Let
  $D$ be an $\R$-divisor on $X$ and  $v\in \mathfrak{M}_{\K}$.
  \begin{enumerate}
  \item \label{item:29} A \emph{$v$-adic Green function} on $D$ is a continuous
    function $g\colon X^{\an}_{v}\setminus |D|\to \R$ such that, for
    each point $p \in X_{v}^{\an}$ and a rational $\R$-section $s=(f,D)$
    with $p\not \in |\div(s)|$, the function $g(p)-\log|f(p)|_{v}$ can
    be extended to a continuous function in a neighbourhood of $p$.
  \item \label{item:30} A \emph{$v$-adic metric} on $D$ is
  an assignment that, to each rational $\R$-section~$s$ associates a
  continuous function $\|s\|\colon X^{\an}_{v}\setminus |\div (s)|\to
  \R_{>0}$ such that, for each $\gamma  \in  {\rm K}(X)^{\times}_{\R}$, it
  holds $\|(\gamma s)(p)\|=|\gamma (p)|_{v} \|s(p)\|$ on an open
  dense subset of $X^{\an}_{v}$.
  \end{enumerate}
\end{defn}

\begin{prop} \label{prop:24} Let $D$ be an $\R$-divisor on $X$ and $v\in
  \mathfrak{M}_{\K}$.  It is equivalent to choose a $v$-adically
  metrized $\R$-divisor $\ov D$ over $D$, to choose a $v$-adic Green
  function for $D$, or to choose a $v$-adic metric on $D$.
\end{prop}
\begin{proof}
  Let $\ov D$ be a $v$-adically metrized $\R$-divisor over $D$. The
  fact that $g_{\ov D}$ is a $v$-adic Green function for $D$ in the
  sense of Definition \ref{def:26}\eqref{item:29} follows from
  Proposition~\ref{prop:5}.

  Now, if $g$ is a $v$-adic Green function for $D$, the assignment
  that to each $s=(f,D)\in \Rat(X,D)^{\times}_{\R}$ associates the
  function
  \begin{displaymath}
    \Vert s(p)\Vert_{v}=|f(p)|_{v}\e^{-g(p)}   
  \end{displaymath}
  is a metric on $D$ in the sense of Definition \ref{def:26}\eqref{item:30}

  Let now $\|\cdot\|$ be a metric on $D$. Choose a decomposition
  $D=\sum_{i} \alpha _{i}D_{i}$ with $\alpha _{i}\in \R$ and
  $D_{i}\in\Car(X)$. For each $i$, choose a metric $\|\cdot\|_{i}$ on
  $\mathcal{O}(D_{i})$. Let $f\in {\rm K}(X)^{\times}$ and
  $s_{i}=(f_{i},D_{i})\in \Rat(X,D_{i})^{\times}_{\R}$ and 
  $s=(f\prod_{i} f_{i}^{\alpha _{i}},D)\in \Rat(X,D)^{\times}_{\R}$. It
  is easy to check that the formula
  \begin{displaymath}
 \|(f,0)\|=\|s\|\cdot \prod_{i} \|s_{i}\|^{-\alpha
    _{i}} 
  \end{displaymath}
  defines a metric on the trivial line bundle
  $\mathcal{O}_{X^{\an}_{v}}$. Denote by $\ov 0$ the zero  divisor equipped
  with this metric. Then $\ov D=\sum \alpha _{i}\ov D_{i}+\ov 0$ is a
  $v$-adically metrized $\R$-divisor in the sense of Definition \ref{def:1}.

  If we apply cyclically this three constructions starting at
  any point, we obtain the identity, showing the equivalence of the
  three points of view.
\end{proof}

\begin{rem} \label{rem:14}
Since the points of view of metrized $\R$-divisors, Green functions
and metrics are equivalent, we will apply the introduced terminology
to any of them. For instance, it makes sense to talk of algebraic or
smooth  metrics on an $\R$-divisor. 
\end{rem}

Let $D$ be an $\R$-divisor on $X$ and $v\in \mathfrak M_{\K}$. There
is a natural action of $\cC^{0}(X_{v}^{\an},\R)$, the space of
real-valued continuous functions on $X_{v}^{\an}$, on the space of
$v$-adic metrics on $D$. This action is given, for $f\in
\cC^{0}(X_{v}^{\an},\R) $, by the formula $\Vert\cdot\Vert\mapsto
\e^{f}\Vert\cdot \Vert$.  By Proposition \ref{prop:24}, this action
might be equivalently written in terms of $v$-adic Green functions as
\begin{math}
  g_{\ov D}\mapsto g_{\ov D}-f.
\end{math}
From this description, the following corollary follows immediately. 

\begin{cor} \label{cor:7} Let $D\in \Car(X)_{\R}$ and $v\in
  \mathfrak M_{\K}$.  The action of $\cC^{0}(X_{v}^{\an},\R)$ on the
  space of $v$-adic metrics on $D$ is free and transitive.
\end{cor}

Given two $v$-adic metrics $\Vert\cdot\Vert$ and  $\Vert\cdot\Vert'$
on $D$, the difference of their Green functions
\begin{math}
  g_{D,\Vert\cdot\Vert} -  g_{D,\Vert\cdot\Vert'}
\end{math}
extends to a continuous function on $X^{\an}_{v}$. The \emph{distance}
between  $\Vert\cdot\Vert$ and  $\Vert\cdot\Vert'$ is defined as 
\begin{displaymath}
  \dist(\Vert\cdot\Vert,\Vert\cdot\Vert')= \sup_{p\in X_{v}^{\an}} \big|g_{D,\Vert\cdot\Vert}(p) -  g_{D,\Vert\cdot\Vert'}(p)\big|.
\end{displaymath}

\begin{cor}\label{cor:8}
  Let $D\in \Car(X)_{\R}$ and $v\in \mathfrak M_{\K}$. The space of
  $v$-adic metrics on~$D$ is complete with respect to the topology
  defined by $\dist$.
\end{cor}

\begin{proof}
  This follows immediately from Corollary \ref{cor:7} and the fact
  that the space  $\cC^{0}(X_{v}^{\an},\R)$ is complete with respect
  to the uniform convergence of functions. 
\end{proof}

\begin{defn} \label{def:24} Let $D$ be an $\R$-divisor on $X$. For
  $v\in \mathfrak M_{\K}$, a $v$-adic metric $\Vert\cdot\Vert$ on~$D$
  is \emph{semipositive} if it is the limit of a sequence of
  semipositive smooth
  (in the Archimedean case) or algebraic (in the non-Archimedean case)
 metrics on~$D$.

 A quasi-algebraic metric
 $\{\Vert\cdot\Vert_{v}\}_{v\in \mathfrak M_{\K}}$ on $D$ is
 \emph{semipositive} if $\Vert\cdot\Vert_{v}$ is {a semipositive}
 $v$-adic metric for all $v$. A quasi-algebraic metrized $\R$-divisor
 is \emph{DSP} if there are semipositive metrized $\R$-divisors
 $\ov D_{1},\ov D_{2}$ on $X$ such that $\ov D= \ov D_{1}-\ov D_{2}$.
\end{defn}

Let $Y$ be a $d$-dimensional cycle on $X$ and $v\in \mathfrak M_{\K}$.
Let 
$\ov D_{i}$, $i=0,\dots, d$, be semipositive $v$-adically metrized
$\R$-divisors on~$X$ such that $D_{0},\dots, D_{d}$ meet $Y$ properly
in the sense of \cite[Definition
 1.4.9]{BurgosPhilipponSombra:agtvmmh}, that is, for all $I\subset \{0,\dots,d\}$, every component
of the intersection 
 \begin{displaymath}
   Y\cap \bigcap _{i\in I}|D_{i}|
 \end{displaymath}
 is of dimension $d-\# I$.  The \emph{$v$-adic height} of $Y$ with
 respect to $\ov D_{0},\dots,\ov D_{d}$, denoted $\h_{v,\ov
   D_{0},\dots,\ov D_{d}}(Y)$, is defined by multilinearity and
 continuity from the local height of a
 cycle with respect to semipositive line bundles \cite[Definition
 1.4.11]{BurgosPhilipponSombra:agtvmmh}. These $v$-adic heights satisfy
 the \emph{B\'ezout formula}
\begin{multline} \label{eq:20}
  \h_{v,\ov D_{0},\dots,\ov
  D_{d}}(Y)= \h_{v,\ov D_{0},\dots,\ov
  D_{d-1}}(Y\cdot D_{d}) \\-\int_{X_{v}^{\an}}\log\Vert s_{D_{d}}\Vert_v
\chern_{1}(\ov D_{0})\wedge \dots\wedge \chern_{1}(\ov D_{d-1}) \wedge
\delta_{Y}, 
\end{multline}
where $\chern_{1}(\ov D_{0})\wedge \dots\wedge \chern_{1}(\ov D_{d-1}) \wedge
\delta_{Y}$ is the measure on $X^{\an}_{v}$ defined by multilinearity
and continuity from the corresponding measures associated to semipositive 
(smooth or algebraic) metrized line bundles. 

For semipositive metrized $\R$-divisors $\ov D_{i}\in \Car(X)_{\R}$,
$i=0,\dots, d$, the \emph{global height} of $Y$ with respect
to $\ov D_{0},\dots,\ov D_{d}$, denoted by $\h_{\ov
  D_{0},\dots,\ov D_{d}}(Y)$, can be defined from local heights
similarly as in Definition 1.5.9
of {\it loc. cit.}. If $D_{0},\dots, D_{d}$ meet $Y$ properly, then
\begin{displaymath}
  \h_{\ov D_{0},\dots,\ov
  D_{d}}(Y)= \sum_{v\in \mathfrak M_{\K}} n_{v }\h_{v,\ov D_{0},\dots,\ov
  D_{d}}(Y).
\end{displaymath}
Finally, the notions of local and global heights of cycles extend to DSP
metrized $\R$-divisors by multilinearity. 

Next, we define arithmetic linear series, arithmetic volume and
$\chi$-arithmetic volume of a metrized $\R$-divisor $\ov D$ on $X$, extending to our
framework the previous definitions by Moriwaki and Yuan. Let $s\in
\Gamma(X,D)^{\times}$ be a nonzero global section of $D$. For each place
$v$, the function $\Vert s\Vert_{v}\colon X^{\an}_{v}\setminus
|\div(s)|\to \R_{>0}$ can be extended to a continuous function
$X^{\an}_{v}\to \R_{\ge0}$ because $\div(s)$ is effective.  We set
\begin{displaymath}
  \|s\|_{v,\sup}=\sup_{p\in X_{v}^{\an}}\|s(p)\|_{v}.
\end{displaymath}
This induces a $v$-adic norm on the $\K_{v}$-vector space
$\LL(D)\otimes \K_{v}$. The collection of norms
$\{\|\cdot\|_{v,\sup}\}_{v}$ gives a structure of generically trivial
adelic vector space on $\LL(D)$, that we denote by $\LL(\ov D)$. The
small (respectively strictly small, strictly small on a set of places)
elements of $\LL(\ov D)$ will be called \emph{small} (respectively
\emph{strictly small}, \emph{strictly small on a set of places})
sections.  We write $\whL(\ov D)=\wh H^{0}(\LL(\ov D))$ for the set of
small sections of $\LL(\ov D)$ and we also set $ \whl(\ov D) =\wh h^{0}(\LL(\ov
D))$ (Definition~\ref{def:12}). Hence,
  \begin{displaymath}
    \whl(\ov D) =
  \begin{cases}
    \log (\# \whL(\ov D))&\text{ if $\K$ is a number field,} \\
    \log(c_{k}) \dim_{k} (\whL(\ov D)) &\text{ if $\K$ is a
      function field.}
  \end{cases}
  \end{displaymath}

  \begin{defn} \label{def:3}
The \emph{arithmetic volume} and the \emph{$\chi$-arithmetic volume}
of a metrized $\R$-divisor $\ov D$ on $X$ are respectively defined as 
  \begin{align*} 
    \avol(X,\ov D)&=\frac{1}{d_{\K}}\limsup_{\ell\to \infty} \frac{
      \whl(\ell \ov D) }{\ell^{n+1}/(n+1)!},\\
  \avol_{\chi}(X,\ov D)&=\frac{1}{d_{\K}}\limsup_{\ell\to \infty} \frac{
      \wh\chi(\LL(\ell \ov D)) }{\ell^{n+1}/(n+1)!}. 
  \end{align*}
  \end{defn}

  \begin{rem}\label{rem:9}
    For metrized divisors on a variety over a number field, the $\limsup$ in the
    above definition is actually a limit \cite[Theorems 2.8 and
    3.1]{BoucksomChen:Obfs}.  We will not use this fact in this text.
  \end{rem}

  By extension, a global $\R$-section $s\in \Gamma(X,D)^{\times}_{\R}$
  is called \emph{small} if $\|s\|_{v,\sup}\le 1$ for all $v\in
  \mathfrak{M}_{\K}$. The set of small elements of
  $\Gamma(X,D)^{\times}_{\R}$ is denoted by $\wh \Gamma(X,\ov
  D)^{\times}_{\R}$. 

  A small $\R$-section $s$ of $\ov D$ is called
  \emph{strictly small} if $\prod_{v}\|s\|_{v,\sup}^{n_{v}}< 1$. If
  $S\subset\mathfrak{M}_{\K}$ is a finite set of places, the $\R$-section $s$
  is \emph{strictly small on} $S$ if $\|s\|_{v,\sup}< 1$ for all $v\in
  S$.

  Lemma~\ref{lemm:10}  has the following useful consequence for the
  existence of small $\R$-sections which are strictly small on a given
  set of places.

\begin{prop}\label{prop:2}
  Let $\ov D$  be a metrized $\R$-divisor on $X$ and  $s\in 
  \Gamma(X,D)^{\times}_{\R}$ a global $\R$-section such that
  \begin{displaymath}
    \prod_{v\in \mathfrak{M}_{\K}}\|s\|_{v,\sup}^{n_{v}}< 1.
  \end{displaymath}
  Let $S\subset \mathfrak{M}_{\K}$ be a finite subset and $0<\eta\le
  1$ a real number. Then there is an integer $\ell_{0}\ge 1$ such that,
  for each $\ell\ge \ell_{0}$, there exists $\alpha_{\ell} \in
  \K^{\times}$ such that $\|\alpha_{\ell} s^{\ell}\|_{v,\sup}\le 1$
  for all $v\in \mathfrak{M}_{\K}$ and
  \begin{math}
    \|\alpha_{\ell} s^{\ell}\|_{v,\sup}< \eta
  \end{math}
  for all $v\in S$. 

In particular, $\alpha _{\ell}s^{\ell}\in
  \wh\Gamma(X,\ell \ov D)^{\times}_{\R}$ is strictly small on $S$.
\end{prop}

\begin{proof}
This follows by applying Lemma \ref{lemm:10} to the real numbers $\gamma
  _{v}=\|s\|_{v,\sup}$. 
\end{proof}

\begin{lem} \label{lemm:5} Let $\ov D$ be a metrized divisor on $X$. If $\ov D$ has
  a strictly small $\R$-section $s$, then there exists a positive
  integer $\ell$ such that $\ell \ov D$ has a strictly small section
  $s_{0}$ with $|\div(s_{0})|= |\div(s)|$.     
\end{lem}

\begin{proof} Write $s=(f,D)$ with
  $f=\prod_{i=1}^{d}f_{i}^{\alpha_{i}}$ where $\alpha_{1},\dots, \alpha_{d}$
  are $\Q$-linearly independent real numbers and $\alpha_{1} \in \Q$. 
  The associated $\R$-Weil divisor can be written down as
  $[\div(s)]=\sum_{j}r_{j}Z_{j}$  with $r_{j}>0$ and $Z_{j}$ an
  irreducible hypersurface. Then
  \begin{displaymath}
    |\div(s)|= \bigcup_{j}Z_{j}= |\alpha_{1}\div(f_{1})+D|\cup\bigcup_{i=2}^{d}|\div(f_{i})|.
  \end{displaymath}
  
  Let $\sigma\colon \R^{d-1}\to \Rat(X,D)_{\R}^{\times}$ be the map defined by 
  \begin{displaymath}
    \sigma(\beta_{2},\dots, \beta_{d})= \Big(
    f_{1}^{\alpha_{1}}\cdot\prod_{i=2}^{d}f_{i}^{\beta_{j}}, D\Big).
  \end{displaymath}
  Then, for each $\beta\in \R^{d-1}$,  $[\div(\sigma(\beta))]=
  \sum_{j}r_{j}(\beta) Z_{j}$ where $r_{j}$ is an affine function. Consider
  the open polyhedral set
  \begin{displaymath}
    \Lambda=\{ \beta\in \R^{d-1}\mid r_{j}(\beta) >0 \text{ for all }
    j\}.
  \end{displaymath}
  This set is nonempty because $(\alpha _{2},\dots,\alpha
  _{d})\in\Lambda $. Let $ \tau\colon\Lambda\to \R$ be the function
  defined by
  \begin{displaymath}
    \tau(\beta)= \prod_{v}\Vert\sigma(\beta)\Vert_{\sup,v}^{n_{v}}.
  \end{displaymath}
For $v\in \mathfrak M_{\K}$, $\beta,\beta'\in \Lambda$
and $0\le \theta\le 1$,  
\begin{displaymath}
  \Vert\sigma(\theta\beta+(1-\theta)\beta')\Vert_{\sup,v}\le
  \Vert\sigma(\beta)\Vert_{\sup,v}^{\theta}
  \Vert\sigma(\beta')\Vert_{\sup,v}^{1-\theta}.
\end{displaymath}
From this, we deduce that $\tau$ is log-concave on $\Lambda$ and, a
fortiori, continuous. Since $\tau(\alpha_{2},\dots,\alpha_{d}) <1$, there exists $\beta\in \Lambda\cap
\Q^{d-1}$ such that $\tau(\beta)<1$. By Proposition~\ref{prop:2},
there exists $\gamma\in  \K^{\times}$ and a positive integer $\ell$
such that $s_{0}=\gamma\, \sigma(\beta)^{\ell}$ is a strictly small section of
$\ell \ov D$.  By construction, 
\begin{displaymath}
  |\div(s_{0})|= \bigcup_{j}Z_{j}= |\div(s)|,
\end{displaymath}
as stated.
  \end{proof}

We introduce the different notions of arithmetic positivity for a
metrized $\R$-divisor. 

  \begin{defn} \label{def:22} Let $X$ be a proper normal variety over $\K$
    and $\ov D$ a metrized $\R$-divisor on $X$.
  \begin{enumerate}
  \item \label{item:66} $\ov D$ is \emph{generated by small
      $\R$-sections} if, for each $p\in X({\ov K})$, there exists
    $s=(f,D)\in\wh \Gamma(X,\ov D)^{\times}_{\R}$ such that $p\not \in
    |\div (s)|$.
  \item \label{item:41} $\ov D$ is  \emph{ample} if the following
    conditions hold: 
  \begin{enumerate}
  \item \label{item:60} the $\R$-divisor $D$ is ample;
  \item \label{item:61} the metric is semipositive;
  \item \label{item:62} for each $\ov M\in \wh \Car(X)_{\R}$ there exists
    an  $\ell_{0}$ such that, for all real numbers $\ell\ge \ell_{0}$,
    the metrized $\R$-divisor $\ov M+\ell \ov D$ is generated by small
    $\R$-sections.
  \end{enumerate}
  \item \label{item:42} $\ov D$ is \emph{nef} if the following conditions hold:
  \begin{enumerate}
  \item \label{item:63} the $\R$-divisor $D$ is nef;
  \item \label{item:64} the metric is semipositive;
  \item \label{item:65} for every  point $p \in X({\ov\K})$ it holds $h_{\ov D}(p)\ge 0$.
  \end{enumerate}
  \item \label{item:57} $\ov D$ is \emph{big} if
  $\avol(X,\ov D)> 0.$
\item \label{item:58} $\ov D$ is \emph{pseudo-effective} if there
  exists a birational map $\varphi\colon X'\to X$ of normal  proper
  varieties over $K$ and a
  metrized $\R$-divisor $\ov E$ on $X'$ such that $ \ell \varphi
  ^{\ast}\ov D +\ov E$ is big for all $\ell \ge 1$.
\item \label{item:59} $\ov D$ is \emph{effective} if $(1,D)\in
  \whL(\ov D)$. Given $\R$-divisors $\ov D_{1},\ov D_{2}$ on $X$, the
  fact that $\ov D_{1}-\ov D_{2}$ is effective is denoted by $\ov
  D_{1}\ge \ov D_{2}$.
  \end{enumerate}
\end{defn}

The notion of metrized $\R$-divisor contains that of arithmetic
$\R$-divisor introduced by Moriwaki \cite{Moriwaki:zdas}.

\begin{exmpl}\label{exm:3} 
  Let $\cX$ be normal projective flat scheme over $\Z$ with
  smooth generic fiber $X=\cX\times \Spec (\Q)$.  An \emph{arithmetic
    $\R$-Cartier divisor} on $\cX$ is a pair $\ov {\cD}=(\cD,g)$
  where $\cD$ is an $\R$-Cartier divisor on $\cX$ and $g$ is a locally
  integrable real function on $X(\C)$ that is invariant under complex
  conjugation. Let $D$ be the restriction of $\cD$ to $X$. The
  function $g$ is called a \emph{Green function for $D$ of $C^{0}$-type}
  (respectively \emph{$C^{\infty}$-type, PSH-type}) if, for every
  $p\in X(\C)$, it can be locally written as
  \begin{displaymath}
    g(x)=u(x)+\sum_{i=1}^{k} (-\alpha_{i})\log|f_{i}(x)|^{2},
  \end{displaymath}
  where $D=\sum_{i=1}^{k} \alpha_{i} D_{i}$ is the decomposition of $D$ into
  irreducible components, $f_{i}$ is a local equation
  for $D_{i}$, and $u$ is a continuous function
  (respectively a smooth function, a plurisubharmonic
  function). If $g$ is a Green function of $C^{0}$-type (respectively
  $C^{\infty}$-type, PSH-type), then $\ov D=(\cD,g)$ is called an
  \emph{arithmetic divisor of $C^{0}$-type} (respectively
  \emph{$C^{\infty}$-type, PSH-type}).

  On the one hand, extending the method of Zhang \cite{Zhang:_small}
  (see also \cite[\S 1.3]{BurgosPhilipponSombra:agtvmmh}) to
  $\R$-Cartier divisors, the
  integral model $\cD$ defines an algebraic $v$-adic metric on
  $D$ for each  place $v\in \mathfrak{M}_{\Q}\setminus \{\infty\}$. On the
  other hand, if $g$
  is a Green function for~$D$ of $C^{0}$-type, then the function
  $\frac{1}{2}g$ is an $\infty$-adic Green function for $D$ which, by
  Proposition \ref{prop:24}, induces an $\infty$-adic metric on
  $D$, The factor $1/2$ comes from the different normalization of
  Green functions in \cite{GilletSoule:ait} and~\cite{Burgos:CDB}.
  
  Therefore, each arithmetic $\R$-divisor $\ov \cD$ of $C^{0}$-type
  defines a quasi-algebraic metrized $\R$-divisor on $X$ that we denote by
  $\ov D$. The metrized $\R$-divisors that arise in this way are
  called \emph{algebraic}. 
\end{exmpl}

We now discuss the relationship between the
  different  {notions of  positivity} for arithmetic
  $\R$-divisors that appear in \cite{Moriwaki:zdas} and
  those in the present text.  Let $\ov \cD=(\ov D,g)$ be an arithmetic
  $\R$-divisor of $C^{0}$-type and $\ov D$  its associated
  algebraic metrized $\R$-divisor.
\begin{enumerate}
\item $\ov \cD$ is of $PSH$-type if and only  {if} the $\infty$-adic
  metric induced by its Green function is semipositive.
\item \label{item:1} If the $\R$-divisor $\cD$ is relatively nef, then
  the  induced $v$-adic metrics are semipositive  for all
  $v\not=\infty$. The
  converse of this result is not established yet, except in the
  case of  equal characteristic zero \cite[Theorem 2.17]{boucksom:snama}.
 
\item $\ov \cD$ is effective (respectively big, pseudo-effective) in
  the sense of \cite{Moriwaki:zdas} if and only if $\ov D$ is
  effective (respectively big, pseudo-effective) in the sense of
  Definition \ref{def:22}.
\item If $\ov \cD$ is nef in the sense of \cite[\S 6.1]{Moriwaki:zdas}
  then $\ov D$ is nef in the sense of Definition \ref{def:22}.  The
  converse is not known because the converse of \eqref{item:1} is not
  known.
\item If $\ov \cD$ is of $C^{\infty}$-type, then it is ample in the
  sense of \cite[\S 
  6.1]{Moriwaki:zdas} if and only if $\ov D$ is ample in the sense of the
  present text. However, 
  the definition of ampleness in \emph{loc. cit.} 
  includes being of $C^{\infty}$-type, while in the present text an
  ample metrized $\R$-divisor is not necessarily of
  $C^{\infty}$-type. 
\end{enumerate}

The following statements contain some  of the  basic properties of
metrized $\R$-divisors.

\begin{prop} \label{prop:6}
  Let $\ov D$ be an ample metrized $\R$-divisor on $X$. Then
  \begin{enumerate}
  \item \label{item:16} $ \ov D$ is generated by strictly small
    $\R$-sections;
  \item \label{item:75} for all subvarieties $Y$ of $X_{\ov\K}$, it holds $\h_{\ov D}(Y)>0$.
  \end{enumerate}
\end{prop}
  
\begin{proof}
  \eqref{item:16} Let $v_{0}\in \mathfrak{M}_{\K}$ and consider the
  trivial divisor $0\in \Car(X)$ with the metric defined by $\|1\|_{v}=1$ for $v\not = v_{0}$ and
  $\|1\|_{v_{0}}=2$. By the ampleness of $\ov D$, there exists $\ell \ge 0$ such that
  $\ov 0+\ell \ov D$ is generated by small $\R$-sections. These
  small $\R$-sections are  strictly small $\R$-sections of $\ell\ov
  D$, which proves the statement. 

  \eqref{item:75} We prove this by induction on $d=\dim(Y)$. Consider
  first the case $d=0$. Let $Y=\{p\}$ where $p$ is a
  point defined over a finite extension $\F$ of
  $\K$. By~\eqref{item:16}, there is a strictly small $\R$-section $s$
  of $\ov D$ such that $p\notin |\div(s)|$. Then
\begin{displaymath}
  \h_{\ov D}(p)=-\sum_{w\in \mathfrak{M}_\F}n_{w} \log \Vert s(p)\Vert_{w}>0. 
\end{displaymath}
Assume now that $Y$ is a subvariety of dimension $d\ge1$ defined over $\F$. Let
  $s$ be a strictly small
$\R$-section of $\ov D$ which meets $Y$ properly. By B\'ezout's
formula \eqref{eq:20},
  \begin{displaymath}
    \h_{\ov D}(Y)=\h_{\ov D}(Y\cdot \div s)-
    \sum_{w\in \mathfrak{M}_{\F}}n_{w}\int_{X^{\an}_{w}}\log
    \|s\|_{w} \,c_{1}(D,\Vert\cdot\Vert_{w})^{\land d}\land \delta _{Y}.
  \end{displaymath} 
  Since $D$ is ample, $Y\cdot\div(s)$ is a nonzero $(d-1)$-dimensional
  effective cycle.  Applying linearity and the induction hypothesis,
  $ {\h}_{\ov D}(Y\cdot \div s)> 0$. The fact that the metric is
  semipositive implies that, for each $w\in \mathfrak{M}_{\F}$, the
  signed measure $c_{1}(D,\Vert\cdot\Vert_{w})^{\land d}\land \delta
  _{Y}$ is positive and its total mass is $\deg_{D}(Y)$. Therefore
  \begin{multline*}
    \h_{\ov D}(Y) > -
    \sum_{w\in \mathfrak{M}_{\F}}n_{w}\int_{X^{\an}_{w}}\log
    \|s\|_{w} \,\chern_{1}(D,\Vert\cdot\Vert_{w})^{\land d}\land \delta
    _{Y} \\ \ge -
    \sum_{w\in \mathfrak{M}_{\F}}n_{w}\log
    \|s\|_{w,\sup } \int_{X^{\an}_{w}}\chern_{1}(D,\Vert\cdot\Vert_{w})^{\land d}\land
    \delta _{Y} > 0,
  \end{multline*}
since $s$ is a strictly small $\R$-section and the last integral is
equal to $\deg_D(Y)$ for all~$w$. 
\end{proof}

\begin{rem} \label{rem:16} In \cite{Zhang:plbas}, Zhang proved a
  Nakai-Moishezon numerical criterion of arithmetic ampleness in the
  form of a converse to Proposition \ref{prop:6}\eqref{item:75} for
  Hermitian line bundles, under some technical hypothesis. It would be
  interesting to know if such a result is true in full generality for
  metrized $\R$-divisors. Namely: let $\ov D$ be a semipositive
  metrized $\R$-divisor such that $\h_{\ov D}(Y)>0$ for all effective
  cycles $Y$ of $X$. Is it true that $\ov D$ is ample?
\end{rem}

\begin{lem}\label{lemm:7} Let 
  $\ov D$ be a metrized $\R$-divisor on $X$.
  \begin{enumerate}
  \item \label{item:87} Let $\varphi\colon Z\to X$ be a 
    birational morphism of normal proper varieties over $\K$. Then 
    \begin{displaymath}
 \wh \vol(X,\ov
    D)=\wh \vol(Z,\varphi^{\ast}\ov D).     
    \end{displaymath}
    In particular, $\ov D$ is big if and only if $\varphi^{\ast}\ov D$
    is big.
  \item \label{item:97} If $\ov D$ is big and $\ov E\in \wh \Car(X)$
    has a small section,  then $\ov D+\ov E$ is big.
  \item \label{item:98} If $\ell_{0} \ov D$ is big for some $\ell_{0} \ge 1$,
    then $\ov D$ is big.
  \end{enumerate}
\end{lem}
\begin{proof}
  \eqref{item:87} Let $s=(f,D)\in \Gamma (X,D)^{\times}$. Then 
  $\div (s)=\div (f)+D$ is effective and so 
  $\div (\varphi^{\ast}s)= \varphi^{\ast}\div(s)$ is also
  effective. Hence, 
  $\varphi^{\ast}s=(f\circ \varphi,\varphi^{\ast}D)\in \Gamma
  (Z,\varphi^{\ast}D)^{\times}$. Thus, there is  a well-defined map
  $$
  \varphi^{\ast}\colon 
  \Gamma (X,D)^{\times}\to \Gamma (Z,\varphi^{\ast}D)^{\times}.
  $$ 

  We claim that this map is a bijection. The injectivity is clear
  because the map $f\mapsto f\circ \varphi$ is a bijection between $
  {\rm K}(X)^{\times}$ and $ {\rm K}(Z)^{\times}$.  Let now $s=(f\circ \varphi
  ,\varphi^{\ast}D)\in \Gamma (Z,\varphi^{\ast}D)^{\times}$ and set
  $s'=(f,D)\in \Rat(X,D)^{\times}$. Let $[\div(s)]$ be the $\R$-Weil
  divisor associated to $\div(s)$. Since $\div(s)$ is effective, the
  same is true for $[\div(s)]$. Since~$X$ is normal and $Z$ is proper,
  for each codimension one point $x\in X^{(1)}$ there is a
  neighbourhood $U$ of $x$ and a section $U\to Z$ of
  $\varphi$ \cite[II, 7.3.5]{GrothendieckDieudonne:EGA}. This implies
  that $[\div(s')]$ is effective. By \cite[IV,
  (21.6.9.1)]{GrothendieckDieudonne:EGA}, it follows that $\div(s')$ is
  effective and so  $s'\in \Gamma (X,D)^{\times}$, proving the
  claim. 

Given $s\in \Gamma(X,D)^{\times}$, then $\|s\|_{v,\sup}=\|\varphi^{\ast}
  s\|_{v,\sup}$ for all  $v$ and so $\varphi^{*}$ induces an
  isomorphism between $\whL(\ov D)$ and $\whL(\varphi^{\ast}\ov D)$. Hence
  $\wh \vol(X,\ov D)=\wh \vol(Z,\varphi^{\ast}\ov D)$, which proves
  the first statement.

  \eqref{item:97} Let $s_{0}$ be a small section of $\ov E$. There is
  an injective map
  \begin{displaymath}
 \wh \LL(\ell\ov D)\hookrightarrow \wh\LL(\ell(\ov D+ \ov E)) 
  \end{displaymath}
given by $s\mapsto
  s_{0}^{\ell}s$. Hence, $\wh\vol(X,\ov D+\ov E) \ge \wh\vol(X,\ov D) $
  and the statement follows.

\eqref{item:98} Assume that $\ell_{0} \ov D$ is big. We have that
\begin{displaymath}
\frac{1}{d_{\K}}\limsup_{\ell\to \infty} \frac{
      \whl(\ell \ov D) }{\ell^{n+1}/(n+1)!}\\ \ge \frac{1}{d_{\K}}\limsup_{\ell\to \infty} \frac{
      \whl( \ell_{0}\ell \ov D) }{( \ell_{0}\ell)^{n+1}/(n+1)!} .
\end{displaymath}
Hence, $ \wh \vol(X,\ov D)  \ge
\wh \vol(X,\ell_{0}\ov D)/\ell_{0}^{n+1}>0$ and so $\ov D$ is big. 
\end{proof}

\begin{prop} \label{prop:27} Let $\ov D$ be a pseudo-effective
  metrized $\R$-divisor on $X$, $\rho\colon X''\to X$ a 
  birational map of normal proper varieties over $\K$, and $\ov A$ an
  ample metrized divisor on $X''$. Then $ \ell \rho^{\ast}\ov D+ \ov
  A$ is big for all $\ell\ge 1$.
\end{prop}

\begin{proof}
  By the pseudo-effectiveness of $\ov D$, there is a birational map
  $\varphi\colon X'\to X$ from a normal proper variety $X'$ and a metrized divisor $\ov E$ on $X'$ such
  that $q\varphi^{*}\ov D+ \ov E$ is big for all $q \ge 1$. Consider
  the fibre product  of $X'$ and $X''$ over $X$ and the
  corresponding commutative diagram of varieties
  \begin{displaymath}
    \xymatrix{
      Y\ar[r]^{p_{1}}\ar[d]_{p_{2}} & X'\ar[d]^{\varphi}\\
      X''\ar[r]_{\rho}& X
    }
  \end{displaymath}
  where $p_{1}$ and $p_{2}$ are  birational maps.  Set
  $\phi=\varphi\circ p_{1}$. The  map $\phi\colon Y\to X$ is
  birational.  Since $A$ is ample, $p_{2}^{*}A$ is big, which implies
  that there is an integer $j_{0}\ge 1$ such that $j_{0}p_{2}^{*}A-
  p_{1}^{*} E$ has a nonzero section $s_{0}$.

  Let $S\subset \mathfrak M_{\K}$ be a finite subset such that
  $\|s_{0}\|_{\sup, v}\le 1$ for all $v\notin S$.  Let
  $\eta=(\sup_{v\in S}\|s_{0}\|_{v,\sup})^{-1}$.  Combining
  Proposition \ref{prop:6}\eqref{item:16}, Lemma \ref{lemm:5} and
  Proposition~\ref{prop:2}, we deduce that there exists $j_{1}\ge 1$
  and a small section $s_{1}$ of $j_{1}p_{2}^{*}\ov A$ such that
  $\|s_{1}\|_{v,\sup}<\eta$ for $v\in S$.  Hence, $s_{1}s_{0}$ is a
  small section of $(j_{0}+j_{1})p_{2}^{*}\ov A- p_{1}^{*}\ov E$. 

  Set $j_{2}=j_{0}+j_{1}$. Then
  \begin{displaymath}
    j_{2}(\ell\phi^{*} \ov D+ p_{2}^{*}\ov A)=(j_{2}\ell\phi^{*} \ov D+
    p_{1}^{*}\ov E)+ (j_{2}p_{2}^{*}\ov A- p_{1}^{*}\ov E).  
  \end{displaymath}
  Since $(j_{2}\ell\phi^{*} \ov D+ p_{1}^{*}\ov
  E)=p_{1}^{\ast}(j_{2}\ell \varphi^{\ast}\ov D+\ov E)$, this is a big
  metrized $\R$-divisor thanks to Lemma \ref{lemm:7}\eqref{item:87}.
  Lemma \ref{lemm:7}\eqref{item:97} and the fact that
  $j_{2}p_{2}^{*}\ov A- p_{1}^{*}\ov E$ has a small section imply
  that $j_{2}(\ell\phi^{*} \ov D+ p_{2}^{*}\ov A)$ is big. By Lemma
  \ref{lemm:7}\eqref{item:98}, $\ell\phi^{*} \ov D+ p_{2}^{*}\ov
  A=p_{2}^{\ast}(\ell \rho ^{\ast}\ov D+\ov A)$ is big. By Lemma
  \ref{lemm:7}\eqref{item:87}, we conclude that $\ell \rho ^{\ast}\ov
  D+\ov A$ is big.
\end{proof}

  In \cite{Moriwaki:tdutav}, Moriwaki proposed an extension of Dirichlet's
  unit theorem to the higher-dimensional case. The
  following is the natural extension of this question to our more general
  setting. 

\begin{ques}\label{ques:1}
  (Dirichlet's unit theorem) Let $\K$ be an $A$-field, $X$ a normal proper  variety
  over~$\K$ and $\ov D$ a metrized $\R$-divisor on $X$. Are the
  following conditions \eqref{item:48} and \eqref{item:52} equivalent?
   \begin{enumerate}
   \item \label{item:48} $\ov D$ is pseudo-effective; 
   \item \label{item:52} there exists $f\in  {\rm K}(X)^\times_{\R}$ such that
     $\ov D+\wh \div(f)\ge 0$.
   \end{enumerate}
\end{ques}

In the setting of Question \ref{ques:1}, the classical Dirichlet's unit
theorem  shows up when considering the
zero-dimensional case. Let $\K$ be an $A$-field and $S\subset \mathfrak{M}_{\K}$  a
finite subset containing the Archimedean
places. Let $U_{S}$ be the group of $S$-units of $\K$ and $H_{S}$ the
hyperplane of $\R^{S}$ defined by $\sum_{v\in S}n_{v}\xi
_{v}=0$. The adelic version of Dirichlet's unit theorem states that the
regulator map
$U_{S}\otimes_{\Z}\R\to H_{S}$ given by $u\mapsto(\log|u|_{v})_{v\in S}$
is an isomorphism \cite[Chapter IV, \S 4, Theorem 9]{Weil:bnt}.

Let now $X=\Spec(\K)$, $D=0$ the zero divisor on $X$ and
$(\xi_v)_{v\in S}\in H_{S}$. We put $\xi_{v}=0$ for $v\not \in
S$. Each $\xi_{v}$ gives a $v$-adic Green function for $D$ and we
denote by $\ov D$ the associated metrized divisor. It can be verified
that this metrized divisor is pseudo-effective (see for instance
Theorem \ref{thm:2}\eqref{item:100}).  Condition \eqref{item:52} for
$\ov D$ is equivalent to the existence of
$u\in\K^\times\otimes_{\Z}\R$ such that $\xi_v-\log|u|_v \geq 0$ for
all $v$. Since $\sum_vn_{v}\log|u|_v=0=\sum_vn_{v}\xi_v$ because of
the product formula, the previous inequality forces $\log|u|_v=\xi_v$
for all places and, in particular, $u\in U_S\otimes_\Z\R$.  Hence, the
implication \eqref{item:48}$\Rightarrow$\eqref{item:52} in Question
\ref{ques:1} is equivalent to the surjectivity of the regulator map in
Dirichlet's unit theorem.

In higher dimension, the fact that \eqref{item:52} implies
\eqref{item:48} follows from the facts that effective metrized
$\R$-divisors are pseudo-effective and that pseudo-effectiveness is
invariant with respect to linear equivalence.  Moriwaki has
proven the reverse implication  when $X$ is smooth
and projective, $D$ is numerically trivial and the metrics at the
finite places come from a common normal projective model over
$\cO_{\K}$ \cite{Moriwaki:tdutav}.

\begin{rem} \label{rem:17}
If $\K$ is a general function field, the answer is
``no'', simply because the analogue of the classical Dirichlet's unit
theorem does not hold. The simplest example is the following. Let $C$
be an elliptic 
curve over a field $k$ of characteristic zero, $\K={\rm K}(C)$ and
$X=\Spec(\K)$. Let 
$\mathcal{D}$ be a divisor of degree zero on $C$ whose class in the
Picard group is non-torsion
and let $\ov D$ be the corresponding metrized $\R$-divisor on
$X$. Then $\ov D$ is pseudo-effective but there is no $f\in
 {\rm K}(X)^{\times}_{\R}$ such that $\ov D+\wh \div (f)\ge 0$. Assume that
 such $f$ exists. Then $\mathcal{D}+\div (f)$ would be effective
 and of degree 0. Hence $\mathcal{D}+\div (f)=0$. Since $\mathcal{D}$
 is a divisor the previous equation implies that we can find a $g\in
 {\rm K}(X)^{\times}$ and $m\in \Z_{\ge 1}$ with $m \mathcal{D}+\div
 (g)=0$. Therefore the 
 class of $\mathcal{D}$ in the Picard group is torsion, contradicting
 the hypothesis. 
\end{rem}

In \S \ref{sec:dirichl-unit-theor} we will show that, for toric
varieties, the answer to Question \ref{ques:1} is positive.

We now turn our attention towards the approximation of pseudo-effective and big
divisors by nef and ample ones. 

\begin{defn} \label{def:7}
  Let $X$ be a normal  proper variety over $\K$
and $\ov D$ a pseudo-effective  metrized $\R$-divisor on~$X$. 
A \emph{Zariski decomposition} of $\ov D$ is a 
birational map $\varphi\colon X'\to X$ of normal proper varieties over $\K$ and a decomposition
\begin{equation*}
 \varphi^{*}\ov D=\ov P+\ov E 
\end{equation*}
with $\ov P,\ov E\in\wh\Car(X')_{\R}$ such that 
$\ov P$ is nef,  $\ov E$ is effective and 
\begin{math}
 \wh\vol(X',\ov P)= \wh\vol(X,\ov D).
\end{math}
\end{defn}

Sometimes it is convenient to consider the following
variant. 

\begin{defn} \label{def:2} Let $X$ be a normal proper  variety over $\K$ and
  $\ov D$ a pseudo-effective metrized $\R$-divisor on~$X$.  Let
  $\Upsilon(\ov D)$ be the set of pairs $(\varphi,\ov P)$, where
  $\varphi\colon X'\to X$ is a birational map of normal proper 
  varieties over $\K$ and $\ov P$ is a
  nef metrized $\R$-divisor on $X'$ such that $\varphi^{\ast}\ov D-\ov
  P\ge 0$.  On $\Upsilon(\ov D)$ we consider the equivalence relation
  $(\varphi,\ov P)\sim (\varphi_1,\ov P_1)$ whenever there exists a
  commutative diagram of birational morphisms
\begin{displaymath}
  \xymatrix{Z \ar[r]^{\nu} \ar[d]^{\nu_1} & X'\ar[d]^{\varphi}\\
  X'_{1}\ar[r]^{\varphi_{1}}& X}
\end{displaymath}
such that $\nu ^{\ast}\ov P=\nu_{1}^{\ast}\ov P_{1}$. On the set of
equivalence classes $\Upsilon(\ov D)/\sim$, we consider the order
relation given by $[(\varphi,\ov P)]\le [(\varphi_{1},\ov P_{1})]$
whenever there is a commutative diagram of birational maps as above
with $\nu ^{\ast}\ov P\le\nu_{1}^{\ast}\ov P_{1}$. 

The \emph{strong Zariski decomposition} of $\ov D$ is the greatest
element of $\Upsilon (\ov D)/\sim$, if it exists.
\end{defn}

If $\ov D$ is a big metrized $\R$-divisor on $X$, then a strong
Zariski decomposition of $\ov D$ gives a Zariski decomposition in the
sense of Definition \ref{def:7} \cite[Proposition~B.1]{Moriwaki:cnadas}.

\begin{ques}\label{ques:3}
  Let $\ov D$ be a pseudo-effective metrized $\R$-divisor on $X$. 
When does~$\ov D$ admit a Zariski decomposition or a strong Zariski decomposition?
\end{ques}

In \cite{Moriwaki:zdas}, Moriwaki showed that a strong Zariski
decomposition of $\ov D$ exists if $\K$ is a number field, $X$ is a
curve, $\ov D$ is big and the metrics at the finite places come from
a common normal projective model over $\cO_{\K}$.  In higher
dimension, it is not true that every big metrized $\R$-divisor admits
a Zariski decomposition. Indeed, there are examples of toric big
metrized divisors on $\P^{2}$ that do not admit a Zariski
decomposition and, a fortiori, do not admit a strong Zariski
decomposition~\cite{Moriwaki:tdutav}. In \S
\ref{sec:dirichl-unit-theor}, we will consider toric Zariski
decompositions and toric strong Zariski decompositions, and we will
give a criterion for such decompositions to exist. Furthermore, in \S
\ref{sec:zariski-decomp} we will show that, under some hypothesis, the
existence of a non-necessarily toric Zariski decomposition of a big
toric metrized $\R$-divisor implies the existence of a toric one.

In the absence of a Zariski decomposition, one can ask for the
existence of a Fujita approximation. 

\begin{ques} \label{ques:4} (Fujita approximation) Let $\ov D$ be a
  big metrized $\R$-divisor on $X$. Let $\varepsilon >0$ be a positive
  real number. Does there exist a  birational proper map $\varphi\colon X'\to
  X$ and metrized $\R$-divisors $\ov A,\ov E\in\wh\Car(X')_{\R}$ such
  that $\ov A$ is ample, $\ov E$ is effective,
\begin{displaymath}
  \varphi^{*}\ov D=\ov A+\ov E\quad\text{and}\quad  \wh\vol(X',\ov
 A)\ge \wh\vol(X,\ov D)-\varepsilon?
\end{displaymath}
\end{ques}

The existence of an arithmetic Fujita approximation was independently
obtained by Yuan \cite{Yuan:valb} and by Chen \cite{Chen:afa} in the
case when $\K$ is a number field, $D$ is a divisor and the metrics at
the infinite places are smooth and those at the finite places come
from a common projective model over $\cO_{\K}$.  More recently,
Boucksom and Chen \cite{BoucksomChen:Obfs} have given a more
elementary proof of this fact. 

In \S \ref{sec:dirichl-unit-theor} we will give a proof of the
Fujita approximation theorem for big toric metrized $\R$-divisors on a toric variety.  

\section{Toric metrized $\R$-divisors} \label{sec:toric-metrized-r}

In this section, we recall the necessary background on the algebraic
and arithmetic geometry of toric varieties from
\cite{BurgosPhilipponSombra:agtvmmh} and
we extend some of the results in this reference to toric metrized
$\R$-divisors. We will follow the notations and conventions in
\cite[Chapters 3 and 4]{BurgosPhilipponSombra:agtvmmh}.

Let $N\simeq \Z^{n}$ be a lattice and $M=N^{\vee}$ the dual
lattice. Set $N_{\R}=N\otimes \R$ and $M_{\R}=M\otimes \R$. The
pairing between $x\in M_{\R}$ and $u\in N_{\R}$ is denoted by $\langle
x,u\rangle$.

Let $K$ be a field and set $\T=\Spec(K[M])\simeq \G_{m}^{n}$ for the
split torus over $K$ corresponding to $N$.  Let $\Sigma $ be a
complete (rational) fan on $N_{\R}$ and $X_{\Sigma }$ the proper toric variety
over $K$ defined by $\Sigma $. We write $X=X_{\Sigma}$ for short. This
is a normal variety of dimension $n$ with an open dense immersion
$\T\hookrightarrow X$ and an action of $\T$ on $X$ that extends the
action of $\T$ on itself by translations.

The toric variety $X$ has a distinguished point $x_{0}$ in its
principal open subset $X_{0}$, corresponding to the unit of the
torus $\T$. A {toric line bundle} is a line bundle $L$ on~$X$
together with the choice of a nonzero point $z_{0}\in L_{x_{0}}$
\cite[Definition 3.3.4]{BurgosPhilipponSombra:agtvmmh}. A {toric
  section} of $L$ is
a section $s$ that is regular and nowhere vanishing on the principal
open subset $X_{0}$, and such that $s(x_{0})=z_{0}$. There is a
bijection between toric divisors and isomorphism classes of toric line
bundles with a toric section. If $(L,s)$ is a
toric line bundle with a toric section, then $\div(s)$ is a
toric divisor \cite[Theorem 3.3.7]{BurgosPhilipponSombra:agtvmmh}.  Conversely,
given a toric divisor $D$ on $X$, the line bundle $\cO(D)$ is a
subsheaf of $\cK_{X}$ and the rational function $1\in  {\rm K}(X)$ provides a
distinguished rational section $s_{D}$ of $\cO(D)$ such that
$\div(s_{D})=D$. This
section does not vanish on $X_{0}$ and so $(\cO(D),s_{D}(x_{0}))$ is a
toric line bundle.  The correspondence $D\mapsto
((\cO(D),s_{D}(x_{0})),s_{D})$ is the inverse of the correspondence
defined by $(L,s)\mapsto \div(s)$.  
Thus, the languages of toric line bundles with toric sections and
that of toric divisors are equivalent.  In the sequel, we will mostly
use the latter and, more generally, that of
$\R$-divisors. We  denote by $\Car_{\T}(X)$ the group of toric divisors on the toric variety $X$.

\begin{defn} \label{def:19} An \emph{$\R$-virtual support function} on
    $\Sigma$ is a function $\Psi$ on $N_{\R}$ such that, for each
  cone $\sigma\in \Sigma$, there exists $m_{\sigma}\in M_{\R}$ such
  that $\Psi(u)= \langle m_{\sigma},u\rangle$ for all $u\in \sigma$.
  A set of functionals $\{m_{\sigma}\}_{\sigma\in \Sigma}$ as above is
  called a set of \emph{defining vectors} of~$\Psi$. If we can choose
  $m_{\sigma}\in M$ for all $\sigma$, then $\Psi$ is called a \emph{virtual support
    function} on $\Sigma$. We respectively denote by $\VSF(\Sigma)$ and by
  $\VSF(\Sigma)_{\R}$ the spaces of  {virtual support functions}
  and of $\R$-virtual support functions.

  A concave virtual support function (respectively, $\R$-virtual
  support function) on $\Sigma$ is called a \emph{support function}
  (respectively, an \emph{$\R$-support function}) {on
    $\Sigma$}. We denote by $\SF(\Sigma)$ the semigroup of support
  functions on $\Sigma$ and by $\SF(\Sigma)_{\R}$ the convex cone of
  $\R$-support functions on $\Sigma$.
\end{defn}

To a toric divisor $D$ one associates a virtual support
function, denoted by $\Psi_{D}$, in the following way: for each $\sigma\in \Sigma$, there is
$m_{\sigma}\in M$ such that $D=\div(\chi^{-m_{\sigma}})$ on the affine
open set $X_{\sigma}$. Then, $\{m_{\sigma}\}_{\sigma}$ is a set of
defining vectors for $\Psi_{D}$. 
The correspondence $D\mapsto \Psi_{D}$ is an isomorphism of
$\Z$-modules between $\Car_{\T}(X)$ and $\VSF(\Sigma)$.

\begin{defn} \label{def:18}
A \emph{toric} $\R$-divisor on $X$ is a finite linear combination
\begin{equation*} 
  D=\sum_{i}\alpha_{i}D_{i}
\end{equation*}
with $\alpha_{i} \in \R$ and $D_{i}$ a toric divisor. To a toric
$\R$-divisor $D$ as above, we associate the $\R$-virtual support
function $\Psi_{D}= \sum_{i}\alpha_{i}\Psi_{D_{i}}$. We
set 
\begin{displaymath}
\Car_{\T}(X)_{\R}=\Car_{\T}(X)\otimes_\Z \R   
\end{displaymath}
for the linear space of toric $\R$-divisors on $X$.  
\end{defn}

There is a group morphism $M\to  {\rm K}(X)^{\times}$ given by
$m\mapsto \chi^{m}$, where $\chi^{m}\in \Hom(\T,\G_{m})$ is
the character corresponding to $m$.  By linearity, we can extend it to
a group morphism $M_{\R}\to  {\rm K}(X)^{\times}_{\R}$. We also denote by
$\chi^{m}$ the image of $m$ under this map. Composing with the map
$\div$, each element $m\in M_{\R}$ gives rise to a toric $\R$-divisor
$\div(\chi^{m})$. For $m\in M_{\R}$, we set $s_{m}=(\chi^{m},D)$ for
the corresponding rational $\R$-section of $D$.

\begin{prop} \label{prop:15} The correspondence $D\mapsto \Psi_{D}$ is an
    isomorphism of linear spaces between
$\Car_{\T}(X)_{\R}$ and $\VSF(\Sigma)_{\R}$.
\end{prop}

\begin{proof}
  The correspondence $D\mapsto \Psi_{D}$ is an isomorphism of
  $\Z$-modules between $\Car_{\T}(X)$ and $\VSF(\Sigma)$. Hence, it
  also defines an isomorphism between $\Car_{\T}(X)_{\R}$ and
  $\VSF(\Sigma)\otimes_{\Z} \R$. The space $\VSF(\Sigma)_{\R}$ can be
  identified with the linear subspace of $\prod_{\sigma\in
    \Sigma^{n}}M_{\R}$ defined by
  \begin{equation} \label{eq:19} \{(m_{\sigma})_{\sigma}\mid
    m_{\sigma}-m_{\tau} \in (\sigma\cap \tau)^{\bot} \text{ for all }
    \sigma,\tau\in \Sigma^{n}\}.
\end{equation} 
This subspace is defined over $\Q$ because the fan $\Sigma$ is
rational, and its restriction to $\prod_{\sigma\in \Sigma^{n}}M$
agrees with $\VSF(\Sigma)$. Hence,
$\VSF(\Sigma)_{\R}=\VSF(\Sigma)\otimes_{\Z}\R$, which proves the statement.
\end{proof}

\begin{defn} \label{def:13}
A nonempty compact subset  $C\subset M_{\R}$ is called a \emph{quasi-rational} polytope
if there are $u_{j}\in N_{\Q}$ and $\gamma_{j}\in \R$, $j=1,\dots, l$,
such that 
\begin{displaymath}
  C=\{x\in M_{\R}\mid \langle x,u_{j}\rangle \ge \gamma_{j}, j=1,\dots, l\}. 
\end{displaymath}
Let $\Sigma_{C}$ denote the normal fan of $C$. We say that $C$ is  \emph{compatible with
$\Sigma$} whenever $\Sigma$ refines $\Sigma_{C}$.
\end{defn}

To a toric $\R$-divisor $D$ on $X$ we associate the subset of $M_{\R}$
defined as
\begin{displaymath}
  \Delta_{D}=\stab(\Psi_{D}),
\end{displaymath}
the stability set of $\Psi_{D}$ (Definition \ref{def:4}).  This set is
either empty or a quasi-rational polytope compatible with $\Sigma$.
It encodes a lot of information about the geometry of the pair
$(X,D)$.  For instance, each element $m\in \Delta _{D}\cap M$
gives a toric section $s_{m}=(\chi^{m},D)\in \Gamma(X,D)^{\times}$. Analogously, each $m\in
\Delta_{D}$ defines a toric $\R$-section $s_{m}\in
\Gamma(X,D)^{\times}_{\R}$.
The set $\{s_{m}\}_{m\in \Delta_{D}\cap M}$ is a basis of $\LL
(D)$. The proofs of these statements are the same as those for Cartier
divisors \cite[\S 3]{Fulton:itv}.

\begin{prop} \label{prop:8} 
Let $D$ be a toric $\R$-divisor on
  $X$. Then 
  \begin{displaymath}
    \vol(X,D)=n!\vol_{M}(\Delta_{D}),
  \end{displaymath}
 where $\vol_{M}$ is
  the Haar measure of $M_{\R}$ normalized so that $M$ has covolume
  $1$.
\end{prop}

\begin{proof}
We have that 
\begin{displaymath}
  \vol(X,D)=\lim_{\ell\to \infty} \frac{\lL(\ell D)}{\ell^{n}/n!}=
  n! \lim_{\ell\to \infty} \frac{\#(\ell\Delta_{D}\cap M)}{\ell^{n}}= n!\vol_{M}(\Delta_{M}).
\end{displaymath}
\end{proof}

\begin{prop} \label{prop:26}
Let $D$ be a toric $\R$-divisor on $X$. Then 
  \begin{enumerate}
\item \label{item:53} $D$ is ample if and only if $\Psi_{D}$ is
  strictly concave.
\item \label{item:83} The following conditions are equivalent:
  \begin{enumerate}
  \item \label{item:84} $D$ is nef;
\item \label{item:85} $\Psi_{D}$ is concave;
\item \label{item:86} there exists a finite number of nef divisors
  $D_{i}$ on $X$ and $\alpha_{i}>0$ such that $D=\sum_{i}{\alpha_{i}}D_{i}$. 
\end{enumerate}
\item  \label{item:46} If $D$ is nef, then $\deg_{D}(X)=n!\vol_{M}(\Delta_{D})$.
  \end{enumerate}
\end{prop}

\begin{proof}
  \eqref{item:53} A toric divisor is ample if and only if its
  corresponding function is strictly concave. Hence, a toric
  $\R$-divisor $D$ is ample if and only if $\Psi_{D}$ is a linear
  combination, with positive real coefficients, of strictly concave
  support functions. For each cone $\sigma\in \Sigma^{n}$,
  choose $u_{\sigma}\in \ri(\sigma)\cap N$, where $\ri(\sigma)$
  denotes the relative interior of $\sigma$. Using the
  identification in~\eqref{eq:19}, we see that 
\begin{displaymath}
\SF(\Sigma)_{\R} =  \{(m_{\sigma})_{\sigma} \mid \langle
  m_{\sigma}-m_{\tau},u_{\sigma}\rangle \le 0 \text{ for all
  } \sigma,\tau\in \Sigma^{n}\}
\end{displaymath}
is a convex rational polyhedral  cone in
$\VSF(\Sigma)_{\R}$. Its interior
\begin{displaymath}
\SF(\Sigma)_{\R}^{\circ} =  \{(m_{\sigma})_{\sigma} \mid \langle
  m_{\sigma}-m_{\tau},u_{\sigma}\rangle < 0 \text{ for all
  } \sigma,\tau\in \Sigma^{n} \text{ such that } \sigma\ne \tau\}  
\end{displaymath}
can be identified with the subset of $\R$-support functions that are
strictly concave on $\Sigma $. Hence, any strictly concave $\R$-support
function can be written as a linear combination with positive real
coefficients of strictly concave support functions,
which proves the statement. 

\eqref{item:83} The equivalence of \eqref{item:84} and
\eqref{item:85} can be proved as in the case of divisors, see for
instance \cite[Theorem 6.1.12]{CoxLittleSchenck:tv}. The equivalence between
\eqref{item:85} and \eqref{item:86} follows from the
facts that a toric divisor is nef if and only if its corresponding
function lies in~$\SF(\Sigma)$ and that $\SF(\Sigma)_{\R}$ is a convex
rational polyhedral cone.

\eqref{item:46} This formula is well-known for toric divisors.  Using
\eqref{item:83}, the general case follows from this one together with
the multilinearity of the mixed degree and of the mixed volume.
\end{proof}

The set of quasi-rational polytopes of $M_{\R}$ which are compatible
with $\Sigma$ forms a convex cone with respect to the multiplication by scalars in $\R_{\ge0}$ and the Minkowski sum of sets.

\begin{prop} \label{prop:20} The correspondence $D\mapsto \Delta_{D}$
  gives an isomorphism between the convex cone of nef toric $\R$-divisors on
  $X$ and the convex cone of quasi-rational polytopes compatible with
  $\Sigma$.
\end{prop}

\begin{proof}
Let $\alpha \in \R_{\ge 0}$ and  $\Psi, \Phi$  two conic concave
functions on $N_{\R}$ (Appendix \ref{sec:conv-analys-asympt}). Then 
\begin{displaymath}
  \stab(\alpha  \Psi)=\alpha\stab(\Psi), \quad \stab(\Psi+\Phi)=
  \stab(\Psi)+\stab(\Phi). 
\end{displaymath}
Since $\Psi$ is recovered from $\stab(\Psi)$ as the Legendre-Fenchel
dual of the indicator function of this convex set (see \cite[Example
2.2.1]{BurgosPhilipponSombra:agtvmmh}), the map $\Psi\to
\stab(\Psi)$ is an isomorphism between the convex cone of conic
concave functions on $N_{\R}$ and that of convex bodies. This
isomorphism sends $\SF(\Sigma)_\R$ onto the convex cone of
quasi-rational polytopes compatible with $\Sigma$.  The statement
follows easily from this since, by propositions \ref{prop:15} and
\ref{prop:26}\eqref{item:83}, the convex cone of nef toric
$\R$-divisors is in one-to-one correspondence with $\SF(\Sigma)_\R$.
\end{proof}

\begin{prop} \label{prop:16}
Let $D$ be a toric $\R$-divisor on $X$. 
  \begin{enumerate}
\item \label{item:43} The following conditions are equivalent:
  \begin{enumerate}
  \item \label{item:37} $D$ is big;
  \item \label{item:76} there exists $m\in M_{\R}$ such that
    $\Psi_{D}(u) < \langle
    m,u\rangle $ for all $u\in N_{\R}\setminus \{0\}$;
  \item \label{item:77}  $\dim(\Delta_{D})=n$.
  \end{enumerate}
\item \label{item:44}  The following conditions are equivalent: 
  \begin{enumerate}
  \item \label{item:78} $D$ is pseudo-effective;
  \item \label{item:79} there exists $m\in M_{\R}$ such that
    $\Psi_{D}(u) \le \langle
    m,u\rangle $ for all $u\in N_{\R}$;
\item \label{item:82}   $\Delta_{D}\ne \emptyset$.
  \end{enumerate}
\item \label{item:45} $D$ is effective if and only if $\Psi_{D}\le
  0$ or, equivalently,  if and only if $0\in
  \Delta_{D}$.
\item \label{item:47} Let $P$ be a nef toric $\R$-divisor on $X$. Then
  $D\ge P$ if and only if $\Delta_{D}\supset \Delta_{P}$. 
  \end{enumerate}
\end{prop}

\begin{proof}
\eqref{item:43} Clearly, \eqref{item:76} and \eqref{item:77} are
equivalent. In case $D$ is a divisor, Proposition~\ref{prop:8}, implies that \eqref{item:37} and  \eqref{item:77} 
are equivalent.

Assume $D$ is a big $\R$-divisor and write $D=\sum_{\alpha_{i}}
\alpha_{i} D_{i}$ with $\alpha_{i}>0$ and $D_{i}$ big. Then there
exist $m_{i}\in M_{\R}$ such that $\Psi_{D_{i}}(u)< \langle m_{i},u\rangle
$  for all $u\in N_{\R}\setminus \{0\}$. Setting
$m=\sum_{i}\alpha_{i}m_{i}$, we have that 
$\Psi_{D}(u) < \langle m,u\rangle$  for all $u\in N_{\R}\setminus \{0\}$, which proves~\eqref{item:76}. 

Conversely, the set 
\begin{displaymath}
  C=\{ \Psi\in \VSF(\Sigma)_{\R}\mid \exists  m\in M_{\R} \text{ such
    that } \Psi(u)<\langle  m,u\rangle  \text{ for all } u\in N_{\R}\setminus \{0\}\}
\end{displaymath}
is an open convex cone.  Let $D$ be an $\R$-divisor such that
$\Psi_{D}\in C$.  Since $ \VSF(\Sigma)$ is dense in $
\VSF(\Sigma)_{\R}$, there exist a finite number of functions
$\Psi_{i}\in C\cap \VSF(\Sigma)$ and positive real numbers
$\alpha_{i}$ such that $\Psi_{D}=\sum_{i}\alpha_{i}\Psi_{i}$.  For
each $i$, let $D_{i}$ be the divisor corresponding to $\Psi_{i}$. Then
$D=\sum_{\alpha_{i}} \alpha_{i} D_{i}$. Hence, $D$ is big since each
$D_{i}$ is big. This proves the statement. 

\eqref{item:44} Let
$\varphi\colon X'\to X$ be a birational toric map and $B$ a toric effective  big
$\R$-divisor on $X'$. 

Suppose first that $\Delta_{\varphi^\ast D}=\Delta_{D}\ne\emptyset$ and let $\ell
>0$. By Lemma \ref{lemm:4}\eqref{item:35},
\begin{displaymath}
  \Delta_{\ell \varphi^\ast D+B}\supset \Delta_{\ell \varphi^\ast D}+\Delta_{B}.
\end{displaymath}
By \eqref{item:43}, the polytope $\Delta_{B}$ has dimension $n$
and so does $  \Delta_{\ell \varphi^\ast D+B}$. Hence, $\ell \varphi^{\ast } D+B$ is
big for all integers $\ell \ge 0$ and so $D$ is pseudo-effective.

Conversely, suppose that $D$ is pseudo-effective. By an argument
similar to the one in the proof of Proposition \ref{prop:27}, one can
verify that the $\R$-divisor $\ell
\varphi^{\ast} D+B$ is big for all $\ell\ge1$. Hence, $\Delta_{\ell
 \varphi^{\ast} D+B}$ is of dimension $n$ and, in particular, nonempty.  By Lemma
\ref{lemm:6}\eqref{item:36},
\begin{displaymath}
\Delta_{D}=\Delta_{\varphi^{\ast}D}= \bigcap_{\ell >0}\Delta_{\varphi^\ast D+\frac{1}{\ell} B}.  
\end{displaymath}
Hence, this polytope is nonempty, since it is the intersection of
nested compact sets. 

\eqref{item:45} Assume that $D$ is effective and write
$D=\sum_{i}\alpha_{i}D_{i}$ with $\alpha_{i}>0$ and $D_{i}$ an
effective divisor. Then $0\in \Delta_{D_{i}}$ for all $i$. By Lemma
\ref{lemm:4}\eqref{item:35}, 
\begin{displaymath}
  \sum_{i}\alpha_{i}\Delta_{D_{i}}=
  \sum_{i}\alpha_{i}\stab(\Psi_{D_{i}}) \subset \stab(\Psi_{D})= \Delta_{D}.
\end{displaymath}
Hence, $0\in\Delta_{D}$ or, equivalently, $\Psi_{D}\le 0$. 

Conversely, the set of functions $\Psi\in \VSF(\Sigma)_{\R}$ such that
$\Psi\le 0$ forms a rational convex cone. Hence, if $\Psi_{D}\le 0$
then there is a finite number of functions $\Psi_{i}\in \VSF(\Sigma)$
such that $\Psi_{i}\le 0$ and $\alpha_{i}>0$ such that
$\Psi_{D}=\sum_{i}\alpha_{i}\Psi_{i}$. Each $\Psi_{i}$ corresponds to
a toric divisor $D_{i}$ and $D=\sum_{i}\alpha_{i}D_{i}$. Each
$D_{i}$ is effective and so is~$D$.

\eqref{item:47} Suppose that $D\ge P$. Then $\Psi_{D-P}\le 0$, which
is equivalent to $\Psi_{D}\le \Psi_{P}$. Hence, $\Delta_{D}\supset
\Delta_{P}$. Conversely, suppose that $\Delta_{D}\supset
\Delta_{P}$. Then $\Delta_{D}\ne\emptyset$, since $P$ is nef and
$\Delta_{D}$ contains $\Delta_{P}$. Hence, the support function
$\Psi_{\Delta_{D}}$ coincides with the concave envelope of $\Psi_{D}$.
Then
\begin{math}
  \Psi_{D}\le \Psi_{\Delta_{D}}\le \Psi_{\Delta_{P}}=\Psi_{P}, 
\end{math}
which proves $\Psi_{D-P}\leq0$ and $D\geq P$. 
\end{proof}

In the toric case, the geometric analogues of Dirichlet's unit theorem
(Question~\ref{ques:1}), Zariski
decomposition (Question \ref{ques:3}) and Fujita approximation
theorem (Question \ref{ques:4}) are easy 
to treat, and all 
three admit a positive answer. Note however that the notion of toric strong Zariski
decomposition that appears in the proposition below is weaker than the
strong Zariski decomposition.  

\begin{prop} \label{prop:21} \ 
Let $D$ be a toric $\R$-divisor on $X$. 
  \begin{enumerate}
  \item \label{item:49} $D$ is pseudo-effective if and only if there
  exists $m\in \Delta_{D}$ such that
  \begin{displaymath}
 D+\div(\chi^{m})\ge 0.   
  \end{displaymath}

\item \label{item:50} Assume that $D$ is pseudo-effective.  Then there
  exist a birational toric map $\varphi\colon X'\to X$ of proper toric
  varieties and toric
  divisors $P, E\in \Car_{\T}(X')_{\R}$ such that $P$ is nef, $E$ is
  effective,
  \begin{displaymath}
\varphi^{*}D= P+E \quad \text{and}\quad    \vol(X,P)= \vol(X,D).
  \end{displaymath}
Moreover, for any other birational toric map
  $\varphi_{1}\colon X_{1}'\to X$ of proper toric varieties and a decomposition $\varphi_{1}^{*}D=
  P_{1}+E_{1}$ with   $P_{1}, E_{1}\in
  \Car_{\T}(X'_{1})_{\R}$ such that $P_{1}$ is nef and $E_{1}$
  is effective, there are proper birational toric maps  $\nu \colon X''\to
  X'$ and $\nu _{1}\colon X''\to X'_{1}$  of proper toric varieties satisfying $\nu^{*}P\ge \nu _{1}^{\ast}P_{1}$. 

\item \label{item:51} Assume that $D$ is big and let $\varepsilon>0$.
  Then there exist a birational toric map $\varphi\colon X'\to X$ 
   of proper toric varieties and $A, E\in \Car_{\T}(X')_{\R}$ such that $A$ is ample, $E$ is
  effective,
  \begin{displaymath}
\varphi^{*}D= A+E \quad \text{and}\quad  \vol(X,A)\ge
  \vol(X,D)-\varepsilon.
  \end{displaymath}

  \end{enumerate}
\end{prop}
\begin{proof}
  \eqref{item:49} By Proposition \ref{prop:16}\eqref{item:44}, $D$ is
  pseudo-effective if and only if $\Psi_{D}-m\leq0$ for some $m\in
  M_\R$. But $\Psi_{D}-m=\Psi_{D+\div(\chi^m)}$ and, by Proposition~\ref{prop:16}\eqref{item:45}, the previous condition is
  equivalent to the fact that $D+\div(\chi^{m})\ge0$.

  \eqref{item:50} By Proposition \ref{prop:16}\eqref{item:44},
  $\Delta_{D}\ne\emptyset$. Let $\Sigma '$ be a refinement of $\Sigma
  $ compatible with $\Delta _{D}$. Let $\varphi\colon X'\to X$ be the
  corresponding birational toric map  of proper toric varieties. Let $P$ be the nef
  toric $\R$-divisor on $X'$ associated to $\Delta _{D}$ under the
  correspondence in Proposition \ref{prop:20}. Set
  $E=\varphi^{\ast}D-P$. By Proposition~\ref{prop:16}(\ref{item:47}),
  $E$ is effective and, by Proposition~\ref{prop:8}, $\vol(X',P)=
  \vol(X,D)$.  Furthermore, let $X'_{1}$, $P_{1}$ and $E_{1}$ as in
  the statement. Let $\Sigma _{1}$ be the fan that determines
  $X'_{1}$. Let $\Sigma ''$ be a common refinement of $\Sigma $,
  $\Sigma '$ and $\Sigma '_{1}$ and let $\nu \colon X''\to X'$ and
  $\nu _{1}\colon X''\to X'_{1}$ be the associated proper
  birational toric maps. By Proposition \ref{prop:16}\eqref{item:47},
  $\Delta_{P_{1}}\subset \Delta_{D} = \Delta_{P}$ and, by the same
  result, $\nu^{*}P\ge \nu^{\ast}_{1}P_{1}$.

  \eqref{item:51} Let $\Sigma'$ be a regular refinement of $\Sigma$.
  The toric variety $X':=X_{\Sigma'}$ is projective and there is a
  birational toric map $\varphi\colon X'\to X$. Let $D'$ be an ample
  toric $\R$-divisor on $X'$.  Let $P$ be the nef $\R$-divisor
  on $X$ given by \eqref{item:50}, for which we have
  $\Delta_{P}=\Delta_{D}$. Let $0\le \gamma <1$ and $\delta>0$
  such that
\begin{equation}\label{eq:15}
  \gamma \Delta_{D}+ \delta \Delta_{D'}\subset \Delta_{D}, \quad
  \vol_{M}(\gamma \Delta_{D}+ \delta \Delta_{D'})\ge
  \vol_{M}(\Delta_{D}) -\frac{\varepsilon}{n!}. 
\end{equation}
Set $A= \gamma \varphi^{*}{P}+ \delta {D'}$ and $E=
\varphi^{*}{D}-A$.   By
Proposition \ref{prop:20},  $\Delta_A = \gamma\Delta_{\varphi^\ast
  P}+\delta\Delta_{D'}\subset \Delta_D=\Delta_{\varphi^\ast D}$. Proposition
\ref{prop:16}(\ref{item:47}) then implies that $E$ is effective.
The virtual support function corresponding to $A$
is $\gamma\Psi_{P}+\delta\Psi_{D'}$, which is strictly concave on
$\Sigma'$. Hence, $A$ is ample. Finally, Proposition~\ref{prop:8} together with \eqref{eq:15} show $ \vol(X',A)\ge\vol(X,D)-\varepsilon$.
\end{proof}

Let $\K$ be a global field and $X$ a proper toric variety over $\K$ of dimension
$n$. 
For each place $v\in \mathfrak{M}_{\K}$, we associate to
the algebraic torus $\T$ an analytic space $\T^{\an}_{v}$. We denote
by $\SS^{\an}_{v}$ its compact torus. In the Archimedean case, it is
isomorphic to~$(S^{1})^{n}$. In the non-Archimedean  case, it is a
compact analytic group, see 
\cite [\S 4.2]{BurgosPhilipponSombra:agtvmmh} for a description. We
denote by $\val_v\colon X_{0,v}^\an\to N_\R$ the valuation map associated to
the place $v$ as in~\cite[(4.1.2)]{BurgosPhilipponSombra:agtvmmh}.

\begin{defn} \label{def:15} 
A $v$-adically metrized $\R$-divisor $\ov D=(D,\Vert\cdot\Vert)$ on $X$ is
  \emph{toric} if $D$ is a toric $\R$-divisor and its Green function is
  invariant with respect to the action of~$\SS^{\an}_{v}$. 
To a toric $v$-adically metrized $\R$-divisor
  $\ov D$ as above, we associate the
  function $\psi_{\ov D}\colon N_{\R}\to \R$ defined, for $u\in
  N_{\R}$, by
\begin{displaymath}
  \psi_{\ov D}(u)=-g_{\ov D}(p)
\end{displaymath}
for any $p\in X^{\an}_{0,v}$ such that $\val_{v}(p)=u$.  We will
alternatively denote this function by $ \psi_{D,\Vert\cdot\Vert}$.

A quasi-algebraic metrized $\R$-divisor $\ov D$ on $X$ is \emph{toric}
if $(D,\Vert\cdot\Vert_{v})$ is a toric $v$-adically metrized
$\R$-divisor for all $v\in\mathfrak M_{\K}$. For each place $v$, we
denote by $\psi_{\ov D,v}$ the function associated to the toric 
$v$-adically metrized $\R$-divisor $(D,\|\cdot\|_{v})$. A toric
quasi-algebraic metrized $\R$-divisor is also called a \emph{toric metrized
  $\R$-divisor}, for short.

 We denote by $\wh\Car_{\T}(X)_{\R,v}$ and $\wh\Car_{\T}(X)_{\R}$ the
 spaces of toric $v$-adically metrized $\R$-divisors on $X$ and of
 toric metrized $\R$-divisors on $X$.
\end{defn}

Let $D_{i}\in \wh\Car_{\T}(X)_{\R}$ and $\alpha_{i}\in \R$, $i=1,2$,
and $v\in \mathfrak M_{\K}$. It is immediate from the definitions
that 
\begin{equation}
  \label{eq:29}
\psi_{\alpha_{1}\ov D_{1}+\alpha_{2}\ov D_{2},v}  =\alpha_{1}\psi_{\ov D_{1},v}  +
\alpha_{2}\psi_{\ov D_{2},v}.  
\end{equation}

\begin{rem} \label{rem:12} When $\ov D$ is the metrized $\R$-divisor associated
  to the toric line bundle with section $(\ov L,s)$, the function
  $\psi_{\ov D,v}$ corresponds to 
  the function $\psi_{\ov L,s,v}$ in the notation of \cite[Definition
  4.3.5]{BurgosPhilipponSombra:agtvmmh}.
\end{rem}

Let $D$ be a toric divisor on $X$ and $v\in \mathfrak M_{\K}$.  Recall
that $D$ has a canonical $v$-adic metric, denoted
$\Vert\cdot\Vert_{D,v,\can}$ \cite[Proposition-Definition
4.3.15]{BurgosPhilipponSombra:agtvmmh}.  The function associated to this
metric agrees with the virtual support function $\Psi_{D}$.  We extend
this construction to toric $\R$-divisors.

\begin{defn}\label{def:25}
  Let $D$ be a toric $\R$-divisor on $X$ and $v\in \mathfrak
  M_{\K}$. Write $D=\sum_{i}\alpha_{i}D_{i}$ with $\alpha_{i}\in \R$
  and $D_{i}$ a toric divisor. For each $i$, let
  $\Vert\cdot\Vert_{D_{i},v,\can}$ be the canonical $v$-adic metric on
  $D_{i}$ and write $\ov
  D_{i,v,\can}=(D_{i},\Vert\cdot\Vert_{D_{i},v,\can})$. We define
  \begin{displaymath}
    \ov  D_{v,\can}=\sum_{i}\alpha_{i}\ov  D_{i,v,\can},
  \end{displaymath}
  and we denote by $\Vert\cdot\Vert_{D,v,\can}$ the corresponding
  \emph{$v$-adic canonical metric} on $D$.  This is a toric $v$-adic
  metric on $D$ and $\psi_{ \ov D_{v,\can}}=\Psi_{D}$. In particular,
  it is independent of the chosen decomposition of $D$.  
\end{defn}

The following result extends \cite[Proposition 4.3.10(2) and
Proposition 4.9.2(1)]{BurgosPhilipponSombra:agtvmmh} to toric $\R$-divisors.
As explained in \cite[\S 4.1]{BurgosPhilipponSombra:agtvmmh}, the
variety with corners $N_{\Sigma}$ is a compactification of the vector
space $N_{\R}$ and, for each $v\in \mathfrak{M}_{\K}$,  there is a
proper map  of topological spaces $\val_{v}\colon
X_{v}^{\an}\to N_{\Sigma}$. 

\begin{prop}\label{prop:22}
Let $D$ be a toric $\R$-divisor on $X$.
\begin{enumerate}
\item \label{item:70} Let $v\in \mathfrak M_{\K}$. The correspondence
  $\Vert\cdot\Vert \mapsto \psi_{D,\Vert\cdot\Vert}
$ is a bijection between the set of toric $v$-adic metrics on $
D$ and the set of functions $\psi\colon N_{\R}\to \R$ such that
$\psi-\Psi_{D}$ extends to a 
continuous function on $N_{\Sigma}$.
\item  \label{item:71}
The correspondence $\{\Vert\cdot\Vert_{v}\}_{v\in  \mathfrak
  M_{\K}}\mapsto \{\psi_{D,\Vert\cdot\Vert_{v}}\}_{v\in  \mathfrak
  M_{\K}}$  is a bijection between the set of quasi-algebraic toric
metrics on $D$ and the set of families of functions $\psi_{v}\colon
N_{\R}\to \R$  
such that $\psi_{v}-\Psi_{D}$ extends to a
continuous function on $N_{\Sigma}$ for all $v$, and 
$\psi_{v}=\Psi_{D}$ for all but a finite number of
$v$. 
\end{enumerate}
\end{prop}

\begin{proof} 
For the local case, Corollary \ref{cor:7} together with the properties 
of the canonical $v$-adic metric of $D$ implies that 
the map 
\begin{displaymath}
  \Vert\cdot\Vert\longmapsto g_{D,\Vert\cdot\Vert}-g_{{D,v,\can}}
\end{displaymath}
gives a bijection between the space of toric $v$-adic metrics on $D$
and that of continuous $\SS_{v}^{\an}$-invariant functions on
$X_{v}^{\an}$. We have that
$g_{D,\Vert\cdot\Vert}-g_{{D,v,\can}}=-(\psi_{D,\Vert\cdot\Vert}-\Psi_{D})\circ\val_{v}$.
Since the map $\val_{v}\colon X_{v}^{\an}\to N_{\Sigma}$ is proper,
the $\SS_{v}^{\an}$-invariant functions on $X_{v}^{\an}$ are in
one-to-one correspondence with continuous functions on $N_{\Sigma}$.

The global case follows from this and \cite[Proposition 4.9.2(1)]{BurgosPhilipponSombra:agtvmmh}.
\end{proof}

Let $\ov D$ be a toric metrized $\R$-divisor on $X$ and $v\in
\mathfrak M_{\K}$. By Proposition \ref{prop:22}\eqref{item:70}, the
function $\psi_{\ov D,v}$ is asymptotically conic (Definition
\ref{def:11}) and its stability set is $\Delta_{D}$, since it
agrees   with that of $\Psi_{D}$.

\begin{defn}\label{def:8} Let $\ov D$ be a toric metrized $\R$-divisor
  on $X$. For each $v\in \mathfrak{M}_{\K}$, the \emph{$v$-adic roof
    function} of $\ov D$ is the concave function on $\Delta_{D}$ defined as
$$
\vartheta
_{\ov D,v}=\psi _{\ov D,v}^{\vee},
$$ 
see Definition \ref{def:4}. The \emph{(global) roof function} of $\ov
D$ is defined as
$$
\vartheta _{\ov
  D}=\sum_{v\in \mathfrak{M}_{\K}}n_{v}\vartheta _{\ov D,v}. 
$$
\end{defn}

\begin{rem} \label{rem:3} In \cite[Definition
  5.1.4]{BurgosPhilipponSombra:agtvmmh}, the
$v$-adic roof function is defined only for toric semipositive metrized
  line bundles with a toric section. Definition~\ref{def:8} extends this definition to arbitrary
  toric metrized $\R$-divisors.
\end{rem}

The following extends \cite[Theorem 4.8.1 and
Proposition 4.9.2(2)]{BurgosPhilipponSombra:agtvmmh} to toric
$\R$-divisors.    

\begin{prop}\label{prop:23}
  Let $D$ be a toric $\R$-divisor on $X$. 
  \begin{enumerate}
  \item \label{item:72} Let $v\in \mathfrak M_{\K}$. The
    maps $\Vert\cdot\Vert \mapsto \psi_{D,\Vert\cdot\Vert}$
    and $\Vert\cdot\Vert\mapsto \vartheta_{D,\Vert\cdot\Vert}$ are
    bijections between the set of toric  semipositive $v$-adic metrics
    on $ D$ and, on one hand,
    the set of concave functions $\psi\colon N_{\R}\to \R$ such that
    $|\psi-\Psi_{D}|$ is bounded and, on the other hand, the set of
    continuous concave functions on $\Delta_{D}$.

  \item \label{item:73} The maps
    $\{\Vert\cdot\Vert_{v}\}_{v}\mapsto
    \{\psi_{D,\Vert\cdot\Vert_{v}}\}_{v} $ and
    $\{\Vert\cdot\Vert_{v}\}_{v}\mapsto
    \{\vartheta_{D,\Vert\cdot\Vert_{v}}\}_{v}$ are
    bijections between the set of toric semipositive  metrics on $ D$
    and, on one hand, the set of families of concave functions
    $\{\psi_{v}\colon N_{\R}\to \R\}_{v}$ such that
    $|\psi_v-\Psi_{D}|$ is bounded for all $v$ and $\psi_{v}=\Psi_{D}$
    for all but a finite number of $v$ and, on the other hand, the set
    of families of continuous concave functions $\{\vartheta_{v}\colon
    \Delta_{D}\to\R\}_{v}$ such that $\vartheta_{v}\equiv 0$ for all
    but a finite number of $v$.
  \end{enumerate}
\end{prop}

\begin{proof}
By \cite[Propositions 2.5.20(2) and
2.5.23]{BurgosPhilipponSombra:agtvmmh}, the first and the second part
of both statements are equivalent.

  For the local case, given a toric semipositive $v$-adic metric $\Vert\cdot\Vert$ on $D$, choose $\varepsilon>0$ and let
  $\Vert\cdot\Vert'$ be a (non necessarily toric) semipositive
  (smooth or algebraic) metric on $D$ such that
\begin{displaymath}
  \dist(\Vert\cdot\Vert,\Vert\cdot\Vert') <\varepsilon.
\end{displaymath}
Set $\ov D'=(D,\Vert\cdot\Vert')$ and write $\ov
D'=\sum_{i}\alpha_{i}\ov D_{i}'$ with $\alpha_{i}>0$ and $\ov D_{i}'$
a semipositive  (smooth or algebraic) metrized divisor on $X$.

For each $i$, let $\Vert\cdot\Vert'_{i,\tor}$ be the toric 
metric on $D_{i}'$ obtained from the metric of $\ov D_{i}'$ by
averaging of the metric along the fibres of the map
$\val_{v}\colon X_{0,v}^\an\to N_{\Sigma}$ as in
\cite[Definition 4.3.3]{BurgosPhilipponSombra:agtvmmh}.  Write $\ov
D_{i,\tor}'=(D_{i},\Vert\cdot\Vert'_{i,\tor})$. This is a toric
semipositive  (smooth or algebraic) metrized divisor and so
\begin{displaymath}
  \ov D_{\tor}':=\sum_{i}\alpha_{i}\ov D_{i,\tor}'
\end{displaymath}
is a toric semipositive  (smooth or algebraic) metrized divisor.  Since $\Vert\cdot\Vert$ is
$\SS_{v}^{\an}$-invariant, it results that
\begin{math}
  \dist(\Vert\cdot\Vert,\Vert\cdot\Vert'_{\tor}) <\varepsilon.
\end{math}
Hence, 
\begin{displaymath}
\sup_{u\in N_{\Sigma}} \big|\psi_{\ov D}(u)-\psi_{\ov D_{\tor}'}(u)\big|<
{\varepsilon}.  
\end{displaymath}
By propositions 4.4.1 and 4.7.1 in {\it loc. cit.}, the functions
$\psi_{\ov D_{i,\tor}'}$ are concave, and so is $\psi_{\ov D_{\tor}'}$
since $\alpha_{i}>0$ for all $i$.  Letting $\varepsilon \to 0$, we
conclude that $\psi_{\ov D}$ is concave.

Conversely, let $\psi\colon N_{\Sigma}\to \R$ be a concave
function such that $\psi-\Psi_{D}$ is bounded. Denote by
$\Vert\cdot\Vert$ the toric $v$-adic metric associated to $\psi$ by 
Proposition~\ref{prop:22}\eqref{item:70} and let $\varepsilon >0$. 
By \cite[Proposition 2.5.23(2)]{BurgosPhilipponSombra:agtvmmh}, there is a 
piecewise affine concave function $\zeta\colon \Delta_{D}\to \R$ such that
\begin{displaymath}
  \sup_{x\in \Delta_{D}} \big|\psi^{\vee}(x)-\zeta(x)\big|< \varepsilon.
\end{displaymath}
By the density of $\Q$, we can choose $\zeta$ defined by a finite number
of affine maps with rational coefficients.  Its Legendre-Fenchel dual
$\phi=\zeta^{\vee}$ is a piecewise affine concave function on a
rational polyhedral complex $\Pi$ such that
\begin{displaymath}
  \sup_{u\in N_{\Sigma}} \big|\psi(u)-\phi(u)\big|<  \varepsilon.
\end{displaymath}
Arguing as in the proof of \cite[Theorem
3.7.3]{BurgosPhilipponSombra:agtvmmh}, we can 
assume that the recession of $\Pi$ agrees with $\Sigma$.

Let $\cP(\Pi)$ denote the space of piecewise affine functions on $\Pi$ and
$\cC(\Pi)\subset \cP(\Pi)$ the subset of concave functions.  The space $\cP(\Pi)$ can
be identified with the linear subspace of $\prod_{\Lambda\in \Pi^{n}}(M_{\R}\times\R)$ given by
\begin{displaymath}
   \big\{ (m_{\Lambda},\gamma_{\Lambda})_{\Lambda\in \Pi^{n}} \mid
   \langle m_{\Lambda},u\rangle+ \gamma_{\Lambda} =\langle
   m_{\Lambda'},u\rangle+ \gamma_{\Lambda'} \text{ for all } \Lambda,\Lambda'\in
   \Pi^{n} \text{ and } u\in
   \Lambda\cap \Lambda'\big\}.
\end{displaymath}
This is a finite-dimensional linear subspace defined over $\Q$. For each $\Lambda\in
\Pi^{n}$, choose a point $u_{\Lambda}\in \ri(\Lambda)\cap N_{\Q}$. 
Then
\begin{displaymath}
 \cC(\Pi)  =\big\{ (m_{\Lambda},\gamma_{\Lambda})_{\Lambda}\in \cP(\Pi) \mid
   \langle m_{\Lambda},u_{\Lambda}\rangle+ \gamma_{\Lambda} \le\langle
   m_{\Lambda'},u_{\Lambda}\rangle+ \gamma_{\Lambda'} \text{ for all } \Lambda,\Lambda'\in
   \Pi^{n}\big\}.
\end{displaymath}
This is a cone of $\cP(\Pi)$ given by a rational H-representation as
in \cite[(2.1.1)]{BurgosPhilipponSombra:agtvmmh}. Hence, it admits a
rational V-representation. This implies that there exist a finite
number of H-lattice piecewise affine concave functions $\phi_{i}\in
\cC(\Pi)$ and positive real numbers $\alpha_{i}$ such that
\begin{displaymath}
  \phi=\sum_{i}\alpha_{i}\phi_{i}.
\end{displaymath}
For each $i$, set $\Phi_{i}=\rec(\phi_{i})$, the recession function of
$\phi_{i}$ (Definition \ref{def:11}). This is a support
function on $\Sigma$ and so it corresponds to a toric divisor on $X$,
that we denote by $D_{i}$. We observe that $\sum_i\alpha_i\Phi_i = \Psi_D$, hence $D=\sum_i\alpha_iD_i$ by Proposition \ref{prop:15}.

In the non-Archimedean case, \cite[Corollary
4.5.9]{BurgosPhilipponSombra:agtvmmh} implies 
that there exists a semipositive algebraic metric $\Vert\cdot
\Vert_{i}$ on $D_{i}$ such that $\psi_{D_{i},\Vert\cdot
\Vert_{i}}=\phi_{i}$. Set $\ov D_{i}=(D_{i},\Vert\cdot
\Vert_{i})$ and $\ov D'=\sum_{i}\alpha_{i}\ov D_{i}$. This gives a
semipositive algebraic $v$-adic metric $\Vert\cdot \Vert'$ on $D$
and 
\begin{displaymath}
  \dist(\Vert\cdot\Vert,\Vert\cdot \Vert') =
  \sup_{u\in N_{\Sigma}} \big|\psi(u)-\phi(u)\big|<
  \varepsilon.
\end{displaymath}
In the Archimedean case, Theorem 4.8.1 in {\it loc. cit.} implies that
there exists a semipositive smooth metric $\Vert\cdot \Vert_{i}$
on $D_{i}$ such that $|\phi_{i}- \psi_{D_{i},\Vert\cdot
  \Vert_{i}}|<\varepsilon/\sum_i\alpha_i$. Let $\Vert\cdot\Vert'$ denote
the induced semipositive smooth metric on $D$. Then
\begin{displaymath}
  \dist(\Vert\cdot\Vert,\Vert\cdot\Vert') =
  \sup_{u\in N_{\Sigma}} \big|\psi(u)- \sum_i\alpha_i\psi_{D_{i},\Vert\cdot
\Vert_{i}}(u)\big|< 
  2\varepsilon.
\end{displaymath}
In both cases, as $\varepsilon$ tends to $0$, this and Definition \ref{def:24} show that the metric $\Vert\cdot\Vert$ is semipositive.

The global statement follows from the local one and Proposition \ref{prop:22}\eqref{item:71}.
\end{proof}

Let $Y$ be a $d$-dimensional cycle on $X$ and $\ov D_{i}$, $i=0,\dots,
d$, semipositive toric metrized $\R$-divisors on~$X$. For a place
$v\in \mathfrak M_{\K}$, the \emph{$v$-adic toric height} of $Y$ with
respect to $\ov D_{0},\dots,\ov D_{d}$, denoted $\h^{\tor}_{v,\ov
  D_{0},\dots,\ov D_{d}}(Y)$, is defined from the $v$-adic toric
height of a cycle (\cite[Definition
5.1.1]{BurgosPhilipponSombra:agtvmmh}) with respect to semipositive
toric line bundles by multilinearity and continuity. 

The global height of $Y$ can be computed as
\begin{displaymath}
  \h_{\ov D_{0},\dots,\ov
  D_{d}}(Y) = \sum_{v\in \mathfrak M_{\K}}n_{v }\htor_{v,\ov D_{0},\dots,\ov
  D_{d}}(Y).
\end{displaymath}
This follows easily from the analogous statement for semipositive
toric line bundles, see Proposition 5.2.4 in {\it loc. cit.}.

Recall that there is a one-to-one, dimension reversing, correspondence
between the cones of $\Sigma$ and the orbits of the action of $\T$ on
$X$. Given a cone $\sigma\in \Sigma$, we denote by $O(\sigma)$ the
corresponding orbit and by $V(\sigma)$ its closure.  Recall also that,
given a support function $\Psi$ on $\Sigma$, we can associate to each
cone $\sigma\in \Sigma$ a face, denoted $F_{\sigma}$, of the polytope
$\stab(\Psi)$.

\begin{prop} \label{prop:1}
Let $\ov D$ be a toric semipositive  metrized $\R$-divisor on
$X$ and $\sigma\in \Sigma$ a cone of codimension $d$. Then, for $v\in \mathfrak M_{\K}$, 
\begin{displaymath}
  \h_{v, \ov {   D}}^{\tor}(V(\sigma))=(d+1)!\int _{F_{\sigma}}\vartheta _{\ov D,v}\dd\vol_{M(F_{\sigma})}
\end{displaymath}
and
\begin{displaymath}
 \h_{\ov {   D}}(V(\sigma))=(d+1)!\int _{F_{\sigma}}\vartheta _{\ov D}\dd\vol_{M(F_{\sigma})},
\end{displaymath}
where $M(F_{\sigma})$ is the lattice induced by $M$ on the affine
space generated by $F_{\sigma}$ and $\vol_{M(F_{\sigma})}$ is the Haar
measure on that affine space such that the covolume of $M(F_{\sigma})$
is one.

In particular,
\begin{displaymath}
  \h_{\ov {   D}}(X)=(n+1)!\int _{\Delta _{D}}\vartheta _{\ov D}\dd\vol_{M}.
\end{displaymath}
\end{prop}

\begin{proof}
  For the local case, let $v\in \mathfrak M_{\K}$. Consider first the
  case when the metric $\Vert\cdot\Vert_{v}$ is smooth semipositive,
  if $v$ is Archimedean, or algebraic semipositive, if $v$ is
  ultrametric.  Write
\begin{displaymath}
  \ov D=\sum_{i}\alpha_{i}\ov D_{i}
\end{displaymath}
with $\ov D_{i}$ a $\T$-divisor on $X$ with a  semipositive (smooth or
algebraic) metric, and $\alpha_{i}>0$. 
For each $\ov D_{i}$, the formula follows from 
\cite[Proposition 5.1.11]{BurgosPhilipponSombra:agtvmmh}. 
It follows for $\ov D$ by the multilinearity of the toric local height
 and of
the mixed integral of concave functions with respect to the
sup-convolution. 

For a semipositive $\ov D$, the formula follows from the continuity
of the toric local height with respect to $\dist$, which follows from
{\it loc. cit.}, Theorem 1.4.17(4), and from the continuity of the
Legendre-Fenchel duality with respect to the uniform convergence ({\it
  loc. cit.}, Proposition 2.2.3).

The global case follows by adding up  all local
toric heights.  
\end{proof}

\begin{rem} \label{rem:4}
  As in \cite[Corollary 5.1.9 and Theorem
  5.2.5]{BurgosPhilipponSombra:agtvmmh} one can
  express the mixed $v$-adic toric height and the mixed global height
  of toric metrized $\R$-divisors in terms of mixed integrals. 
\end{rem}

\section{Small sections, arithmetic and $\chi$-arithmetic volumes of
  toric metrized $\R$-divisors}\label{sec:small-sect-arithm}

In this section, we give formulae for the arithmetic and the
$\chi$-arithmetic volumes of a toric metrized $\R$-divisor in terms of
its global roof function.  We keep the notations of the previous
section.  In particular, $X$ is a proper toric variety of dimension
$n$ over a global
field $\K$.

We show first how the positivity of the roof function determines the
existence of small toric sections.

\begin{prop} \label{prop:3} Let $\ov D$ be a toric metrized
  $\R$-divisor on $X$.  For $\ell\ge1 $ and $m\in \ell \Delta_{D}$ we
  set $s_{m}\in\Gamma(X,\ell D)_{\R}^{\times}$ for the corresponding
  $\R$-section.
  \begin{enumerate}
  \item \label{item:13} Let $v\in \mathfrak{M}_{\K}$. Then 
  \begin{math}
    -\log \|s_{m}\|_{v,\sup}=\ell \vartheta_{\ov D,v}(m/\ell).
  \end{math}
\item \label{item:14} If  $\vartheta_{\ov D}(m/\ell)>0$,
  then there exists $e\ge 1$ and $\gamma \in \K^{\times}$
  such that $\gamma s_{m}^{e}\in \wh\Gamma(X,e \ell \ov D)^{\times}_{\R}$.
  \end{enumerate}
\end{prop}

\begin{proof}
  We compute
  \begin{align*}
    \vartheta _{\ov D,v}(m/\ell)&=\psi _{\ov D,v}^{\vee}(m/\ell)\\
    &=\inf_{u\in N_{\R}}(\langle m/\ell,u \rangle -\psi
    _{\ov D,v}(u))\\
    &=\frac{1}{\ell}\inf_{x\in X_{0}}\big(-\log |\chi^{m}(x)|_{v}+
    {g_{\ell \ov D,v}(x)}\big) \\
    &=\frac{-1}{\ell}\sup_{x\in X_{0}}\log \|s_{m}(x)\|_{v}\\
    &=\frac{-1}{\ell}\log \|s_{m}\|_{v,\sup},
  \end{align*}
  which proves the first statement. The second statement follows from
  the first, using Lemma \ref{lemm:10} as in the proof of Proposition
  \ref{prop:2}.
\end{proof}

Next we show that, for every place, the basis of toric global sections
is orthogonal 
with respect to the sup-norm.

\begin{prop} \label{prop:4} Let $\ov D$ be a toric
  metrized $\R$-divisor. Then, for all $v\in \mathfrak{M}_{\K}$, the set
  $\{s_{m}\}_{m\in \Delta_{D} \cap M}$ is an orthogonal basis of
  $\LL(D) $ with respect to the norm $\|\cdot\|_{v,\sup}$.
\end{prop}

\begin{proof}
  Fix an integral basis of $N$ so that there is an isomorphisms $N\simeq
  \Z^{n}\simeq M$.  Let $v\in \mathfrak{M}_{\K}$ and choose an integer
  $r\ge 1$ such that $\Delta_{D} \subset [-r,r]^{n}$ and, if $v$ is
  non-Archimedean, such that $|2r+1|_{v}=1$. After possibly taking a
  finite extension of $\K$, we may assume that $\K_{v}$ contains a
  primitive root of $1$ of order $2r+1$, which we denote by $\omega
  $. For every $a=(a_{1},\dots,a_{n})\in [-r,r]^{n}\cap N$ we consider
  the element $t_{a}=(\omega ^{a_{1}},\dots,\omega ^{a_{n}})\in
  \SS_{v  }^{\an}$.  Such an element determines an automorphism of
  $X_{{v}}^{\an}$.  Let $s=(f,D)\in \LL(D)$ be a global section.  Then
  $t_{a}^{*}s\in \LL(D)$ since $D$ is toric. For $p\in X_{v}^{\an}$,
  \begin{displaymath}
    \Vert t_{a}^{\ast}s\Vert _{v}(p)= |f(t_{a}* p)|_{v} \e^{-g_{\ov D,v}(p)}=|f(t_{a}* p)|_{v} \e^{-g_{\ov D,v}(t_{a}*p)}=\Vert s\Vert _{v}(t_{a}*p)
  \end{displaymath}
because of the $\SS_{v}^{\an}$-invariance of $\Vert \cdot \Vert _{v}$.
Hence, 
  \begin{equation}
    \label{eq:1}
    \|t_{a}^{\ast}s\|_{v,\sup}= \sup_{p\in X_{v}^{\an}}\Vert t_{a}^{\ast}s\Vert _{v}(p)= \sup_{p\in X_{v}^{\an}}\Vert s\Vert _{v}(t_{a}*p)
    =  \|s\|_{v,\sup}.
  \end{equation}
Write 
  \begin{displaymath}
    s=\sum_{m\in \Delta_{D}\cap M }\gamma _{m}s_{m}   
  \end{displaymath}
  with $\gamma_{m}\in \K$. For convenience, we also put $\gamma _{m}=0$
  for $m\in ([-r,r]^{n}\setminus \Delta_{D}) \cap M$.  Then, for $a\in
  [-r,r]^n\cap N$, 
\begin{displaymath}
  t_{a}^{*}s=\sum_{m\in [-r,r]^{n}\cap M}\gamma_{m}t_{a}^{*}s_{m}= 
\sum_{m}\chi^{m}(t_{a})\gamma_{m}s_{m}=
\sum_{m}\omega^{\langle a,m\rangle}\gamma_{m}s_{m}.
\end{displaymath}
Consider the matrix $\Omega =(\omega ^{\langle a,m\rangle})_{a,m}$. Its inverse is given by
\begin{displaymath}
  \Omega ^{-1}=(2r+1)^{-n}\big(\omega ^{-\langle a,m\rangle}\big)_{m,a}.
\end{displaymath}
Therefore, for $m\in [-r,r]^{n}\cap M$,
\begin{displaymath}
  \gamma _{m}s_{m}=\sum_{a\in [-r,r]^{n}\cap N}(2r+1)^{-n}\omega
  ^{-\langle a,m\rangle} t_{a}^{*}s.
\end{displaymath}
If $v$ is Archimedean, 
\begin{displaymath}
  \|\gamma _{m}s_{m}\|_{v,\sup}\le (2r+1)^{n}\max_{a}\|(2r+1)^{-n}\omega
  ^{-\langle a,m\rangle} t_{a}^{*}s\|_{v,\sup}=\max_{a}\|t_{a}^{*}s\|_{v,\sup}.
\end{displaymath}
while, if $v$ is non-Archimedean,
\begin{displaymath}
  \|\gamma _{m}s_{m}\|_{v,\sup}\le \max\|(2r+1)^{-n}\omega
  ^{-\langle a,m\rangle} t_{a}^{*}s\|_{v,\sup}=\max_{a}\|t_{a}^{*}s\|_{v,\sup}.
\end{displaymath}
In both cases, we deduce from \eqref{eq:1} that 
\begin{displaymath}
  \max_{m\in \Delta_{D}\cap M }\|\gamma _{m}s_{m}\|_{v,\sup}\le \|s\|_{v,\sup}.
\end{displaymath}
Hence,  $\{s_{m}\}_{m\in \Delta_{D} \cap M}$ is an orthogonal basis
of $ \LL(D)$ with respect to $\|\cdot\|_{v,\sup}$ for all~$v$.
\end{proof}

\begin{cor} \label{cor:4} Let $\ov D$ be a toric $\R$-divisor on $X$ and
  $s\in \LL(D)$. Write
  \begin{displaymath}
    s=\sum_{m\in \Delta_{D}\cap M }\gamma _{m}s_{m}
  \end{displaymath}
  with $\gamma_{m}\in \K$. If $v$ is Archimedean,
  then
  \begin{displaymath}
    \max_{m\in \Delta_{D}\cap M }\|\gamma _{m}s_{m}\|_{v,\sup}\le \|s\|_{v,\sup}
    \le \#(\Delta_{D} \cap M) \max_{m\in \Delta_{D}\cap M }\|\gamma _{m}s_{m}\|_{v,\sup},
  \end{displaymath}
while, if $v$ is non-Archimedean,
  \begin{displaymath}
    \|s\|_{v,\sup}=\max_{m\in \Delta_{D}\cap M }\|\gamma _{m}s_{m}\|_{v,\sup}.
  \end{displaymath}
\end{cor}

\begin{proof} The upper bound for $\|s\|_{v,\sup}$ follows from the
  triangle inequality, in the Archimedean case, and from the ultrametric
  inequality, in the non-Archimedean case.
  The remaining inequality follows from Proposition \ref{prop:4}.
\end{proof}

\begin{cor}\label{cor:1}  Let $\ov D$ be a  toric metrized $\R$-divisor and
  $s\in \whL(\ov D)$. Write
  \begin{displaymath}
    s=\sum_{m\in
    \Delta_{D}\cap M} \gamma _{m}s_{m}
  \end{displaymath}
  with $\gamma_{m}\in \K$. Then $\gamma _{m}s_{m}\in \whL(\ov D)$ for all $m\in \Delta_{D}\cap M$.
\end{cor}

\begin{proof}
This follows immediately from Corollary \ref{cor:4}.
\end{proof}
 
Let $\ov D$ be a toric $\R$-divisor on $X$ and consider the set
\begin{displaymath}
  \Theta _{\ov D}=\{x\in \Delta_{D} \mid \vartheta _{\ov D}(x)\ge 0\}.
\end{displaymath}
If $\Theta _{\ov D}\not = \emptyset$, it is a convex subset of $\Delta_{D}$.

\begin{thm}\label{thm:1} Let $\ov D$ be a toric 
  metrized $\R$-divisor. Then  
  the arithmetic volume of $\ov D$ is given by
  \begin{displaymath}
    \avol(X,\ov D)=
    (n+1)!\int _{\Delta_{D} }\max(0,\vartheta_{\ov D}) \dd\vol_{M}=(n+1)!\int _{\Theta _{\ov D}}\vartheta_{\ov D} \dd\vol_{M},
  \end{displaymath}
  while the $\chi$-arithmetic volume of $\ov D$ is given by
  \begin{displaymath}
    \wh \vol_{\chi}(X,\ov D)=(n+1)!\int _{\Delta_{D} }\vartheta_{\ov D} \dd\vol_{M}.    
  \end{displaymath}
\end{thm}

\begin{proof}
 We consider first the case when $\K$ is a number field. Let $\ell
  \ge 1$ and set $\Delta=\Delta_{D}$ for short. For each $m\in \ell\Delta \cap M$, we consider the
  $\mathfrak{M}_{\K}$-divisors $\mathfrak{c}_{m}$ and
  $\mathfrak{c}_{m}'$ given by, for $v|\infty$,
  \begin{displaymath}
 \log(   \mathfrak{c}_{m,v})=
    -\log\|s_{m}\|_{v,\sup} , \quad 
\log(   \mathfrak{c}'_{m,v})=
-\log\|s_{m}\|_{v,\sup}-\log( \#(\ell\Delta \cap M))
   \end{displaymath}
and, for $v\nmid\infty$, 
  \begin{displaymath}
 \log(   \mathfrak{c}_{m,v})= \log(   \mathfrak{c}'_{m,v})=
      \lambda _{v}\Big\lfloor \frac{-\log\|s_{m}\|_{v,\sup}}{\lambda _{v}}\Big\rfloor.
   \end{displaymath}
By Corollary~\ref{cor:4},
\begin{displaymath}
 \bigoplus_{m\in \ell\Delta\cap
    M} \wh\LL(\mathfrak{c}'_{m}) s_{m}\subset 
 \whL(\ell \ov D)\subset \bigoplus_{m\in \ell\Delta\cap
    M} \wh\LL(\mathfrak{c}_{m}) s_{m}.
\end{displaymath}
Observe that, by the quasi-algebricity of the metric, there exists a
finite set $S\subset \mathfrak{M}_{\K}$ such that, for all $v\notin
S$, we have that $\mathfrak{c}_{m,v}=\mathfrak{c}'_{m,v}=1$ for all
$\ell\ge1$ and $m\in \ell\Delta\cap M$. With Proposition
\ref{prop:3}\eqref{item:13}, this implies that
\begin{displaymath}
\wh\deg(\mathfrak{c}_{m})= d_{\K}\ell \vartheta _{\ov D}(m/\ell)+O(1),
\quad 
\wh\deg(\mathfrak{c}'_{m})=d_{\K}\ell \vartheta _{\ov D}(m/\ell)+O(\log(\ell)).
\end{displaymath}
Applying Lemma \ref{lemm:2}, we deduce that
\begin{displaymath}
\Big|    \whl(\ell\ov D)- 
  d_{\K}\ell\sum_{m\in \ell \Delta \cap M}\max(0,\vartheta _{\ov
    D}(m/\ell))\Big| = O(\ell ^{n}\log(\ell)).
\end{displaymath}
Therefore,
\begin{align*}
  \avol(X,\ov D)=&
  \frac{1}{d_{\K}}\limsup_{\ell\to \infty} \frac{ \whl(\ell\ov D)}{\ell^{n+1}/(n+1)!}\\
  =&(n+1)!\limsup_{\ell\to \infty}\bigg(\frac{1}{\ell^{n}} \sum_{m\in
    \ell\Delta \cap
      M}\max(0,\vartheta _{\ov 
      D}(m/\ell))+ O\bigg(\frac{\log(\ell)}{\ell}\bigg)\bigg)\\=&
  (n+1)!\int _{\Delta }\max(0,\vartheta_{\ov D}) \dd\vol_{M}.
\end{align*}
We next prove the formula for the $\chi$-arithmetic volume.  Using
propositions \ref{prop:4}, \ref{prop:7} and
\ref{prop:3}\eqref{item:13}, we see that
\begin{align*}
\frac{1}{d_{\K}}\wh\chi(\LL(\ell\ov D))  &=  \sum_{v}
  \sum _{m\in \ell\Delta \cap M}n_{v}\log(\|s_{m}\|_{v,\sup}^{-1})+
  O(\# (\ell\Delta\cap M)\log(\# (\ell\Delta\cap M)))\\
  &= \sum_{v} \sum _{m\in \ell\Delta \cap M}n_{v} \ell \vartheta_{\ov
   D,v}
  (m/\ell) + O(\ell^{n}\log (\ell))\\
  &= \sum _{m\in \ell\Delta \cap M}\ell \vartheta_{\ov D} (m/\ell) +
  O(\ell^{n}\log (\ell)).
\end{align*}
Therefore,
\begin{align*}
  \wh\vol_{\chi}(X,\ov D)&=
  (n+1)!\limsup_{\ell\to \infty}\bigg(\frac{1}{\ell^{n}} \sum_{m\in \ell\Delta \cap
    M}\vartheta _{\ov 
    D}(m/\ell)+ O\bigg(\frac{\log(\ell)}{\ell}\bigg)\bigg)\\
  &=(n+1)!\int _{\Delta
  }\vartheta _{\ov D}\dd \vol_{M}.
\end{align*}

We consider next the function field case. Let $\pi \colon B\to C$ be a
dominant morphism of smooth projective curves and consider the
function field $\K={\rm K}(B)$ with the adelic structure defined
in Example \ref{exm:2}. Let $\ell\ge 1$. 
For each $m\in \ell \Delta \cap M$, consider the
$\mathfrak{M}_{\K}$-divisor given by 
\begin{displaymath}
  \log(\mathfrak c_{m,v})= \lambda _{v}\left\lfloor \frac{-\log
      \Vert s_{m}\Vert _{v,\sup}}{\lambda _{v}}\right\rfloor  
\end{displaymath}
and let $D(\mathfrak c_{m})$ be the associated Weil divisor of $B$.
Observe that, for $\gamma\in \K$, we have that
$\log|\gamma|_{v}\le -\log\|s_{m}\|_{v,\sup}$ if and only if
$\ord_{v}(\gamma)+ \big\lfloor \frac{-\log
  \Vert s_{m}\Vert _{v,\sup}}{\lambda _{v}}\big\rfloor\ge 0$. Hence, this
condition holds for all $v$ if and only if $ \div(\gamma)+
D(\mathfrak c_{m})\ge 0$, namely, if and only if $\gamma\in \LL(D(\mathfrak c_{m}))$.
Corollary \ref{cor:4} and equation \eqref{eq:28} then show
\begin{displaymath}
  \whL(\ell\ov D) = \bigoplus_{m\in \ell \Delta \cap M}
  \LL(D(\mathfrak c_{m}))s_{m} = \bigoplus_{m\in \ell \Delta \cap M}
  \whL(\mathfrak c_{m})s_{m}.   
\end{displaymath}
As in the number field case, there is a
finite subset $S\subset \mathfrak{M}_{\K}$ independent of
$\ell$, that contains the support of $D(\mathfrak c_{m})$ for all $m$.
Therefore, Proposition \ref{prop:3}\eqref{item:13} imply
\begin{displaymath}
\wh \deg(\mathfrak c_{m}) = -d_{\K}\sum_{v}n_v\log\|s_{m}\|_{v,\sup}+O(1) =
  d_{\K} \ell \vartheta _{\ov D}(m/\ell)+O(1).
\end{displaymath}
Applying Lemma \ref{lemm:2}, we deduce that
\begin{displaymath}
\Big|    \whl(\ell\ov D)- 
  d_{\K}\ell\sum_{m\in \ell \Delta \cap M}\max(0,\vartheta _{\ov
    D}(m/\ell))\Big| = O(\ell ^{n}).
\end{displaymath}
Therefore,
\begin{math}
    \avol(X,\ov D)=
  (n+1)!\int _{\Delta }\max(0,\vartheta_{\ov D}) \dd\vol_{M},
\end{math}
as stated.  The formula for $\wh\vol_{\chi}(X,\ov D)$ can be proved using the
same approach as for the case of a number field.
\end{proof}

\begin{rem}
  Usually, the computation of the arithmetic volume is done by comparing
  the sup-norm with the $L^{2}$-norm with the help of Gromov's inequality, and then
  using the fact that the $L^{2}$-norm is a Hermitian norm, see for instance
  \cite{GilletSoule:amar}. Here, we bypass the use of the
  $L^{2}$-norm with the fact that
  the basis of toric sections is orthogonal with respect to the
  sup-norm for all places, both  Archimedean or  ultrametric. 
\end{rem}

\begin{cor}\label{cor:2}
  Let $\ov D$ be a toric metrized $\R$-divisor on $X$.
  \begin{enumerate}
  \item \label{item:17}  $\wh \vol_{\chi}(X,\ov D)\le
    \avol(X,\ov D)$, with equality when  $\vartheta _{\ov
      D}(x)\ge 0$ for all $x\in \Delta_{D} $.
\item \label{item:5} If $D$ is big, then $\wh \vol_{\chi}(X,\ov D)=
    \avol(X,\ov D)$ if and only if   $\vartheta _{\ov
      D}(x)\ge 0$ for all $x\in \Delta_{D} $.
  \item \label{item:18} If the metric is semipositive, then $\wh \vol_{\chi}(X,\ov D)=\h_{\ov
      D}(X)$.
  \end{enumerate}
\end{cor}
\begin{proof}
  This follows easily from Theorem \ref{thm:1}, together with
  Proposition \ref{prop:1} for \eqref{item:18}.
\end{proof}

The next examples shows that the $\chi$-arithmetic volume and the
height may differ when $\ov D$ is not semipositive. 

\begin{exmpl} \label{exm:4} Let $X= \P^{1}_{\Q}$ and $\ov D=\ov 0$ the
  zero  divisor equipped with a smooth toric metric at
  the Archimedean place and with the canonical metric at the
  non-Archimedean places. The associate functions are
\begin{displaymath}
  \psi_{\ov D,v}= 
  \begin{cases}
    f & \text{ if } v=\infty,\\
    0 & \text{ if } v\ne \infty,
  \end{cases}
\end{displaymath}
for a bounded smooth function $f\colon \R\to \R$. 

We have that $\Delta_{D}=\{0\}$ and so $ \wh{\vol}_{\chi}(X,\ov D)=0$
thanks to Theorem \ref{thm:1}. For the height, we first note that
$\lim_{|u|\to \infty}f'(u)=0$, which follows from the fact that the
metric corresponding to the Archimedean place is smooth. Then, the
B\'ezout formula gives
\begin{displaymath}
  \h_{\ov D}(X)= \h_{\ov D}(\div(s)) - \sum_{v} \int_{X_{v}^{\an}}\log\Vert s\Vert  \chern_{1}(\cO(D), \Vert\cdot\Vert_{v} )= \int_{\R} f f'' \dd u. 
\end{displaymath}
Integrating by parts, we obtain that 
\begin{displaymath}
  \int_{\R} f f'' \dd u= \Big[ff'\Big]^{\infty}_{-\infty} -
  \int_{\R}(f')^{2}\dd u= -\int_{\R}(f')^{2}\dd u. 
\end{displaymath}
Hence,
\begin{displaymath}
\h_{\ov D}(X)=   -\int_{\R}(f')^{2}\dd u \le 0,
\end{displaymath}
with equality if and only $f$ is constant. Since $f$ is bounded on
$\R$, this last condition is equivalent to the fact that $f$ is
concave or, equivalently, that  $\ov D$ is semipositive.
Therefore, 
\begin{equation}\label{eq:8}
\h_{\ov D}(X)\le \wh{\vol}_{\chi}(X,\ov D),
\end{equation}
with equality if and only if $\ov D$ is semipositive. 
\end{exmpl}

\begin{exmpl}\label{exm:5}
  Let $X= \P^{2}_{\Q}$ and $D=0$.  Let $g\colon \R^{2}\to
  \R^{2}$ be the support function of the standard simplex of
  $\Delta^{2}$. Let $\tau\in \R_{\ge0}$ and $f_{\tau}\colon
  \R^{2}\to \R$ be the function defined by
  $f_{\tau}(u_{1},u_{2})=g(u_{1}-\tau,u_{2}-\tau)$ for
  $(u_{1},u_{2})\in \R^{2}$.

Let $v_{1},v_{2}$ be two different places of $\Q$ and
$\tau_{1},\tau_{2}\in \R_{\ge0}$. Consider the
toric metric on $D$ given, under the correspondence in Proposition
\ref{prop:22}\eqref{item:71}, by
\begin{displaymath}
  \psi_{v}=
\begin{cases}
  g-f_{\tau_{1}} & \text{ if } v=v_{1},\\
  f_{\tau_{2}}-g & \text{ if } v=v_{2},\\
0 &\text{ otherwise}.
\end{cases}
\end{displaymath}
The obtained metric is DSP, since each $\psi_{v}$ is a
difference of concave functions.

By Theorem \ref{thm:1},  $\wh\vol_{\chi}(X,\ov D)=0$ since the polytope 
associated to $D$ is a point. 
For the height, we have $  \h_{\ov D}(X)= \htor_{v_{1},\ov D}(X)+
\htor_{v_{2},\ov D}(X)$.
By \cite[Remark 5.1.10]{BurgosPhilipponSombra:agtvmmh}, the quantity $
 \htor_{v_{1},\ov D}(X)$ can be computed as
\begin{multline}\label{eq:11}
\MI(g^{\vee}, g^{\vee},
g^{\vee})-3\MI(f_{\tau_{1}}^{\vee}, g^{\vee}, g^{\vee}) + 
3\MI(f_{\tau_{1}}^{\vee}, f_{\tau_{1}}^{\vee}, g^{\vee}) -
\MI(f_{\tau_{1}}^{\vee}, f_{\tau_{1}}^{\vee}, f_{\tau_{1}}^{\vee}),
\end{multline}
where $\MI$ denotes the mixed integral of concave functions, see
\cite[Definition~2.7.16]{BurgosPhilipponSombra:agtvmmh}.  
We have that $f_{\tau_{1}}^{\vee}(x_{1},x_{2})=\tau_{1}(x_{1}+x_{2})$ and
$g^{\vee}(x_{1},x_{2})=0$ for $(x_{1},x_{2})\in \Delta^{2}$.  Hence,
\begin{displaymath}
  \MI(g^{\vee}, g^{\vee},
  g^{\vee}) =3!\int_{\Delta^{2}}g^{\vee}=0, \quad 
 \MI(f_{\tau_{1}}^{\vee}, f_{\tau_{1}}^{\vee},
  f_{\tau_{1}}^{\vee})=3!\int_{\Delta^{2}}f_{\tau_{1} }^{\vee}= 2\tau_{1}
\end{displaymath}
and, using the definition of the mixed integral and
computing the relevant sup-convolutions, we can verify that
\begin{displaymath}
  \MI(f_{\tau_{1}}^{\vee}, g^{\vee},
  g^{\vee}) = \tau_{1},\quad \MI(f_{\tau_{1}}^{\vee}, f_{\tau_{1}}^{\vee},
  g^{\vee}) = 2\tau_{1}.
\end{displaymath}
The formula \eqref{eq:11} then implies that $\htor_{v_{1},\ov D}(X)=
\tau_{1}$. Similarly, 
 $\htor_{v_{2},\ov D}(X)=
(-1)^{3}\tau_{2}=-\tau_{2}$. Hence, 
\begin{displaymath}
  \h_{\ov D}(X)=
-\tau_{1}+\tau_{2}.
\end{displaymath}
Varying $\tau_{1}$ and $\tau_{2}$, 
this height can be any real number.
\end{exmpl}

\begin{rem} \label{rem:13} It would be interesting to see if the
  inequality \eqref{eq:8} holds for any toric DSP metrized
  $\R$-divisor on $\P^{1}_{\Q}$.  Example \ref{exm:5} shows that in
  dimension 2 and in the absence of approachability, there is no
  relation between the height and the $\chi$-arithmetic volume.
\end{rem}

\section{Positivity properties of toric metrized $\R$-divisors}
\label{sec:posit-propert-toric}

In this section, we give criteria for the different positivity conditions
for toric metrized $\R$-divisors.

\begin{thm}\label{thm:2}
  Let $\ov D$ be a toric metrized $\R$-divisor on $X$ and $\vartheta
  _{\ov D}\colon\Delta_{D}\to\R$ its global roof function.
  \begin{enumerate}
  \item \label{item:7} $\ov D$  is ample if
    and only if  $\Psi_{D} $ is strictly
    concave on $\Sigma $, the
    function $\psi _{\ov D,v}$ is concave for all $v\in \mathfrak
    M_{\K}$,  and $ \vartheta _{\ov D}(x)>0$ for all $x\in \Delta_{D}$;

  \item \label{item:8} $\ov D$ is nef if and only if $\psi _{\ov D,v}$ is concave for
    all $v\in \mathfrak M_{\K}$ and $ \vartheta _{\ov D}(x)\ge0$ for
    all $x\in \Delta_{D}$;
  \item \label{item:9} $\ov D$ is big if and only if
    $\dim(\Delta_{D})=n $ and there exists $x\in\Delta_{D}$ such that
    $\vartheta _{\ov D}(x)>0$; 
  \item \label{item:100} $\ov D$ is pseudo-effective if and only if 
    there exists $x\in\Delta_{D}$ such that $\vartheta
    _{\ov D}(x)\ge0$; 
\item \label{item:88} $\ov D$ is effective if and only if $0\in
  \Delta_{D}$ and $\vartheta_{\ov D,v}(0)\ge0$ for all $v\in
  \mathfrak{M}_{\K}$.  
  \end{enumerate}
\end{thm}

\begin{proof} Write $\Psi =\Psi _{D}$ and $\Delta
  =\stab(\Psi )$ for short.  We start by proving \eqref{item:7}. By Proposition
  \ref{prop:26}\eqref{item:53}, $D$ is  (geometrically) ample if and
  only if $\Psi $ is strictly concave on $\Sigma $ and, by Proposition
  \ref{prop:23}\eqref{item:73}, the metric of $\ov D$ is semipositive
  if and only if $\psi _{\ov D,v}$ is concave for all $v$. Thus it
  only remains to prove that, under the assumption that $D$ is ample
  and $\ov D$ semipositive, the ampleness of $\ov D$ is equivalent to
  the positivity of $\vartheta _{\ov D}$.
  
  Assume that $\ov D$ is ample.  Let $m_{0}$ be a vertex of $\Delta
  $ and $p_{0}\in
  X(\K)$ its corresponding $\T$-invariant point. 
By propositions \ref{prop:6}\eqref{item:75} and
  \ref{prop:1},
  \begin{displaymath}
    0 < \h_{\ov D}(p_{0})=\vartheta _{\ov D}(m_{0}).
  \end{displaymath}
  By the concavity of $\vartheta _{\ov D}$, we deduce that $\vartheta
  _{\ov D}(x)> 0$ for all $x\in \Delta $. 

Conversely, assume that  $\vartheta _{\ov D}$ is positive on $\Delta$. Let $\ov M $ be
a metrized $\R$-divisor on $X$. Since $D$ is  ample, there is
a positive integer $\ell'$ such that $M + \ell' D$ is
generated by $\R$-sections. Let $s_{1},\dots, s_{r}$ be a set of
generating $\R$-sections. By the quasi-algebraicity of the metrics,
$\|s_{i}\|_{v,\sup} =1$ for all $v$ outside a finite subset $S\subset
\mathfrak{M}_{\K}$. Put
\begin{displaymath}
 \eta =\min_{i} \min_{v\in S}\|s_{i}\|_{v,\sup}^{-1}.
\end{displaymath}
Let $m_{1},\dots, m_{l}$ be the vertices of $\Delta$ and $s_{m_{j}}$
the $\R$-section of $D$ corresponding to $m_{j}$.  By Proposition
\ref{prop:3}\eqref{item:13} and the hypothesis that $\vartheta_{\ov
  D}$ is positive,
 \begin{displaymath}
   \prod_{v}\Vert s_{m_{j}}\Vert_{v,\sup}^{n_{v}}= \e^{-\vartheta _{\ov D}(m_{j})}<1.
 \end{displaymath}
 Therefore, we can apply Proposition \ref{prop:2} to each
 $s_{m_{j}}$ to obtain $\ell_{0}\ge1$ and $\alpha_{j}\in \K^{\times}$
 such that, for each $\ell\ge \ell_{0}+\ell'$ and $j=1,\dots, r$,
\begin{displaymath}
  \|\alpha _{j}s_{m_{j}}^{ \ell-\ell'}\|_{v,\sup}
  \begin{cases}
    \le 1&\text{ if }v\not\in S,\\
    < \eta&\text{ if }v\in S.
  \end{cases}
\end{displaymath}
The set of global sections
\begin{math}
 \alpha _{j}s_{i}  s_{m_{j}}^{ \ell-\ell'}
\end{math}, $i=1,\dots,r$, $j=1,\dots,l$,  
generates the $\R$-divisor  $M+ \ell D=(M + 
\ell' D) + (\ell-\ell') D$. Moreover,  for $v\in S$,
\begin{displaymath}
  \|\alpha _{j}s_{i} 
  s_{m_{j}}^{\ell-\ell'}\|_{v,\sup}\le 
\|s_{i}\|_{v,\sup}\|\alpha
_{j}s_{m_{j}}^{ (\ell-\ell')}\|_{v,\sup} <\eta^{-1}\eta= 1,
\end{displaymath}
while \begin{math}
  \|\alpha _{j}s_{i} 
  s_{m_{j}}^{ \ell-\ell'}\|_{v,\sup}\le 1,
\end{math}
for $v\not\in S$.  Thus $\ov M +  \ell\ov D$ is generated by
small $\R$-sections and $\ov D$ is ample, finishing the proof of
\eqref{item:7}.

We next prove \eqref{item:8}.  By Proposition
\ref{prop:26}\eqref{item:83}, the $\R$-divisor $D$ is nef if and only
if $\Psi $ is concave and, by Proposition
\ref{prop:23}\eqref{item:73}, $\ov D$ is semipositive if and only if
the functions $\psi _{\ov D,v}$ are concave for all $v$. We are
reduced to show that, under the assumption that $D$ is nef and that
$\ov D$ is semipositive, $\ov D$ being nef is equivalent to
the nonnegativity of $\vartheta _{\ov D}$.

  Assume that $\ov D$ is nef. As in the first part of the ample case, let $m_{0}$ be a
  vertex of $\Delta $ and $p_{0}\in X({\K})$ a $\T$-invariant
closed  point corresponding to $m_{0}$. By Proposition~\ref{prop:1},
  \begin{displaymath}
    0 \le  \h_{\ov D}(p_{0})=\vartheta _{\ov D}(m_{0}).
  \end{displaymath}
  By the concavity of $\vartheta _{\ov D}$ we deduce $\vartheta
  _{\ov D}(x)\ge 0$ for all $x\in \Delta $. 

  Conversely, assume  that $\vartheta _{\ov D}$ is nonnegative and
  let $p\in X({\F})$ be a $\F$-rational point for a finite extension $\F$
  of $\K$. Suppose that $p$ lies in an orbit $O(\sigma )$ for a cone
  $\sigma\in \Sigma$. Choose a point $m_{\sigma}$ in the face
  $F_{\sigma }$ of $\Delta $ corresponding to $\sigma $. Then $
  \div(s_{m_{\sigma}})=D+\div(\chi^{m_{\sigma}})$ meets the closure $V(\sigma)$ properly,
  and so we can restrict this $\R$-divisor to $V(\sigma)$.  Set $\ov
  D'= \ov D+\wh \div(\chi^{m_{\sigma}})$. Then
  \begin{displaymath}
\h_{\ov D}(p)= \h_{\ov D'}(p)= 
    \h_{\ov D'\mid_{V(\sigma )}}(p).
\end{displaymath}
Furthermore, the roof function $\vartheta _{\ov D'|_{V(\sigma)}}$
coincides, up to translation by $m_{\sigma}$, with the restriction
$\vartheta _{\ov D}\big|_{F_{\sigma }}$, by the analogue for
$\R$-divisors of \cite[Proposition 4.8.8]{BurgosPhilipponSombra:agtvmmh}.

Thus, we are reduced to prove the nonnegativity of the height for
points in the principal orbit.  Let $m\in \Delta$ and
$s_{m}$ be the corresponding monomial $\R$-section.  Since~$p$ lies in the
principal orbit, it is not contained in the support of
$\div(s_{m})$. Then, using Proposition \ref{prop:3}\eqref{item:13},
  \begin{multline*}
    \h_{\ov D}(p)=\sum_{w\in \mathfrak{M}_{\F}}-n_{w}\log\|s_{m}(p)\|_{w}\ge
    \sum_{w\in \mathfrak{M}_{\F}}-n_{w}\log\|s_{m}\|_{w,\sup}\\=\sum_{w\in
      \mathfrak{M}_{\F}}n_{w}\vartheta _{\ov D,w}(m)=
    \sum_{v\in
      \mathfrak{M}_{\K}}n_{v}\vartheta _{\ov D,v}(m)=
    \vartheta _{\ov
      D}(m)\ge 0,
  \end{multline*}
which concludes the proof of \eqref{item:8}. 

  The statement \eqref{item:9} is a direct consequence of Theorem
  \ref{thm:1} and the definition of big metrized $\R$-divisor. 

  We now prove \eqref{item:100}. By the toric Chow lemma, there is a
   birational toric map $\varphi\colon X'\to X$ with $X'$
  projective. Let $A$ be an effective ample toric divisor on $X'$ and write
  $\Phi =\Psi _{A}$ and $\Gamma=\Delta_{A}=\stab(\Phi)$ for the
  corresponding support function and polytope. By Proposition
  \ref{prop:16}\eqref{item:45},  $\Phi \le 0 $.

  Choose $v_{0}\in \mathfrak{M}_{\K}$ and consider the 
  toric metric on $A$ given by the family of functions
   \begin{displaymath}
     \phi _{v}=
     \begin{cases}
       \Phi -1&\text{ if }v=v_{0}, \\
       \Phi &\text{ otherwise.}
     \end{cases}  
   \end{displaymath}
   Denote by $\ov A$ the obtained toric metrized divisor. We have
   that $\phi_{v}^{\vee}\equiv 1$ if $v=v_{0}$ and
   $\phi_{v}^{\vee}\equiv 0$ otherwise.  Hence, for any $x\in\Gamma$,
\begin{displaymath}
  \vartheta_{\ov A}(x) = \sum_{v}n_{v}
  \phi_{v}^{\vee}(x)= n_{v_{0}}>0.
\end{displaymath}
By  \eqref{item:7}, $\ov A$ is ample.

Suppose that $\ov D$ is pseudo-effective. By Proposition
\ref{prop:27}, the metrized $\R$-divisor $\ell\varphi^{*}\ov D + \ov
A$ is big for all $\ell\ge1$. The virtual support function and the
polytope corresponding to this divisor are $\ell\Psi + \Phi$ and
$\stab(\ell\Psi+\Phi)$, respectively. Write $\psi_{v}=\psi_{\ov D,v}$ for convenience. For $v\in \mathfrak M_{\K}$, the function
associated to the $v$-adic metric on this $\R$-divisor is
$\ell\psi_{v}+ \phi_{v}$, and the $v$-adic roof function is
$(\ell\psi_{v}+\phi_{v})^{\vee}$.

Set $\Delta_{\ell}= \frac{1}{\ell} \stab(\ell\Psi+\Phi)=
\stab\big(\psi_{v}+\frac{1}{\ell}\phi_{v}\big)$ and
$\vartheta_{v,\ell}=(\psi_{v} + \frac{1}{\ell} \phi_{v})^{\vee}$.
Then, for $x\in \Delta_{\ell}$,
\begin{displaymath}
(\ell \psi_{v}+\phi_{v})^{\vee}(\ell x) = 
  \ell \vartheta_{v,\ell} (x).
\end{displaymath}
Set $\vartheta_{\ell}=\sum n_{v}\vartheta _{v,\ell}$.  Since $
\ell\varphi^{*}\ov D + \ov A $ is toric and big, by
\eqref{item:9}, there exists a $x_{\ell}\in \Delta _{\ell}$ such that
$\vartheta _{\ell}(x_{\ell})>0$.

By Proposition \ref{prop:22}\eqref{item:70}, the functions $\psi_{v}$ and
$\phi_{v}$ are asymptotically conic and, by construction,
$0\in\stab(\phi_{v})$. Since  by Lemma \ref{lemm:6}\eqref{item:36} one has $\Delta_\ell\supset\Delta_{\ell+1}$ for $\ell\geq1$, the sequence $(x_{\ell})_{\ell}$ lies in the polytope $\Delta_{1}$.
By choosing a convergent subsequence, we can assume that it converges
to a point $  x_{\infty}\in \bigcap_{\ell}\Delta_{\ell}$. By Lemma
\ref{lemm:6}\eqref{item:36}, 
$  x_{\infty}\in\Delta$ and, 
by Lemma \ref{lemm:6}\eqref{item:15}, for $\ell'\ge \ell$,
\begin{displaymath}
\vartheta_{v,\ell}(x_{\ell'} )\ge
\vartheta_{v,\ell'}(x_{\ell'})+\Big(\frac{1}{\ell}-\frac{1}{\ell'}\Big)
\phi_{v}^{\vee}(0) \ge \vartheta_{v,\ell'}(x_{\ell'}).
\end{displaymath}
Hence, $\vartheta_{\ell}(x_{\ell'} )\ge
\vartheta_{\ell'}(x_{\ell'})\ge 0$. By continuity,
$\vartheta_{\ell}(x_{\infty} )=\lim_{\ell'\to\infty}
\vartheta_{\ell}(x_{\ell'} )\ge 0$.  Applying again Lemma
\ref{lemm:6}\eqref{item:15},
\begin{displaymath}
  \vartheta_{\ov D}(x_{\infty})= \sum_{v}n_{v}\psi_{v}^{\vee}(x_{\infty})=\sum_{v}n_{v}\lim_{\ell\to\infty}
  \vartheta_{v,\ell}(x_{\infty})=\lim_{\ell\to\infty}
  \vartheta_{\ell}(x_{\infty})\ge 0,
\end{displaymath}
which proves the statement. 

Conversely, suppose now that there exists $x\in \Delta $ with
$\vartheta _{\ov D}(x)\ge 0$ maximal. In particular, $\Delta \not =
\emptyset$.  The virtual support function corresponding to the
$\R$-divisor $\ell \varphi^{*}D + A$ is equal to $\ell \Psi+\Phi$.  By
Lemma \ref{lemm:4}\eqref{item:35}, $\stab(\ell \Psi+\Phi) \supset
\ell\Delta+\Gamma$. In particular, this is a polytope of dimension
$n$. Furthermore, for $v\in \mathfrak{M}_{\K}$, the corresponding
function is $\ell \psi_{v}+ \phi_{v}$. Applying Lemma
\ref{lemm:4}\eqref{item:74}, we deduce that
\begin{displaymath}
(\ell \psi_{v}+ \phi_{v})^{\vee}(\ell x) \ge \bigl((\ell
\psi_{v})^{\vee}\boxplus \phi_{v}^{\vee}\bigr)(\ell x) = \ell\psi_{v}^{\vee} (x)  +  \phi_{v}^{\vee} (0).  
\end{displaymath}
Hence, 
\begin{multline*}
  \vartheta _{\ell \varphi^{*}\ov D + \ov A}(\ell x) =\sum_{v} n_{v}
  (\ell \psi_{v}+ \phi_{v})^{\vee}
  (\ell x) \\ \ge \sum_{v} n_{v} (\ell\psi_{v}^{\vee} (x)
  +  \phi_{v}^{\vee} (0)) 
\ge \ell \vartheta_{\ov D}(x)+ n_{v_{0}} >0.
\end{multline*}
By \eqref{item:9}, this implies that $\ell \varphi^{*}\ov D + \ov A$
is big and therefore $\ov D$ is pseudo-effective.

We finally prove \eqref{item:88}. The metrized $\R$-divisor $\ov D$ is
effective if and only if $D$ is effective and $s_{D}$ is small which,
by propositions~\ref{prop:16}\eqref{item:45} and
\ref{prop:3}\eqref{item:13}, is equivalent to the fact that
$0\in\Delta_{D}$ and $\vartheta_{\ov D,v}(0)\ge 0$, for all $v\in
\mathfrak{M}_{\K}$.
\end{proof}

\begin{cor}
  Let $\ov D$ be {a semipositive} toric metrized $\R$-divisor on $X$
   {with $D$ big}. Then $\h_{\ov D}(X)=\avol(X,\ov D)$ if and
  only if $\ov D$ is nef.
\end{cor}

\begin{proof}
  This follows easily from Proposition \ref{prop:1} and theorems
  \ref{thm:1} and \ref{thm:2}\eqref{item:8}.
\end{proof}

We also obtain the following arithmetic analogue of the Nakai-Moishezon
criterion for toric varieties. 

\begin{cor} \label{cor:9} Let $\ov D$ be {a semipositive} toric
  metrized $\R$-divisor on $X$.
  \begin{enumerate}
  \item \label{item:89} If $D$ is nef, the following conditions are equivalent: 
    \begin{enumerate}
    \item \label{item:91} $\ov D$ is nef;
\item \label{item:92} $\h_{\ov D}(p)\ge 0$ for every $\T$-invariant
  point $p$; 
\item \label{item:93} $\h_{\ov D}(Y)\ge 0$ for every subvariety $Y$ of $X_{\ov \K}$.
    \end{enumerate}
  \item \label{item:90} If $D$ is ample, the following conditions are
    equivalent:
    \begin{enumerate}
    \item \label{item:94} $\ov D$ is ample;
\item \label{item:95} $\h_{\ov D}(p)> 0$ for every $\T$-invariant
  point $p$; 
\item \label{item:96} $\h_{\ov D}(Y)> 0$ for every subvariety
  $Y$ of $X_{\ov \K}$. 
    \end{enumerate}
  \end{enumerate}
\end{cor}

\begin{proof}
We first prove  \eqref{item:89}. The equivalence between \eqref{item:91}
and  \eqref{item:92} follows from the proof of Theorem
\ref{thm:2}\eqref{item:8}. 

It is clear that \eqref{item:93} implies \eqref{item:92}. For the
converse, let $Y$ be a subvariety of $X$ defined over a finite
extension $\F$ of $\K$. Using the same argument as in the proof of
Theorem~\ref{thm:2}\eqref{item:8}, we reduce to points
in the principal orbit and we assume that $Y$ meets this principal
orbit.

  We prove that $\h_{\ov D}(Y)\ge 0$ by induction on the dimension of
  $Y$. The 0-dimensional case follows from the fact that $\ov D$ is
  nef. Assume now that $Y$ has dimension $d\ge1$. Let $m\in \Delta$ and $s_{m}$ the
  corresponding toric $\R$-section of $D$. By B\'ezout formula~\eqref{eq:20},
  \begin{displaymath}
    \h_{\ov D}(Y)=\h_{\ov D}(Y\cdot \div (s_{m}))-
    \sum_{w\in \mathfrak{M}_{\F}}n_{w}\int_{X^{\an}_{w}}\log
    \|s_{m}\|_{w} \, \chern_{1}(D, \Vert\cdot\Vert_{w})^{\land d}\land \delta _{Y}.
  \end{displaymath}
  We have that $h_{\ov D}(Y\cdot \div s_{m})\ge 0$ by the inductive
  hypothesis, and observe that $\|s_{m}\|_{w}\le
  \|s_{m}\|_{w,\sup}$. The fact that the metric is semipositive
  implies that, for each $w\in \mathfrak{M}_{\F}$, the signed measure
  $\chern_{1}(D, \Vert\cdot\Vert_{w})^{\land d}\land \delta _{Y} $ is
  nonnegative and of total mass $\deg_{D}(Y)$. Using Proposition
  \ref{prop:3}\eqref{item:13},
  \begin{multline*}
    \h_{\ov D}(Y) \ge - \sum_{w\in
      \mathfrak{M}_{\F}}n_{w}\int_{X^{\an}_{w}}\log
    \|s_{m}\|_{w,\sup}\, \chern_{1}(D,
    \Vert\cdot\Vert_{w})^{\land d}\land \delta _{Y}\\
    = \sum_{w\in \mathfrak{M}_{\F}}n_{w}\vartheta_{\ov D,w}(m)
    \deg_{D}(Y) =\vartheta _{\ov D}(m)\deg_{D}(Y).
  \end{multline*}
  Theorem \ref{thm:2}\eqref{item:8} implies that $\vartheta _{\ov
    D}(m)\ge0$, which concludes the proof of \eqref{item:89}.  The
  proof of \eqref{item:90} can be done similarly. 
\end{proof}

\begin{prop} \label{prop:14}
  Let $\ov D,\ov E$ be two toric metrized $\R$-divisors on a toric
  variety $X$ such that $\ov
  E$ is semipositive. The following conditions are equivalent: 
  \begin{enumerate}
  \item \label{item:38} $\ov D\ge \ov E$;
  \item \label{item:39} $\psi_{\ov D,v}\le \psi_{\ov E,v}$ for all
    $v\in \mathfrak M_{\K}$;
  \item \label{item:40} $\Delta_{E}\subset \Delta_{D}$ and
    $\vartheta_{\ov E,v}(x)\le \vartheta_{\ov D,v}(x)$ for all $x\in
    \Delta_{E}$ and $v\in \mathfrak M_{\K}$.
  \end{enumerate}
\end{prop}

\begin{proof}
  $\eqref{item:38} \Leftrightarrow \eqref{item:39}$ We have that $\ov
  D\ge\ov E$ if and only if the $\R$-section $s_{\ov D-\ov E}$ is
  small.  This is equivalent to the fact that, for $v\in \mathfrak
  M_{\K}$, $\psi_{\ov D-\ov E,v}(u)\le 0$ for all $u\in N_{\R}$, which
  in turn is equivalent to $\psi_{\ov E,v}(u)\ge \psi_{\ov D,v}(u)$
  for all $u$.

$\eqref{item:39} \Rightarrow \eqref{item:40}$ By considering the corresponding recession functions,
  we deduce that $\Psi_{E}(u)\ge \Psi_{D}(u)$ for all
  $u$, which implies that
  \begin{multline*}
    \Delta_{E}=\{x\in M_{\R}\mid \langle x,u\rangle\ge\Psi_{E}(u)
    \text{ for all } u\in N_{\R}\}\\
\subset\{x\in M_{\R}\mid \langle x,u\rangle\ge\Psi_{D}(u)
    \text{ for all } u\in N_{\R}\} =\Delta_{D}.
  \end{multline*}
Similarly, for $x\in \Delta_{E}$,
\begin{displaymath}
  \vartheta_{\ov E,v}(x)=\inf_{u\in N_{\R}}\big( \langle x,u\rangle- \psi_{\ov E,v}(u)\big)
  \le \inf_{u\in \N_{\R}}\big( \langle x,u\rangle- \psi_{\ov D,v}(u)\big)
= \vartheta_{\ov D,v}(x).
\end{displaymath}
  
$\eqref{item:40} \Rightarrow \eqref{item:39}$ Let $v\in \mathfrak
M_{\K}$. Since $\ov E$ is semipositive, the function $\psi_{\ov E,v}$
is concave and so $\psi_{\ov E,v}=\vartheta_{\ov
  E,v}^{\vee}$. The fact that $\vartheta_{\ov E,v}\le \vartheta_{\ov
  D,v}$ on $ \Delta_{E}$ implies that $\vartheta_{\ov E,v}^{\vee} \ge
\vartheta_{\ov D,v}^{\vee}$ on $ N_{\R}$. Moreover,
$$
\vartheta_{\ov D,v}^{\vee}= \conc(\psi_{\ov D,v})\ge
\psi_{\ov D,v},
$$
where $\conc(\psi_{\ov D,v})$ denotes the concave envelope of the
function $\psi_{\ov D,v}$ (Appendix
\ref{sec:conv-analys-asympt}). Hence $\psi_{\ov E,v}\ge \psi_{\ov
  D,v}$, which concludes the proof.
\end{proof}

\begin{exmpl}\label{exm:11}
  This example is due to Moriwaki \cite{Moriwaki:bad}. Let $X=\P^{n}_{\Q}$
  with homogeneous coordinates $(z_{0}:\dots:z_{n})$ and
  $D=\div(z_{0})$. Then $X$ is a toric
  variety and $D$ is a toric divisor corresponding to the support
  function $\Psi
  \colon \R^{n}\to \R$ given by
  \begin{displaymath}
    \Psi (u_{1},\dots,u_{n})=\min(0,u_{1},\dots,u_{n}).
  \end{displaymath}
  Let $\bfalpha=(\alpha_{0},\dots,\alpha_{n})$ be a collection of
  positive real numbers. Consider the toric semipositive metric on $D$
  given, under the correspondence in
  Proposition~\ref{prop:22}\eqref{item:71}, by the family of functions
  \begin{displaymath}
    \psi _{v} {(u_{1},\dots, u_{n})}=
  \begin{cases}
      \frac{-1}{2}\log(\alpha_{0}+\alpha_{1}\e^{-2u_{1}}+\dots+\alpha_{n}\e^{-2u_{n}})&
      \text{ if }v=\infty,\\
      \Psi  &\text{ otherwise.}
  \end{cases}
  \end{displaymath}
For $v=\infty$, this metric agrees
  with the {weighted} Fubini-Study metric given by
  \begin{displaymath}
    \|s(z_{0}:\dots:z_{n})\|^{2}_{\infty}=
    \frac{z_{0}\ov z_{0}}
    {\alpha_{0}z_{0}\ov z_{0}+\dots +\alpha_{n}z_{n}\ov z_{n}}
  \end{displaymath}
 {for the toric section $s$ corresponding to the linear form $z_{0}$.}
  For the non-Archimedean places, it  agrees with the canonical
  metric and is induced by the canonical model
  $\mathcal{O}_{\P^{n}_{\Z}}(1)$.  We denote $\ov D_{\bfalpha}$ this toric
  metrized divisor.

  Let $\{e_1,\dots,e_n\}$ be the standard basis of
  $M_{\R}=(\R^{n})^{\vee}$ and $(x_{1},\dots,x_{n})$ the
  coordinates of a point $x\in M_{\R}$ with respect to this basis. Set also $e_0=0$ and
  $x_{0}=1-\sum_{i} x_{i}$. The polytope $\Delta =\stab(\Psi
  )$ is the 
  standard simplex $\conv(e_{0},\dots,e_{n})$ of $\R^{n}$. Computing as in
  \cite[Example~2.4.3]{BurgosPhilipponSombra:agtvmmh}, one sees that
  the local roof functions are given, for  $(x_{1},\dots,x_{n})\in
  \Delta$,  by
  \begin{displaymath}
    \vartheta _{\ov D_{\bfalpha},v}(x_{1},\dots,x_{n})=
    \begin{cases}
      \frac{-1}{2}\sum_{i=0}^{n}x_{i}\log(x_{i}/\alpha_{i})&\text{ if }v=\infty,\\
      0&\text{ otherwise.}
    \end{cases}
  \end{displaymath}
This is an adelic family of continuous
concave functions on $\Delta $. By {\it loc. cit.}, Theorem 4.8.1(2), the metrized
divisor $\ov D_{\alpha}$ is semipositive. 
Its roof function is 
$$
\vartheta_{\ov D_{\bfalpha}}(x_{1},\dots,x_{n})
  =\frac{-1}{2}\sum_{i=0}^{n}x_{i}\log(x_{i}/\alpha_{i}). 
$$
The values of the roof function at the vertices of $\Delta $ are given by
$\vartheta_{\ov D_{\bfalpha}}(e_{i})=(1/2)\log(\alpha_{i})$,
$i=0,\dots, n$. Moreover
  \begin{displaymath}
    \sup_{x\in \Delta }\vartheta_{\ov D_{\bfalpha}}(x)=\frac{1}{2}\log(\alpha_{0}+\dots+\alpha_{n}).
  \end{displaymath}
  Write $\Theta _{\bfalpha}=\{x\in \Delta \mid \vartheta _{\ov D_{\bfalpha}}(x)\ge 0\}$.
  From Theorem \ref{thm:2}, we deduce part of the main result of \cite{Moriwaki:bad}:
  \begin{enumerate}
  \item $\ov D_{\bfalpha}$ is ample if and only if $\alpha_{i}>1$, $i=0,\dots,n$;
  \item $\ov D_{\bfalpha}$ is nef if and only if $\alpha_{i}\ge 1$, $i=0,\dots,n$;
  \item $\ov D_{\bfalpha}$ is big if and only if $\sum_{i} \alpha_{i}>1$;
  \item $\ov D_{\bfalpha}$ is pseudo-effective if and only if $\sum_{i}\alpha_{i}\ge 1$;
  \item $\ov D_{\bfalpha}$ is effective if and only if $\alpha_{0}\ge 1$.
\end{enumerate}
  By \cite[Theorem 5.2.5]{BurgosPhilipponSombra:agtvmmh} and Theorem
  \ref{thm:1}, 
  \begin{equation*}
    \h_{\ov D_{\bfalpha}}(X)=\avol_{\chi}(X,\ov D_{\bfalpha})=(n+1)!  \int
    _{\Delta }\vartheta_{\ov D_{\bfalpha}} \dd\vol =\frac{1}{2}
    \sum_{i=0}^{n}\bigg( \log(\alpha_{i})+\sum_{j=1}^{i}\frac{1}{j}\bigg).
  \end{equation*}
Apparently, there is no such a simple formula for the arithmetic volume
$$\avol(X,\ov D_{\bfalpha})=(n+1)! \int _{\Theta
  _{\bfalpha}}\vartheta_{\ov D_{\bfalpha}} \dd\vol.$$
\end{exmpl}

\section{Toric versions of Dirichlet's unit theorem,  Zariski decomposition
  and Fujita
  approximation} \label{sec:dirichl-unit-theor}

In this section, we give a complete answer to questions \ref{ques:1}
and \ref{ques:4}  in the toric case and a partial answer to
Question \ref{ques:3}. Namely, we give a criterion for when a
\emph{toric} Zariski decomposition or a strong \emph{toric} Zariski
decomposition exists. We will show in \S \ref{sec:zariski-decomp} that,
for big toric arithmetic $\R$-divisors and under some restrictions,
the existence of a non-necessarily toric Zariski decomposition
implies the existence of a toric one.

\begin{defn} \label{def:27}
  Let $X$ be a proper toric variety over $\K$ and $\ov D$ a toric metrized
  $\R$-divisor on $X$. 
  A \emph{toric Zariski decomposition} of $\ov D$ is a Zariski
  decomposition $\varphi^{\ast}\ov D=\ov P+\ov E$ such that $\varphi$
  is a  birational toric map  of proper toric varieties  and $\ov P$ (hence $\ov E$) is a
  toric metrized $\R$-divisor.  
  
  We denote by $\Upsilon _{\T}(\ov D)\subset \Upsilon
  (\ov D)$ the subset formed by the elements $(\varphi,\ov P)$,
  with $\varphi$ a proper birational toric map  of proper toric varieties  and $\ov P$ a
  toric metrized $\R$-divisor. The equivalence relation $\sim$ in Definition \ref{def:2} induces
  an equivalence relation on $\Upsilon _{\T}(\ov D)$ and the order
  relation on $\Upsilon (\ov D)/\sim$ induces an order relation on
  $\Upsilon_{\T} (\ov D)/\sim$. A \emph{toric strong Zariski
    decomposition} is the greatest element  of $\Upsilon _{\T}(\ov
  D)/\sim$, if it exists. 
\end{defn}

\begin{thm} \label{thm:4}
Let $X$ be a proper toric variety over $\K$ and $\ov D$ a toric metrized
$\R$-divisor on $X$. 

\begin{enumerate}
\item \label{item:54} Assume that $\K$ is an $A$-field. Then, for every  $a\in \Theta _{\ov D}$ there exists
  an $\alpha\in \K^{\times}\otimes \R$ such that
\begin{equation}\label{eq:32}
  \ov D+ \wh\div(\alpha \chi^{a})\ge 0.
\end{equation}
Therefore, if $\ov D$ is pseudo-effective, then there exists $a\in \Delta
_{D}$ and $\alpha\in \K^{\times}\otimes \R$ satisfying
\eqref{eq:32}. 
\item \label{item:55} If $\ov D$ is pseudo-effective, then there
  exists a toric strong Zariski decomposition of $\ov D$ if and only
  if $\Theta_{\ov D}$ is a quasi-rational polytope. 

  \item \label{item:26} If $\ov D$ is big, the following statements are equivalent: 
    \begin{enumerate}
    \item \label{item:101} there exists a toric Zariski decomposition of $\ov D$;
\item \label{item:102} there exists a strong toric Zariski decomposition of $\ov D$; 
\item \label{item:103} $\Theta_{\ov D}$ is a quasi-rational polytope.
    \end{enumerate}
\item \label{item:56} Assume that $\ov D$ is big. Then, for every
  $\varepsilon>0$, there exists a birational toric map $\varphi\colon
  X'\to X$  of proper toric varieties  and
toric 
metrized $\R$-divisors $\ov A$, $\ov E$ on $X'$ such that $\ov A$ is
ample, $\ov E$ is effective,
\begin{displaymath}
 \varphi^{*}\ov D=\ov A+\ov E\quad\text{and}\quad \wh\vol(X',\ov A)\ge
 \wh\vol(X,\ov D)-\varepsilon. 
\end{displaymath}
\end{enumerate}
\end{thm}

\begin{proof}
  \eqref{item:54} Let $S$ be a
  finite subset of $\mathfrak M_{\K}$ containing the Archimedean
  places and those places such that $\vartheta_{\ov D,v}(a)\ne 0$.
  Let $\gamma_{v}\in \R$, $v\in \mathfrak M_{\K}$, such that
  $\gamma_{v}\le \vartheta_{\ov D,v}(a)$ for all $v\in S$, $\gamma_{v}=0$
  for $v\notin S$, and $\sum_{v}n_{v}\gamma_{v}=0$. Dirichlet's unit
  theorem for $A$-fields \cite[Chapter IV, \S 4, Theorem 9]{Weil:bnt}
  implies that there exists $\alpha\in \K^{\times}\otimes \R$ such
  that $\log|\alpha|_{v}=\gamma_{v}$ { for all} $v$.  Set $\ov D'=\ov D+
  \wh\div(\alpha \chi^{a})$. Then, for all $v\in \mathfrak M_{\K}$ and
  $x\in M_{\R}$,
\begin{displaymath}
  \vartheta_{ \ov D',v}(x) =   \vartheta_{ \ov D,v}(x+a)-\gamma_{v}.
\end{displaymath}
In particular, $\vartheta_{ \ov D',v}(0) =   \vartheta_{ \ov
  D,v}(a)-\gamma_{v}\ge 0$ for all $v$ and so $\ov D'\ge0$, as stated. 

If $\ov D$ is pseudo-effective, then $\Theta _{\ov D}\not =\emptyset$
by Theorem \ref{thm:2}\eqref{item:100}, which proves the second
statement.

\eqref{item:55} Let $\varphi\colon X'\to X$ be a birational toric map
and $\ov P$ a nef toric metrized $\R$-divisor on $X'$ such that $\ov
P\le \varphi^{*}\ov D$.  In particular, $\ov P$ is semipositive. By
Proposition \ref{prop:14}, $\Delta_{P}\subset \Delta_{\varphi^{*}D}=
\Delta_{D}$ and $\vartheta_{\ov P,v}(x)\le \vartheta_{\varphi^{*}\ov
  D,v}=\vartheta_{\ov D,v}(x)$ for all $v\in \mathfrak M_{\K}$ and $x\in {\Delta_{P}}$.
Furthermore, by Theorem~\ref{thm:2}\eqref{item:8}, $\vartheta_{\ov
  P}\ge 0$ on ${\Delta_{P}}$.  Hence, for  $x\in {\Delta_{P}}$,
\begin{equation}\label{eq:30}
\vartheta_{\ov D}(x) = \sum_{v}n_{v}\vartheta_{\ov D,v}(x) \ge \vartheta_{\ov
  P}(x)\ge 0
\end{equation}
and $\Delta_{P}\subset \Theta_{\ov D}$. 

Assume that $(\varphi,\ov P)$ is maximal.
Suppose that $\Theta_{\ov
  D}$ is not a quasi-rational polytope. Then $\Delta_{P}\ne
\Theta_{\ov D}$ and so there is a quasi-rational polytope
$\Delta'\supsetneq \Delta_{P}$  contained in $\Theta_{\ov D}$.
Let $\Sigma'$ be a common refinement of $\Sigma$ and
$\Sigma_{\Delta'}$. Set $X'=X_{\Sigma'}$ and let $\varphi\colon X'\to
X$ be the associate birational toric  map.
Let $P'$ be the toric $\R$-divisor of $X'$ determined by $\Delta'$ and
$\ov P'$ the toric semipositive metrized $\R$-divisor associated to
the restriction to $\Delta'$ of the family of concave functions
$\{\vartheta_{\ov D,v}\}_{v}$ under the correspondence in Proposition
\ref{prop:23}\eqref{item:73}.  We have that $\ov P'$ is nef (by
Theorem \ref{thm:2}\eqref{item:8}), $\ov P\le \varphi^{*} \ov P'$ (by
Proposition \ref{prop:14}) and $\ov P\ne \ov P'$ since the associated
polytopes are different. Hence, $\ov P$ is not maximal contradicting
the hypothesis. We conclude that~$\Theta_{\ov D}=\Delta_P$ is a
quasi-rational polytope.

Conversely, assume that $\Theta_{\ov D}$ is a quasi-rational
polytope. Let $\Sigma'$ be a common refinement of $\Sigma$
and $\Sigma_{\Theta_{\ov D}}$. Set $X'=X_{\Sigma'}$ and 
let $\varphi\colon X'\to X$  be the associated birational toric map. 

Let $P$ be the toric $\R$-divisor of $X'$ determined by $\Theta_{\ov
  D}$ and $\ov P$ the toric semipositive metrized $\R$-divisor
associated to the restriction to $\Theta_{\ov D}$ of the family of
concave functions $\{\vartheta_{\ov D,v}\}_{v}$ under the
correspondence in Proposition \ref{prop:23}\eqref{item:73}.  Hence,
$\ov P$ is nef (by Theorem \ref{thm:2}\eqref{item:8}) and $\ov P\le
\varphi^{*} \ov D$ (by Proposition \ref{prop:14}).

Now we show the maximality of the class of $(\varphi,\ov P)$. Let
$\Sigma_{1}$ be a refinement of $\Sigma$, $\varphi_{1}\colon X_{1}\to
X$ the corresponding birational toric map and $\ov P_{1}$ a nef toric
metrized $\R$-divisor on $X_{1}$ with $\varphi_{1}^{*}\ov D\ge \ov
P_{1}$.  By Proposition \ref{prop:14}, $\Delta_{P_{1}}\subset
\Delta_{D}$ and $\vartheta_{\ov P_{1},v}(x)\le \vartheta_{\ov D,v}(x)$
for all $v\in \mathfrak M_{\K}$ and $x\in \Delta_{P_{1}}$.  Since $\ov
P_{1}$ is nef, Theorem \ref{thm:2}\eqref{item:8} implies that
$\vartheta_{\ov P_{1}}(x)\ge 0$ for all $x\in \Delta_{P_{1}}$. Hence,
$\Delta_{P_{1}}\subset \Theta_{\ov D}$.  By construction, $
\Delta_{P_{1}}\subset\Delta_{P}$ and $ \vartheta_{\ov P_{1}}(x)\le
\vartheta_{\ov P}(x) $ for $ x\in\Delta_{P_{1}}$.

Taking a common refinement $\Sigma''$ of $\Sigma'$ and $\Sigma_{1}$,
we consider the corresponding birational toric maps $\nu\colon
X''\to X'$ and $\nu_{1}\colon X''\to X_{1}$. 
Proposition \ref{prop:14} then implies that $\nu^{*}\ov P\ge
\nu_{1}^{*}\ov P_{1}$, which proves the statement.

\eqref{item:26} Since a big metrized divisor is pseudo-effective the
equivalence of \eqref{item:102} and \eqref{item:103} follows from
\eqref{item:55}. Furthermore \eqref{item:102} and \eqref{item:103}
imply \eqref{item:101}, it remains to prove that \eqref{item:101}
implies \eqref{item:103}. Assume that $\varphi^{\ast}\ov D=\ov P+\ov
E$ is a toric Zariski decomposition. As before,  $\Delta _{P}\subset
\Theta
_{\ov D}$ and $\vartheta _{\ov P}\le \vartheta _{\ov D}$. The
equality of arithmetic volumes implies 
\begin{equation}\label{eq:27}
  \int_{\Delta_{P}} \vartheta_{\ov P} \dd \vol_{M}=
  \int_{\Theta_{\ov D}} \vartheta_{\ov D} \dd \vol_{M}.
\end{equation}
Since  $\ov D$ is big, the interior of $\Theta_{\ov D}$ is nonempty and
the function $ \vartheta_{\ov D}$ is strictly positive on
it. Then, equation \eqref{eq:27}
implies that $\Theta_{\ov D}$ is equal to $\Delta_{P}$, which is a 
quasi-rational polytope. 

\eqref{item:56} For each $\varepsilon > 0$ we pick a quasi-rational
polytope $\Delta'$ contained in the interior of $\Theta_{\ov D}$  and such that 
\begin{displaymath}
(n+1)!\int_{\Theta_{\ov
    D}\setminus \Delta'} \vartheta_{\ov D} \dd \vol_{M}\le
\varepsilon.
\end{displaymath}
Let $\Sigma'$ be a common refinement of $\Sigma$ and
$\Sigma_{\Delta'}$. We consider the corresponding toric variety
$X'=X_{\Sigma'}$ and the birational toric map $\varphi\colon X'\to X$. We set
$\ov A$ for the toric $ \R$-divisor corresponding to $\Delta'$
together with the metrics induced by the
restriction to this polytope of the family of concave functions
$\{\vartheta_{\ov D,v}\}_{v}$.

By concavity,  $\vartheta_{\ov A}$ is strictly positive on $\Delta'$.
Theorem \ref{thm:2}\eqref{item:7} then implies that $\ov A$ is ample.
By Proposition \ref{prop:14}, $\ov A\le \varphi^{*} \ov D$ and, by
construction, 
\begin{displaymath}
  \wh\vol(X',\ov A)=(n+1)!\int_{\Delta'} \vartheta_{\ov D} \dd
  \vol_{M}\ge (n+1)!\int_{\Theta_{\ov D}} \vartheta_{\ov D} \dd
  \vol_{M} -\varepsilon
= \wh\vol(X,\ov D)-\varepsilon,
\end{displaymath}
which concludes the proof. 
\end{proof}

\begin{rem}\label{rem:5}
  If $\ov D$ as in the previous theorem is pseudo-effective but 
  not big, then a toric Zariski decomposition always exists. Take any
  $a\in \Theta _{\ov D}$. Since the set $\{a\}$ is a quasi-rational polytope,
  using the construction of Theorem \ref{thm:4}\eqref{item:55} we can
  associate to it a nef toric metrized divisor on $X$ such that $\ov
  P \le \ov D$ and $\avol(X,\ov P)=0=\avol(X,\ov D)$. Clearly, in this
  case,  the decomposition may be non-unique. 
\end{rem}

\begin{exmpl} \label{exm:14} Consider again the toric metrized divisor
  $\ov D_{\bfalpha}$ on $\P^{n}_{\Q}$ in Example~\ref{exm:11}. We
  also use the notation and results therein.

First suppose that $\sum_{i}\alpha_{i}\ge1$ or, equivalently, that $\ov D_{\bfalpha}$
is pseudo-effective. By Theorem \ref{thm:4}\eqref{item:54},
Dirichlet's unit theorem holds true  in this case, as it does for any
pseudo-effective metrized toric $\R$-divisor over an $A$-field.

If $\sum_{i}\alpha_{i}=1$, then the set $\Theta_{\bfalpha}$ is a
point. Otherwise, this is a compact subset of $\Delta$ of
dimension $n$ with smooth boundary.  In this case,
$\Theta_{\bfalpha}$ is not a polytope, unless
$n=0,1$. Hence, by Theorem \ref{thm:4}\eqref{item:55}, $\ov D$ admits a strong
toric Zariski decomposition if and only if either
$\sum_{i}\alpha_{i}=1$, or $\sum_{i}\alpha_{i}>1$ and $n=0,1$.

Now suppose that $\sum_{i}\alpha_{i}> 1$ or, equivalently, that $\ov
D$ is big. Then, by Theorem~\ref{thm:4}\eqref{item:26}, $\ov D$ admits
a toric Zariski decomposition (and a strong Zariski decomposition) if
and only if $n=0,1$.

Finally, Theorem \ref{thm:4}\eqref{item:56} shows that a Fujita
approximation of $\ov D_{\bfalpha}$ always exists, as it does for any
big metrized toric $\R$-divisor over a global field.
\end{exmpl}

\section{Zariski decomposition on toric varieties}
\label{sec:zariski-decomp}

In the previous section, we gave a characterization  for when a toric Zariski
decomposition exists. Now, we will  study when the existence of
a not necessarily toric Zariski decomposition implies the existence of
a toric one. 
For technical reasons, we will restrict this study to algebraic metrized
$\R$-divisors arising from arithmetic $\R$-divisors as in Example
\ref{exm:3} and assume $\K=\Q$. 

\begin{defn}\label{def:17} Let $X$ be a smooth projective toric
  variety over $\Q$, $\cX$ a model of~$X$ over $\Z$ and $\ov {\cD}$ an
  arithmetic $\R$-divisor on $\cX$ of $C^{0}$-type. Let $D$ be the
  restriction of~$\cD$ to $X$ and $\ov D$ the algebraic metrized
  $\R$-divisor on $X$ associated to $\ov \cD$. We say that $\ov {\cD}$
  is a \emph{toric arithmetic $\R$-divisor} if $\ov D$ is a toric
  metrized $\R$-divisor. In particular, $D$ is a toric $\R$-divisor on
  $X$.
\end{defn}

The following is the main result of this section. 

\begin{thm}\label{thm:5}
  Let $X$ be a smooth projective toric variety over $\Q$, $\ov \cD$ a
  big toric arithmetic $\R$-divisor and $\ov D$ the associated toric
  metrized $\R$-divisor. Then, the following conditions are
  equivalent:
  \begin{enumerate}
  \item \label{item:3} there is a birational map $\sigma \colon
    \cY\to \cX$ of flat normal generically smooth 
   projective schemes over $\Z$ and a
   decomposition $\sigma ^{\ast}\ov \cD=\ov \cP+\ov \cE$ with $\ov \cP$ and
   $\ov \cE$ arithmetic $\R$-divisors on $\cY_{\Q}$ such that the
   corresponding metrized $\R$-divisors give a Zariski decomposition
   of $\ov D$;
  \item \label{item:2} there is a toric Zariski decomposition of
    $\ov D$;
 \item \label{item:4} $\Theta _{\ov D}$ is a quasi-rational polytope.
  \end{enumerate}
\end{thm}

Before proving the theorem we will need some preliminaries.

 Let $X$ be a proper variety over a field $K$. A \emph{divisorial
  valuation} of $X$ is a valuation $\nu$ on ${\rm K}(X)$ given, for $f\in
{\rm K}(X)$ by
\begin{displaymath}
  \nu(f)=\ord_{H}(\sigma^{*}f),
\end{displaymath}
where $\sigma\colon Y\to X$ is a birational
map of normal  proper
varieties over $K$ and $H\in \Div(Y)$ is  a prime Weil divisor.
We denote by $\DV(X)$ the set of all divisorial valuations of $X$. 

A divisorial valuation $\nu$ defines a map
\begin{math}
  \mult_{\nu}\colon \Car(X)\rightarrow \Z
\end{math}
given, for $D\in \Car(X)$, by $\mult_{\nu}(D)=
\nu(f_{D})$ where $f_{D}\in {\rm K}(X)$ is a local equation of $D$
on a neighbourhood of the point $\sigma(H)$. This map is a group
morphism and so it extends to a map 
\begin{displaymath}
  \mult_{\nu}\colon \Car(X)_{\R}\longrightarrow \R.
\end{displaymath}

Let $\rho\colon Z\to X$ be a birational map of normal proper
varieties. Since $\DV(X)=\DV(Z)$, for $\nu \in
\DV(X)$ the map $\mult_{\nu}$ extends to a map $ \Car(Z)_{\R}\to
\R$.

We will consider the following notion of arithmetic multiplicity for
metrized $\R$-divisors over a global field $\K$. 

\begin{defn} \label{def:21} Let $X$ be a proper variety over $\K$
  and $\ov D$ a metrized $\R$-divisor on $X$. The \emph{arithmetic
    multiplicity} of $\ov D$ is the function defined, for $\nu\in
  \DV(X)$, by
\begin{displaymath}
  \mu_{\ov D}(\nu)=
  \inf\Big\{\frac{1}{\ell} 
  \mult_{\nu}(\div(s)) \, \Big|\ 
  \ell\ge 1, s\in \wh{\Gamma}(X,\ell\ov D)^{\times}\Big\}.
\end{displaymath}  
Let $\rho\colon Y\to X$ be a birational map from a normal proper
variety $Y$ over $\K$ and $E\in \Car(Y)_{\R}$.  We say that the arithmetic
multiplicity of $\ov D$ is \emph{represented} by $E$ when $\mu_{\ov D}(\nu)=
\mult_{\nu}(E)$ for all $\nu\in \DV(X)$. 

The \emph{arithmetic multiplicity} of an arithmetic $\R$-divisor is
defined as the arithmetic multiplicity of its associated metrized
$\R$-divisor.
\end{defn}

We will show that, for a toric metrized $\R$-divisor $\ov D$, the
convex set $\Theta_{\ov D}$
can be expressed as the intersection of a family of halfspaces defined
by the arithmetic multiplicity of $\ov D$. When this arithmetic
multiplicity is representable, this convex set 
can be expressed as the intersection of a \emph{finite} sub-family of these
halfspaces which implies that, in this case, $\Theta_{\ov D}$ is a polytope.

\begin{prop}\label{prop:12}
  Let $X$ be a proper  toric variety over $\K$ and $\ov D$ a big
  toric metrized $\R$-divisor on $X$. Then 
  \begin{enumerate}
\item \label{item:31} for all $\nu\in \DV(X)$,
  \begin{displaymath}
  \mu_{\ov D}(\nu)= \inf\{ \mult_{\nu}(\div(s_{a})) \mid a\in \Theta_{\ov D}\cap M_{\Q}\};     
  \end{displaymath}
  \item \label{item:10} 
    $\Theta_{\ov D}=\{a\in M_{\R}\mid \mu_{\ov D}(\nu)\le
    \mult_{\nu}(\div(s_{a})) \text{ for all } \nu\in \DV(X) \};
    $
  \item \label{item:27} assume that there is a birational map $\sigma
    \colon Y\to X$ from a normal proper variety $Y$ over $\K$ and an
    $\R$-divisor $E$ on $Y$ that represents $\mu _{\ov D}$. Then, there
    are prime Weil divisors $H_{i}\in \Div(Y)$, $i=1,\dots, l$, such
    that
  \begin{displaymath}
    \Theta_{\ov D}=\{a\in M_{\R}\mid \mu_{\ov D}(\nu_{i})\le
    \mult_{\nu_{i}}(\div(s_{a})),\ i=1,\dots,l \},         
  \end{displaymath}
  where $\nu_{i}$ is the divisorial valuation defined by $H_{i}$.  In
  particular, $\Theta _{\ov D}$ is a quasi-rational polytope.
  \end{enumerate}
\end{prop}

\begin{proof} 
\eqref{item:31} 
Let $\ell\ge 1$ and $m\in \Theta_{\ell \ov D}\cap
  M$. Proposition \ref{prop:3}\eqref{item:13} implies that $s_{m}=(\chi^{m},\ell D)\in
  \wh \Gamma(X,\ell \ov D)^{\times}$. Since
  $\mult_{\nu}(\div(s_{m})) = \ell \mult_{\nu}(\div(s_{a}))$ for
  $a=m/\ell$, it follows that
  \begin{displaymath}
 \mu_{\ov D}(\nu)\le
 \inf\{ \mult_{\nu}(
 \div(s_{a})) \mid a\in \Theta_{\ov D}\cap M_{\Q}\}.      
   \end{displaymath}

   For the reverse inequality, let $s\in \wh \Gamma(X,\ell \ov
   D)^{\times}$. Write $s=(f,\ell D)$ with $f=\sum_{m\in
     \ell\Delta_{D}} c_{m}\chi^{m}$. By Corollary \ref{cor:1}, $
   c_{m}\chi^{m} \in \wh \Gamma(X,\ell \ov D)^{\times}$ for all $m$
   such that $c_{m}\ne 0$. For all such $m$, Proposition \ref{prop:3}\eqref{item:13}
   together with the product formula imply that $m\in \ell\Theta_{\ov
     D}$. Hence,
  \begin{displaymath}
    \mult_{\nu}(\div(s)) = \ell \mult_{\nu}(D)+ \nu\bigg(\sum_{m\in \ell\Delta_{D}} c_{m}\chi^{m}\bigg)
\ge \ell \mult_{\nu}(D) +
\min_{m\in \ell\Theta_{\ov
  D}\cap M} \nu(\chi^{m}). 
  \end{displaymath}
This implies that
\begin{math}
  \mu_{\ov D}(\nu)
  \ge 
  \inf\{ \mult_{\nu}(\div(s_{a})) \mid a\in \Theta_{\ov D}\cap M_{\Q}\},      
   \end{math}
which proves the formula. 

\eqref{item:10} Write $\Omega =\{a\in M_{\R}\mid \mu_{\ov D}(\nu)\le
\mult_{\nu}(\div(s_{a})) \ \forall \nu\in \DV(X) \}$ for
short. By~\eqref{item:31}, $\Theta_{\ov D}\cap M_{\Q}\subset
\Omega$. Since $\ov D$ is big, $\Theta_{\ov D}\cap M_{\Q}$ is dense in
$\Theta_{\ov D}$ and so  $\Theta_{\ov D}\subset \Omega$.

    For the reverse inclusion, let $b\in M_{\R}\setminus \Theta _{\ov
      D}$. Since $\Theta _{\ov D}$ is convex and closed, there is a
    real number $\varepsilon >0$ and a
    primitive element $u\in N$ such that $\langle b,u \rangle<\langle
    x,u \rangle-\varepsilon $ for all $x\in \Theta _{\ov
      D}$. Let $\Sigma '$ be a complete unimodular regular refinement
    of $\Sigma $ containing the ray $\R_{\ge 0}u$. Let $H$ be the
    prime Weil divisor of $X_{\Sigma '}$ corresponding to this
    ray and $\nu_{H}$ the associated divisorial valuation of $X$. Then, for any
    $a=m/\ell\in \Theta _{\ov
      D}\cap M_{\Q}$,
    \begin{displaymath}
      \nu_{H}(\chi^{b})=\langle b,u \rangle < \langle a,u
      \rangle-\varepsilon = \nu_{H}(\chi^{m})/\ell-\varepsilon .
    \end{displaymath}
    By \eqref{item:31}, 
    \begin{math}
      \mult_{\nu_{H}}(\div(s_{b}))<\mu _{\ov D}(\nu_{H})
    \end{math}
    and so $b\not \in \Omega $, which proves the statement.

\eqref{item:27}
  Let $\{H_{1},\dots,H_{l}\}$ be the set of prime Weil divisors on $Y$ 
  containing all the components of $E$ and  of $\sigma
  ^{-1}(X\setminus X_{\Sigma ,0})$. For short, write 
  \begin{displaymath}
\Omega '=\{a\in M_{\R}\mid
  \mu_{\ov D}(\nu_{i})\le
    \mult_{\nu_{i}}(\div(s_{a})),\ i=1,\dots,l \}.    
  \end{displaymath}
By \eqref{item:10},
  \begin{math}
    \Theta_{\ov D}\subset \Omega '.         
  \end{math}
  For the reverse inclusion, let $b\in \Omega '$. Then, for all $\nu\in \DV(X)$,
  \begin{multline*}
    \mu _{\ov
      D}(\nu)=\mult_{\nu}(E)=\sum_{i=1}^{l}\mult_{\nu_{i}}(E)\mult_{\nu}(H_{i})=
    \sum_{i=1}^{l}\mu_{\ov D}(\nu_{i})\mult_{\nu}(H_{i})\\
    \le \sum_{i=1}^{l}\mult_{\nu_{i}}(\div(s_{b}))\mult_{\nu}(H_{i})=
    \mult_{\nu}(\div(s_{b})).
  \end{multline*}
  By  \eqref{item:10}, this implies that $b\in \Theta _{\ov D}$ and so
  we obtain the first statement. 
The quasi-rationality of
  $\Theta _{\ov D}$ follows from the fact that
  \begin{displaymath}
    \mu _{\ov D}(\nu_{i})= \mult_{\nu_{i}}(\div s_{a})
  \end{displaymath}
  is an affine equation in $a$ with integral slope.
\end{proof}

The relationship between the arithmetic multiplicity and the Zariski
decomposition is given by the following result of Moriwaki
\cite[Theorems 2.5 and 4.1.1]{Moriwaki:rlsbc}. 

\begin{thm} \label{thm:9} Let $\sigma\colon \cY\to\cX$ be a birational
  morphism of generically smooth normal projective varieties over
  $\Z$. Denote by $X$ and $Y$ the generic fibre of $\cX$ and $\cY$,
  respectively. Let $\ov {\cD}$ be a big arithmetic
  $\R$-divisor on $\cX$. If $\ov {\cD}$ admits a Zariski decomposition 
  $\sigma^{*}\ov \cD=\ov \cP+\ov \cE$ with $\ov \cP,\ov \cE$ arithmetic
  $\R$-divisors on $\cY$, then the
  arithmetic multiplicity of $\ov \cD$ is
  represented by the $\R$-divisor $E=\cE\mid_Y\in \Car(Y)_{\R}$.
\end{thm}

\begin{proof}[Proof of Theorem \ref{thm:5}]
  By Theorem
  \ref{thm:4}\eqref{item:26}, we know that \eqref{item:2} is equivalent
  to \eqref{item:4}. Theorem \ref{thm:9} and Proposition
  \ref{prop:12}\eqref{item:27} show that 
  \eqref{item:3} implies \eqref{item:4}. 

It only remains to prove that
  \eqref{item:4} implies \eqref{item:3}. The only difficulty is to
  show that the metrized $\R$-divisors appearing in the Zariski
  decomposition $\ov D=\ov P + \ov E$ obtained by Theorem
  \ref{thm:4}\eqref{item:26} come from arithmetic
  $\R$-divisors. Clearly, it is enough to show that this is the case
  for the nef part $\ov P$. 

  Suppose that $\Theta_{\ov D}$ is a quasi-rational polytope.  Let
  $\Sigma'$ be a complete unimodular regular refinement of $\Sigma$
  and $\Sigma_{\Theta_{\ov D}}$. Set $X'=X_{\Sigma'}$ and let
  $\varphi\colon X'\to X$ be the associated birational toric map.
  Recall the construction of $\ov P$ in the proof of
  Theorem \ref{thm:4}\eqref{item:26}:  $P$ is the toric $\R$-divisor on
  $X'$ determined by $\Theta_{\ov D}$ and $\ov P$ is the toric
  semipositive metrized $\R$-divisor associated to the restriction to
  $\Theta_{\ov D}$ of the family of concave functions
  $\{\vartheta_{\ov D,v}\}_{v}$ under the correspondence in
  Proposition \ref{prop:23}\eqref{item:73}.  For $v\ne \infty$, the
  function $\psi_{\ov D,v}$ is piecewise affine, by \cite[Proposition
  4.5.10(1)]{BurgosPhilipponSombra:agtvmmh}, and so is each local roof
  function $\vartheta_{\ov D,v}$.

For each $v\in \mathfrak{M}_{\Q}$, consider the functions
\begin{displaymath}
\zeta_{v}=\vartheta_{v}\big|_{\Theta_{\ov D}}\qquad \text{and} \qquad
\varphi_{v}=\zeta_{v}^{\vee} 
\end{displaymath}
and the finite set of places 
\begin{math}
  S= \{v\in \mathfrak{M}_{\Q}\setminus \{\infty\} \mid \zeta_{v}\not\equiv
  0\}.
\end{math}
For each $v\in S$, let $p_{v}\in
\Z$ be the corresponding prime number. 
Choose a subdivision $\Pi_{v}$ of $N_{\R}$
so that $\varphi_{v}$ is piecewise affine on $\Pi_{v}$ and $\rec (\Pi
_{v})=\Sigma '$. 
This  can be done as follows. Let $\Psi'$ be a
strictly concave function on $\Sigma'$, which exists because of the
condition that $\Sigma'$ is a regular fan. 
Then, $\Pi_{v}$ can be constructed as the subdivision determined by
the concave function $\varphi_{v}+\Psi'$ as in \cite[Definition
2.2.5]{BurgosPhilipponSombra:agtvmmh}.

For each finite subset $S'\subset S$ we denote $p_{S'}=\prod_{v\in
  S'}p_{v}$ and $\Z_{S'}=\Z[1/p_{S'}]$.
Consider the  toric scheme $\cX_{\Pi_{v}}$ over
$\Z_{S\setminus \{v\}}$ obtained by the  construction in \cite[\S
3.5]{BurgosPhilipponSombra:agtvmmh} using $\Z_{S\setminus \{v\}}$ and
$p_{v}$ in place of  $K^{\circ}$ and $\varpi$. The function
$\varphi_{v}$ defines a Cartier
 divisor $\cP_{\varphi_{v}}$ on $\cX_{\Pi_{v}}$ as in the case of
 toric varieties over a field. The restriction of $\cP_{\varphi_{v}}$
 to the generic fibre $X'$ agrees with $P$.

For $v,w\in S$, the restriction of the models $\cX_{\Pi_{v}}$
and $\cX_{\Pi_{w}}$ to  $\Z_{S\setminus \{v,w\}}$ can be identified 
by using the element $p_{v}/p_{w}$. Under this identification, the
Cartier divisors $\cP_{\varphi_{v}}$ and $\cP_{\varphi_{w}}$
correspond to each other. Since these identifications satisfy the
cocycle condition, we can glue together the schemes
$\cX_{\Pi_{v}}$, $v\in S$, 
into a model $\cX'$ over $\Z$ of $X'$, and the divisors
$\cP_{\varphi_{v}}$ into a model $\cP$ of $P$. Then
\begin{displaymath}
  g(x)=-2\varphi_{\infty}(\val_{\infty}(x))
\end{displaymath}
is a Green function of $C^{0}$ and $PSH$-type for $P$. Hence
$\ov{\cP}=(\cP,g)$ is a nef arithmetic $\R$-divisor and, by construction, $\ov P$ is
its associated metrized $\R$-divisor, which concludes the proof of the theorem.
\end{proof}

\begin{exmpl}\label{exm:15}
  Consider again the toric metrized
  divisor $\ov D_{\bfalpha}$ on $\P^{n}_{\Q}$ in Example~\ref{exm:11} and let $\ov \cD_{\bfalpha }$ be the associated toric
  arithmetic divisor. 

  With the notation therein, suppose that $n\ge 2$, and that
  $\sum_{i}\alpha_{i}>1$ or, equivalently, that $\ov D_{\bfalpha}$ is
  big. As noted in Example \ref{exm:14}, the convex set
  $\Theta_{\bfalpha}$ is not a polytope in this case.  Hence, Theorem
  \ref{thm:5} shows that $\ov \cD_{\bfalpha}$ does not admit a Zariski
  decomposition into arithmetic $\R$-divisors.
\end{exmpl}

\begin{rem}
  In principle, Theorem \ref{thm:5} leaves open the possibility that 
  the existence of a general Zariski decomposition of $\ov D$ does not imply
  the existence of a toric one. It would be interesting to settle this
  question and extend Theorem \ref{thm:5} to arbitary metrized
  $\R$-divisors. 
  Since the main reason why we restrict ourselves to
  arithmetic $\R$-divisors is the use of Theorem \ref{thm:9}, one step
  in this direction would be to extend this last result to metrized
  $\R$-divisors. 
\end{rem}

\appendix
\section{Convex analysis of asymptotically conic functions} \label{sec:conv-analys-asympt}

In this appendix we extend some definitions and constructions from
\cite[Chapter 2]{BurgosPhilipponSombra:agtvmmh} to functions which are not necessarily concave. We will freely use the notations and conventions in {\it loc. cit.}.

  \begin{defn} \label{def:4} Let $f\colon N_{\R}\to \R$ be a function.
The \emph{stability set} of $f$ is the subset of $M_{\R}$ given by 
\begin{displaymath}
  \stab(f)=\{x\in M_{\R} \mid x-f \text{ is bounded below}\}.
\end{displaymath}
When $f$ is a concave function, this coincides with the definition of
stability set in convex analysis \cite{Rockafellar:ca}.  If
$\stab(f)\not = \emptyset$, this is a convex set. The function $f$
determines a concave function $f^{\vee}\colon \stab(f)\to \R$ defined
as
\begin{displaymath}
  f^{\vee}(x)=\inf_{u\in N_{\R}}\langle x,u\rangle-f(u).
\end{displaymath}
When $f$ is concave, the function $f^{\vee}$ is the
Legendre-Fenchel dual of $f$. 
\end{defn}

\begin{lem} \label{lemm:4}
Let $f,g\colon
  N_{\R} \to \R$ be two functions. Then
  \begin{enumerate}
  \item \label{item:35}  
\begin{math}
    \stab(f+g)\supset \stab(f)+\stab(g). 
  \end{math}
\item \label{item:74}
$(f+g)^{\vee}(x)\ge
    (f^{\vee}\boxplus g^{\vee})(x)$
for all $x\in \stab(f)+\stab(g)$.
\end{enumerate}
\end{lem}

\begin{proof}
\eqref{item:35} Let $x\in \stab(f )+\stab(g )$ and take
  $y\in \stab(f )$ and $z\in \stab(g )$ such that $x=y+z$. Then
  $y-f $ and $z-g $ are bounded below. Therefore $(y+z)-(f+g)$ 
  is also bounded below, and so $x\in\stab(f +g )$.

\eqref{item:74} For  $x\in \stab(f)+\stab(g)$ we have that
\begin{displaymath}
  (f^{\vee}\boxplus g^{\vee})(x)=
    \sup_{y+z=x}(f^{\vee}(y)+g^{\vee}(z)), \quad
     (f+g)^{\vee}(x)=\inf_{u} (\langle x,u\rangle-f(u)-g(u))
\end{displaymath}
and
  \begin{multline*}
        \sup_{y+z=x}(f^{\vee}(y)+g^{\vee}(z))
    =\sup_{y+z=x} \Big(\inf_{u}(\langle y,u\rangle-f(u))+
    \inf_{v}(\langle z,v\rangle-g(v))\Big)\\
    \le \sup_{y+z=x}\inf_{u}(\langle y,u\rangle+\langle z,u\rangle-f(u)-g(u))
    = \inf_{u} (\langle x,u\rangle-f(u)-g(u)),
   \end{multline*}
from where we deduce the statement. 
\end{proof}

A conic function on $N_{\R}$ is a function
$\Psi\colon N_{\R}\to \R$ such that $\Psi(\gamma u)=\gamma \Psi(u)$
for all $u\in N_{\R}$ and $\gamma\ge0$.

\begin{defn} \label{def:11}
Let  $f\colon N_{\R}\to \R$ be a function. We say that $f$
is \emph{asymptotically conic} if there is a conic function
$\Psi$ on $ N_{\R}$ such that $|f-\Psi|$ is bounded. 
Such a conic function $\Psi$ is necessarily unique. We call it the
\emph{recession function} of $f$ and we denote it by $ \rec(f)$.
\end{defn}

\begin{rem} \label{rem:7}
Let $f\colon N_{\R}\to \R$ be an asymptotically conic function. Then 
\begin{displaymath}
  \rec(f)(u)=\lim_{\gamma\to \infty}\frac{f(\gamma u)}{\gamma}.
\end{displaymath}
Hence, for a concave  asymptotically conic function, the
 notion of recession function coincides with the usual one in convex
 analysis, see for instance \cite[Theorem~8.5]{Rockafellar:ca}.  
\end{rem}

The stability set of an asymptotically conic function $f\colon
N_{\R}\to \R$ agrees with that of its recession function. Hence,
\begin{displaymath}
  \stab(f)=\{x\in M_{\R}\mid \langle x,u\rangle \ge \rec(f)(u) \text{
    for all } u\in N_{\R}\}.
\end{displaymath}

If $f$ is an asymptotically conic function $f\colon N_{\R}\to \R$ with
$\stab(f)\not=\emptyset$, the \emph{concave envelope} of $f$, denoted
$\conc(f)$,  is defined as the smallest concave 
function $h\colon N_{\R}\to \R$ such that $h\ge f$.
We have that $f^{\vee}= \conc(f)^{\vee}$ and $f^{\vee\vee}=\conc(f)$.

\begin{lem} \label{lemm:6}
Let $f,g\colon
  N_{\R} \to \R$ be two asymptotically conic functions such that $0\in \stab(g)$. Then
  \begin{enumerate}
  \item \label{item:36}  $\stab(f+\varepsilon' g)\subset
\stab(f+\varepsilon g)$ for $0\le \varepsilon '\le
\varepsilon$ and
  \begin{displaymath}
   \bigcap_{\varepsilon > 0} \stab(f+\varepsilon g) = \stab(f).
  \end{displaymath}

\item \label{item:15} Assume that $\stab(f)\not=\emptyset$. Then, for $0\le \varepsilon '\le \varepsilon$,
  \begin{displaymath}
 (f+\varepsilon
  g)^{\vee}\big|_{\stab(f+\varepsilon'g)}\ge
  (f+\varepsilon'g)^{\vee} + (\varepsilon
  -\varepsilon ')g^{\vee}(0)
  \end{displaymath}
and, for $x\in \stab(f)$,
\begin{displaymath}
\lim_{\varepsilon \to 0}\big(f+\varepsilon g\big)^{\vee}(x) = f^{\vee}(x).
\end{displaymath}
  \end{enumerate}
 \end{lem}

\begin{proof}
\eqref{item:36} Let $0\le \varepsilon' \le \varepsilon $. By Lemma \ref{lemm:4}\eqref{item:35},
\begin{displaymath}
  \stab (f +\varepsilon g  ) 
  \supset
  \stab (f +\varepsilon'g  ) + \stab ( (\varepsilon-\varepsilon')
  g ) \supset
  \stab (f +\varepsilon'g  ) 
\end{displaymath}
because $\varepsilon -\varepsilon'\ge 0$ and so $0\in
(\varepsilon-\varepsilon')\stab ( g)= \stab (
(\varepsilon-\varepsilon') g )$.  Now let $x\in\stab (f +\varepsilon g
)$ for all $\varepsilon
>0$. We have that $\stab (f +\varepsilon g )=\stab (\rec(f)
+\varepsilon \rec(g))$ and so, for all $u\in N_{\R}$,
\begin{displaymath}
  \langle x,u\rangle \ge \rec(f) (u)+\varepsilon \rec(g) (u).
\end{displaymath}
Letting $\varepsilon \to 0$, we obtain
\begin{math}
  \langle x,u\rangle \ge \rec(f) (u).
\end{math}
Since this holds for all $u\in N_{\R}$, we deduce that $x\in
\stab(\rec(f))=\stab(f )$. Hence, $\bigcap_{\varepsilon > 0} \stab(f+\varepsilon g ) \subset \stab(f )$. The reverse inclusion follows
from the argument above with $\varepsilon'=0$.

\eqref{item:15} Let $0\le \varepsilon'\le \varepsilon$. By Lemma
\ref{lemm:4}\eqref{item:74}, for $x\in \stab(f+\varepsilon' g)$,
\begin{displaymath}
  \big(f+\varepsilon g\big)^{\vee}(x) \ge
  \big(\big(f+\varepsilon' g)^{\vee}\boxplus\big(\big(\varepsilon-\varepsilon'\big) g\big)^{\vee}\big)(x) \ge
(f+\varepsilon' g)^{\vee}(x) + \big(\varepsilon-\varepsilon'\big)g^{\vee}(0), 
\end{displaymath}
which proves the first assertion. 

Letting $\varepsilon'=0$, we deduce that $\liminf_{\varepsilon\to 0}   \big(f+\varepsilon g\big)^{\vee}(x) \ge
f^{\vee}(x)$ for $x\in \stab(f)$.
For the reverse inequality, given $\delta >0$  let $u_{0}\in N_{\R}$ such that
\begin{math}
   f^{\vee}(x) \ge \langle x,u_{0}\rangle
    - f(u_{0})- \delta 
\end{math}.
Then
\begin{displaymath}
  (f+\varepsilon  g )^{\vee}(x) \le \langle x,u_{0}\rangle
    - f(u_{0}) - \varepsilon g (u_{0})\le 
   f^{\vee}(x) - \varepsilon   g (u_{0})+\delta.
\end{displaymath}
Letting $\varepsilon\to 0$, we deduce that 
\begin{displaymath}
\limsup_{\varepsilon\to 0}  (f+\varepsilon g)^{\vee}(x) \le f^{\vee}(x)+ \delta .
\end{displaymath}
Since this holds for
all $\delta >0$, we obtain $\limsup_{\varepsilon\to 0} (f+\varepsilon g
)^{\vee}(x) \le f^{\vee}(x)$. Therefore 
$\lim_{\varepsilon\to 0} (f+\varepsilon g)^{\vee}(x)$ exists and is
equal to $f^{\vee}(x)$, which concludes the proof.
\end{proof}

\newcommand{\noopsort}[1]{} \newcommand{\printfirst}[2]{#1}
  \newcommand{\singleletter}[1]{#1} \newcommand{\switchargs}[2]{#2#1}
  \def\cprime{$'$}
\providecommand{\bysame}{\leavevmode\hbox to3em{\hrulefill}\thinspace}
\providecommand{\MR}{\relax\ifhmode\unskip\space\fi MR }
\providecommand{\MRhref}[2]{%
  \href{http://www.ams.org/mathscinet-getitem?mr=#1}{#2}
}
\providecommand{\href}[2]{#2}


\end{document}